\documentclass{amsart}

\usepackage{amsfonts,amsthm,amssymb,euscript}
\usepackage[pdftex]{graphicx}
\usepackage{graphics}
\usepackage[matrix,arrow]{xy}
\usepackage[active]{srcltx}

\usepackage[colorlinks=true]{hyperref}

\newcommand{\C}{\mathbb{C}}
\newcommand{\R}{\mathbb{R}}
\newcommand{\Z}{\mathbb{Z}}
\newcommand{\Q}{\mathbb{Q}}
\newcommand{\N}{\mathbb{N}}
\newcommand{\D}{\mathbb{D}}
\newcommand{\M}{\mathcal{M}}
\newcommand{\J}{\mathcal{J}}
\newcommand{\B}{\mathcal{B}}

\newcommand{\T}{\mathcal{T}}
\newcommand{\U}{\mathcal U}
\newcommand{\V}{\mathcal V}
\newcommand{\W}{\mathcal W}
\renewcommand{\P}{\mathcal{P}}
\renewcommand{\sl}{{\rm sl}}

\newcommand{\wind}{{\rm wind}}

\newcommand{\util}{{\tilde u}}
\newcommand{\vtil}{{\tilde v}}

\newcommand{\jtil}{{\tilde J}}

\newcommand{\jbar}{{\bar J}}

\theoremstyle{plain}
\newtheorem{theorem}{Theorem}[section]
\newtheorem{proposition}[theorem]{Proposition}
\newtheorem{lemma}[theorem]{Lemma}
\newtheorem{corollary}[theorem]{Corollary}

\theoremstyle{definition}
\newtheorem{definition}[theorem]{Definition}

\theoremstyle{remark}
\newtheorem{remark}[theorem]{Remark}

\title[Elliptic bindings]{Elliptic bindings for dynamically convex Reeb flows on the real projective three-space}
\author{Umberto L. Hryniewicz}
\address[Umberto L. Hryniewicz]{Universidade Federal do Rio de Janeiro -- Departamento de Matem\'atica Aplicada, Av.\ Athos da Silveira Ramos 149, Rio de Janeiro RJ, Brazil 21941-909.}
\email{umberto@labma.ufrj.br}
\author{Pedro A. S. Salom\~ao}
\address[Pedro A. S. Salom\~ao]{Universidade de S\~ao Paulo,  Instituto de Matem\'atica e Estat\'istica -- Departamento de Matem\'atica, Rua do Mat\~ao, 1010 - Cidade Universit\'aria - S\~ao Paulo SP, Brazil 05508-090.}
\email{psalomao@ime.usp.br}


\begin{document}

\begin{abstract}
The first result of this paper is that every contact form on $\R P^3$ sufficiently $C^\infty$-close to a dynamically convex contact form admits an elliptic-parabolic closed Reeb orbit which is $2$-unknotted, has self-linking number $-1/2$ and transverse rotation number in $(1/2,1]$. 
Our second result implies 
that any $p$-unknotted periodic orbit with self-linking number $-1/p$ of a dynamically convex Reeb flow on a lens space of order $p$ is the binding of a rational open book decomposition, whose pages are global surfaces of section. 

As an application we show that in the planar circular restricted three-body problem for energies below the first Lagrange value and large mass ratio, there is a special link consisting of two periodic trajectories for the massless satellite near the smaller primary -- lunar problem -- with the same contact-topological and dynamical properties of the orbits found by Conley in~\cite{conley} for large negative energies. Both periodic trajectories bind rational open book decompositions with disk-like pages which are global surfaces of section. In particular, one of the components is an elliptic-parabolic periodic orbit.
\end{abstract}

\maketitle


\section{Introduction}\label{sec_intr}

\subsection{Celestial Mechanics and Symplectic Dynamics}

It is the purpose of this paper to study Hamiltonian flows restricted to dynamically convex energy levels diffeomorphic to $\R P^3$. We prove the existence of a rational open book decomposition with disk-like pages and an elliptic-parabolic binding orbit, such that each page is a global surface of section for the Hamiltonian flow. The binding orbit and an orbit associated to any fixed point of the first return map form a special link. This special pair of orbits has the same dynamical and contact-topological properties of
\begin{itemize}
\item the retrograde/direct orbits of Poincar\'e near the heavy primary in the planar circular restricted three-body problem (PCR3BP) with large mass ratio,
\item the orbits obtained by Hill~\cite{hill} in the lunar problem for large negative energies, and
\item the orbits bounding the annulus-like global surface of section obtained by Conley~\cite{conley} for the PCR3BP with any mass ratio for large negative energies, see also Kummer~\cite{kummer}.
\end{itemize}
In all these classical situations one of the orbits is elliptic-parabolic.

The reason for the results of Hill~\cite{hill} on the lunar problem, and of Conley~\cite{conley}, to be known only for large negative energies is that they rely on perturbation theory of periodic orbits, where the unperturbed system is integrable. The power of pseudo-holomorphic curves is that they allow for the analysis of systems {\it far from integrable}. This is exploited in the lunar problem by Albers, Fish, Frauenfelder, Hofer and van Koert in~\cite{AFFHvK}, where it is shown, roughly speaking, that for energies below the first critical value and large mass ratio, the dynamics of the satellite near the light primary can be lifted, via Levi-Civita regularization, to the Hamiltonian flow on a convex energy level $S\subset \R^4$. Results of Hofer, Wysocki and Zehnder~\cite{convex} can be applied to obtain a disk-like global surface of section $D$ for the flow lifted to $S$.

However, $S$ possesses antipodal symmetry, and it is the quotient $\Sigma := S/\{\pm {\rm id}\} \simeq \R P^3$ that is in 1-1 correspondence with the energy level of the regularized system.  This left unanswered the question of whether the periodic orbit $P$ obtained as the projection of $\partial D$ to $\Sigma$, and the projection of a periodic orbit associated to a fixed point of the return to $D\setminus\partial D$, retain the same properties as the orbits of Hill and Conley. Namely,
\begin{itemize}
\item[a)] Does $D$ descend to a (rational) disk-like global surface of section on $\Sigma$, in particular, is $P$ a $2$-unknot with self-linking number $-1/2$?
\item[b)] Is there another $2$-unknotted periodic orbit with self-linking number $-1/2$ which is simply linked with $P$?
\item[c)] Is $P$ elliptic-parabolic?
\end{itemize}
If the answers to these questions would be affirmative then we would obtain a pair of orbits with the same properties of the orbits found by Hill and Conley for large negative energies. Theorem~\ref{main1} below provides, after a minimal amount of work, a pair of orbits with all these properties. The affirmative answer to question c) is, of course, crucial if one aims at studying stability  using KAM theory; this will be pursued in future work.

Our methods come from Symplectic Dynamics~\cite{BH}, where one is concerned with the study of global properties of Hamiltonian systems using symplectic methods. Among such methods the theory of pseudo-holomorphic curves, first introduced in Symplectic Geometry by Gromov~\cite{gromov}, plays a central role in recent applications. These techniques were explored by Floer~\cite{Fl1,Fl2,Fl3} to obtain a major breakthrough on the Arnol'd conjecture on the number of fixed points of Hamiltonian diffeomorphisms. In the set-up of the Hamiltonian Arnol'd conjecture one looks for periodic orbits with prescribed period, but a different question is to ask for periodic orbits with prescribed energy. According to Ekeland~\cite{ekeland} the latter is a harder problem, and the sophisticated tools now come from the seminal work of Hofer~\cite{93} on the Weinstein conjecture, where punctured curves in non-compact symplectic cobordisms were first studied. Part of their power come from the fact that their perturbation and compactness properties are (mostly) well understood, so that their survival through deformations of the system can be studied.

Thanks to pioneering work of Hofer, Wysocki and Zehnder~\cite{93,props1,props2,props3}, pseudo-holomorphic curve theory at its current state is well-developed for a special class of energy levels, namely, those possessing a stable Hamiltonian structure. This class includes energy levels of contact-type, i.e., those admitting transverse infinitesimal symplectic dilations of the ambient phase space. Many energy levels of important Hamiltonian systems with two degrees of freedom coming from Physics and Differential Geometry are of contact-type, for example, this is classical and well-known for geodesic flows on surfaces, and for magnetic flows on surfaces with high enough energy. Recent important work of Albers, Frauenfelder, van Koert and Paternain~\cite{AFvKP} demonstrates that this is also the case for the PCR3BP when the energy is below or slightly above the first critical value.

The answer to question a) above relates to the problem of finding antipodal symmetric closed orbits bounding antipodal symmetric systems of global surfaces of section, for a Reeb flow on standard $S^3$ with antipodal symmetry. The existence of symmetric closed Reeb orbits is extremely relevant, and was recently explored by Frauenfelder and Kang~\cite{FK} for another kind of symmetry. Results of Kang~\cite{Kang} can be used to study existence of symmetric orbits.

Finally, there is an application of our results to geodesic flows of Finsler metrics on the $2$-sphere with a suitable pinching condition on the flag curvatures. More precisely, our results imply that if $r$ is the reversibility of the Finsler metric and its flag curvatures $K$ satisfy
\begin{equation}\label{pinching_cond_rad}
\left(\frac{r}{r+1}\right)^2 < K \leq 1
\end{equation}
then the geodesic flow admits a simple closed geodesic which is elliptic-parabolic, so that its lift to the unit tangent bundle is the binding of a rational open book decomposition with disk-like pages. Each page is a global surface of section and a fixed point of the first return map corresponds to a second closed geodesic. Under the pinching condition~\eqref{pinching_cond_rad}, the existence of two geometrically distinct closed geodesics -- one of them being simple -- was already known by Rademacher~\cite{rad}. Our method not only produces two closed orbits as in~\cite{rad} for a broader class of Hamiltonian systems, but also shows that both these orbits globally organize the Hamiltonian flow as the bindings of two rational open book decompositions by disk-like pages which are global surfaces of section.

\subsection{Main results}

In order to state our results we recall basic notions in contact geometry and establish some notation. A contact form $\alpha$ on a $3$-manifold $M$ is a $1$-form such that $d\alpha$ is non-degenerate when restricted to the associated contact structure $\xi=\ker\alpha$. The Reeb vector field induced by $\alpha$ is defined by equations
\begin{equation}\label{defn_Reeb_vector}
\begin{array}{ccc}  i_{X_\alpha}\alpha = 1, & & i_{X_\alpha}d\alpha = 0. \end{array}
\end{equation}
We denote by $\P(\alpha)$ the set of closed $\alpha$-Reeb orbits, defined as the set of equivalence classes of pairs $P=(x,T)$ where $x:\R\to M$ is a periodic trajectory of $X_\alpha$ and $T>0$ is a period of $x$. The set $x(\R)$ is called the geometric image of $P$, and two such pairs are declared equivalent when their geometric images and periods coincide. We call $P=(x,T)\in\P(\alpha)$ contractible if the loop $t\in\R/T\Z\mapsto x(t)$ is contractible, and it is called simply covered, or prime, if $T$ is the minimal positive period of $x$. If $n\geq 1$ then the $n$-th iterate of $P=(x,T)$ is the orbit $P^n=(x,nT)$. When the first Chern class $c_1(\xi)$ vanishes on $\pi_2(M)$, any contractible $P\in\P(\alpha)$ has well-defined invariants of the transverse linearized dynamics: the Conley-Zehnder index $\mu_{CZ}(P)\in\Z$ and the transverse rotation number $\rho(P)\in\R$.
See Section~\ref{sec_prelim} for a more detailed discussion.

\begin{definition}[Hofer, Wysocki and Zehnder~\cite{convex}]\label{defn_dyn_convex}
A contact form $\alpha$ on a three-manifold $M$ is dynamically convex if $c_1(\ker\alpha)$ vanishes on $\pi_2(M)$ and every contractible $P\in\P(\alpha)$ satisfies $\mu_{CZ}(P)\geq 3$.
\end{definition}

The contact structure $\xi$ is said to be overtwisted if there exists an embedded disk $D\hookrightarrow M$ such that $T\partial D\subset \xi$ and $T_pD\neq \xi_p$ for all $p\in\partial D$. Such a disk $D$ is called an overtwisted disk. When such disks do not exist $\xi$ is called tight. The dichotomy {\it tight versus overtwisted} introduced by Eliashberg plays a central role in three-dimensional contact topology. The following statement exhibits strong contact-topological restrictions imposed by dynamical convexity.

\begin{theorem}[Hofer, Wysocki and Zehnder~\cite{char2}]
If $\alpha$ is a dynamically convex contact form on the closed $3$-manifold $M$ then $\pi_2(M)$ vanishes and $\xi=\ker\alpha$ is a tight contact structure.
\end{theorem}

A knot $K\subset M$ is called $k$-unknotted, $k\geq1$, if there exists an immersion $u:\D\to M$ such that $u|_{{\rm int}(\D)}$ is an embedding and $u|_{\partial \D}:\partial\D \to K$ is a $k$-covering map. The map $u$ is called a $k$-disk for~$K$. If $K$ is oriented then we call $u$ oriented if $u|_{\partial \D}$ is orientation preserving when we endow $\partial\D$ with its counter-clockwise orientation. Simple topological arguments show that if $k\neq k'$ then a knot $K$ cannot be simultaneously $k$-unknotted and $k'$-unknotted. In fact, if $K$ is $k$-unknotted then its order in the fundamental group is precisely $k$. If $K$ is transverse to $\xi$ then we always orient it by $\alpha$. The rational self-linking number $\sl(K,u) \in \Q$ is well-defined with respect to a $k$-disk $u$ for $K$ as in~\cite{BK}, and it is independent of $u$ when the first Chern class $c_1(\xi)$ vanishes on $\pi_2(M)$, in which case we simply write $\sl(K)$. The rational self-linking number is invariant under transverse isotopies of transverse knots. The reader is referred to Section~\ref{sec_prelim} for more details.

One way of finding manifolds with contact forms is to search for a Liouville vector field, which is a vector field $Y$ defined on a symplectic manifold $(W,\omega)$ satisfying $\mathcal L_Y\omega=\omega$. If $M$ is a hypersurface sitting inside $(W,\omega)$ then $M$ has contact-type if it is transverse to some Liouville vector field $Y$ defined near $M$. In this case $\theta := i_Y\omega$ is a primitive of $\omega$ and $\alpha := \iota^*\theta$ is a contact form, where $\iota:M\hookrightarrow W$ is the inclusion map. Every manifold equipped with a contact form arises in this way. As a simple example, consider the standard Liouville form on $\C^2$
\begin{equation}\label{std_liouville_form_S3}
\lambda_{\rm std} := \frac{1}{4i} (\bar z_0 dz_0 - z_0 d\bar z_0 + \bar z_1 dz_1 - z_1 d\bar z_1)
\end{equation}
where $(z_0,z_1)$ are complex coordinates. Then $d\lambda_{\rm std}$ is the standard symplectic structure on $\C^2$. The radial vector field $Y_0(z_0,z_1) = (z_0,z_1)/2$ is Liouville because it satisfies $i_{Y_0}d\lambda_{\rm std}=\lambda_{\rm std}$ and, consequently, $\lambda_0 := \iota^*\lambda_{\rm std}$ is a contact form on $S^3 = \{(z_0,z_1) \in \C^2 : |z_0|^2 + |z_1|^2=1\}$, where $\iota:S^3\hookrightarrow\C^2$ is the inclusion map.

Given integers $p\geq q\geq 1$ satisfying $\gcd (p,q)=1$, there is a free action of $\Z_p=\Z/p\Z$ on $S^3$ generated by the diffeomorphism $g_{p,q}(z_0,z_1) := (e^{ 2\pi i/p} z_0, e^ {2 \pi i q/p} z_1)$. The lens space $L(p,q)$ is the space of trajectories $$ L(p,q):= S^3 / \Z_{p}. $$ With this convention $L(1,1) = S^3$, and $L(2,1) \simeq SO(3) \simeq \R P^3$ is diffeomorphic to the unit tangent bundle of any Finsler metric on $S^2$. We denote by $$ \pi_{p,q}:S^3 \to L(p,q) $$ the quotient projection. It is well-known that $$ \begin{array}{ccc} \pi_1 (L(p,q)) \simeq \Z_p & \text{and} & \pi_{2}(L(p,q)) \simeq 0, \end{array} $$ where the choice of base point is omitted. Since $g_{p,q}^*\lambda_0 = \lambda_0$, we get that $\lambda_0$ descends to a contact form on $L(p,q)$, still denoted by $\lambda_0$. The contact structure $$ \xi_0 = \ker \lambda_0 \subset  TL(p,q) $$ is universally tight: its lift to the universal covering is tight. This property determines $\xi_0$ up to a diffeomorphism, i.e., if $\xi=\ker \lambda$ is a contact structure on $L(p,q)$ and $\pi_{p,q}^*\xi$ is tight then there exists a diffeomorphism $h:L(p,q) \to L(p,q)$ satisfying $h_*\xi = \xi_0$. In this case, the contact form $\lambda$ is also called universally tight.

A closed Reeb orbit is said to be elliptic-parabolic if its transverse Floquet multipliers lie in the unit circle, in three-dimensions this means that $P$ is not hyperbolic.

The main result of this paper reads as follows.

\begin{theorem}\label{main1}
Let $\lambda$ be any dynamically convex contact form on $\R P^3=L(2,1)$. There exists $T>0$ and a $C^\infty$-neighborhood $\U$ of $\lambda$ with the following property: every contact form in $\U$ admits a $2$-unknotted elliptic-parabolic periodic Reeb trajectory with self-linking number $-1/2$ and prime period smaller than $T$. Its transverse rotation number associated to the prime period belongs to $\left(1/2,1\right]$.
\end{theorem}

The proof is based on arguments of~\cite{convex}. However, since our manifold is not the $3$-sphere, this strategy presents some delicate difficulties involving certain holomorphic curves which are non-regular in the sense of Fredholm theory and that cannot be avoided.
The somewhat special case of $L(2,1)$ can be handled by taming the behavior of possibly non-regular curves in a certain non-cylindrical symplectic cobordism. The proof heavily relies on the non-trivial intersection theory developed by Siefring~\cite{sie1,sie2} and on dynamical characterizations of the tight $3$-sphere due to Hofer, Wysocki and Zehnder~\cite{char1,char2} further generalized in~\cite{HLS}.

The second main result completes the study of the existence of rational disk-like global surfaces of section on universally tight lens spaces for Reeb flows initiated in~\cite{convex,char1,char2} followed by~\cite{HS,hryn,openbook,HLS}. In~\cite{HS} this question was completely answered for non-degenerate tight Reeb flows on the three-sphere.

\begin{theorem}[\cite{HS}]\label{teo_HS}
Let $\lambda$ be a non-degenerate contact form defining the standard contact structure on the $3$-sphere, and let $P\in\P(\lambda)$ be a prime closed Reeb orbit. Then $P$ is the boundary of a disk-like global surface of section for the $\lambda$-Reeb flow if, and only if, the following two conditions hold:
\begin{itemize}
\item[i)] $P$ is unknotted, $\sl(P)=-1$ and $\mu_{CZ}(P)\geq 3$.
\item[ii)] Every $P'\in\P(\lambda)$ satisfying $\mu_{CZ}(P')=2$ is linked with $P$.
\end{itemize}
\end{theorem}

In~\cite{PS}, one shows that condition (ii) in Theorem~\ref{teo_HS} is sharp in the following sense: fixing any integer $N\geq 3$, there exist Reeb flows on the tight $3$-sphere admitting precisely one closed Reeb orbit $P'$ with $\mu_{CZ}(P')=2$ and if $P$ is an unknotted closed Reeb orbit with  $\mu_{CZ}(P)\in\{3,\ldots,N\}$, then  $P'$ is not linked to $P$. In particular, $P$ is not the boundary of a global surface of section. In such cases one obtains different types of surfaces of sections, the so called systems of transversal sections introduced in \cite{fols}, where  $P'$ is one the binding orbits. See also the recent paper by Fish and Siefring~\cite{FS}.

Theorem~\ref{thm_22} below will show, as a special case, that  conditions (i) and (ii) in Theorem \ref{teo_HS} are sufficient even when the contact form is degenerate. It should be noted, however, that in very degenerate situations it is not clear whether the above conditions are necessary. The statement of Theorem~\ref{thm_22} requires a preliminary technical definition.

\begin{definition}[\cite{HLS}, Definition~6.3]\label{def_special_robust}
Let $\alpha$ be a contact form on the $3$-manifold $M$, and let $K\subset M$ be a $p$-unknot tangent to $X_\alpha$ and oriented by $\alpha$. An oriented $p$-disk $u_0:\D\to M$ for $K$ is special robust for $(\alpha,K)$ if the following holds:
\begin{itemize}
\item[a)] The singular characteristic distribution of $u_0(\D\setminus\partial\D)$ induced by $\ker\alpha$ has precisely one singular point $e$, which is positive. Moreover, it is possible to find coordinates $(x,y,z)$ on a neighborhood $V$ of $e$ such that $e \simeq (0,0,0)$, $\alpha \simeq dz+xdy$ and $$ u_0(\D) \cap V \simeq \{z = -\frac{1}{2}xy\}. $$
\item[b)] $\exists \epsilon>0$ such that for every sequence of smooth functions $h_k:M \to (0,+\infty)$ satisfying $h_k \to 1 \ \text{in} \ C^\infty$, $h_k|_{K}\equiv1$ and $dh_k|_{K} \equiv 0$ $\forall k$, one finds $k_0\geq 1$ such that
\begin{equation}\label{special_robust_cond_b}
1-\epsilon < |z| < 1, \ k\geq k_0 \Rightarrow X_{h_k\alpha}|_{u_0(z)} \not\in d{u_0}|_z(T_z\D).
\end{equation}
\end{itemize}
\end{definition}

\begin{proposition}[\cite{HLS}, Proposition~6.4]\label{prop_nice_disk}
Let $M,\alpha,K$ be as in Definition~\ref{def_special_robust}, and assume in addition that $\xi=\ker\alpha$ is tight. Let $u$ be an oriented $p$-disk for $K$ and denote by $\rho(u)$ the transverse rotation number of the closed Reeb orbit obtained by iterating $K$ $p$-times, computed with respect to a trivialization of $u^*\xi$. If $\sl(K,u)=-1/p$ and $\rho(u)\neq 1$ then there exists an oriented $p$-disk $u_0$ which is special robust for $(\alpha,K)$. Moreover, $u_0$ can be constructed arbitrarily $C^0$-close to $u$.
\end{proposition}

Let $K\subset L(p,q)$ be a Hopf fiber. A   \textit{rational open book decomposition with binding $K$ and disk-like pages} is a smooth fibration $\pi : L(p,q)\setminus K \to S^1$ so that the closure of each fiber $\pi^{-1}(t)$ is the image of an oriented $p$-disk for $K$, with an additional normal form near the binding \cite[Definition 4.4.4]{gei}.

\begin{theorem}\label{thm_22}
Consider a contact form $\lambda$ on $L(p,q)$ defining its standard contact structure $\xi_{\rm std}$, and assume that $X_\lambda$ admits a periodic trajectory $x:\R\to L(p,q)$, with minimal period $T_{\rm min}>0$, which is an order $p$ rational unknot with self-linking number $-1/p$. Set $T=pT_{\rm min}$, $P = (x,T)$ and assume that $\rho(P)>1$. Consider $\P^*(\lambda) \subset \P(\lambda)$ the set of closed $\lambda$-Reeb orbits in $L(p,q)\setminus x(\R)$ which are contractible in $L(p,q)$ and have transverse rotation number equal to $1$. Fix a special robust $p$-disk $u_0$ for $x(\R)$, oriented by $d\lambda$, and suppose that every $P'\in\P^*(\lambda)$ satisfies
\begin{equation}\label{linked_conv}
\begin{array}{ccc}
P' \ \text{is not contractible in} \ L(p,q)\setminus x(\R) & \text{or} & \int_{P'}\lambda> 1+\int_\D |u_0^*d\lambda|.
\end{array}
\end{equation}
Then $x(\R)$ bounds a rational disk-like global surface of section for the Reeb flow of~$\lambda$ which is a page of a rational open book decomposition of $L(p,q)$ with binding $x(\R)$. Moreover, all pages of this open book are disk-like global surfaces of section.
\end{theorem}

\begin{corollary}\label{coro_main_2}
If $\lambda$ is any dynamically convex contact form on $L(p,q)$ then every $p$-unknot $K$ which is tangent to the Reeb vector field of $\lambda$ and has rational self-linking number $-1/p$ must bound a $p$-disk which is a global surface of section for the Reeb flow. Moreover, this $p$-disk is a page of a rational open book decomposition of $L(p,q)$ with binding $K$ such that all pages are disk-like global surfaces of section.
\end{corollary}

Consider the link $l_0 = \pi_{p,q}(S^1\times 0 \cup 0\times S^1) \subset L(p,q)$. This link is transverse to $\xi_{\rm std}$, and we shall call a Hopf link any transverse link in $(L(p,q),\xi_{\rm std})$ which is transversely isotopic to $l_0$. Combining Theorem~\ref{main1} with Corollary~\ref{coro_main_2} and the results from~\cite{HLS} we obtain the following statement.

\begin{corollary}\label{coro_total}
Let $\lambda$ be any dynamically convex contact form on $\R P^3=L(2,1)$. Then $\ker\lambda$ is contactomorphic to $\xi_{\rm std}$ and there exists a pair $P,P'$ of $\lambda$-Reeb orbits forming a Hopf link and binding rational open book decompositions with disk-like pages. Each page of both open books is a global surface of section for the $\lambda$-Reeb flow. These orbits are $2$-unknotted, have self-linking number $-1/2$, and one of them is elliptic-parabolic with transverse rotation number in $(-1/2,1]$.
\end{corollary}

It is a hard problem to decide whether any universally tight contact form on $L(p,q)$ admits a Hopf link consisting of a pair of closed Reeb orbits. In the dynamically convex case this is answered for $p=1$ by~\cite{convex} and for $p=2$ by the above corollary. The existence of Hopf links in lens spaces is of dynamical relevance since one can use the results of~\cite{HMS} to obtain infinitely many closed Reeb orbits from a very simple non-resonance relation. This non-resonance relation plays the role of the boundary twist-condition in the classical Poincar\'e-Birkhoff theorem.

It is an open conjecture that any compact regular energy level in $\R^{2n}$ which is strictly convex carries an elliptic-parabolic periodic orbit. This conjecture has been established in special cases by Ekeland~\cite{ekeland_elliptic}. Assuming antipodal symmetry this has been established by Dell'Antonio, D'Onofrio and Ekeland~\cite{AOE}. There are results in this direction by Hu, Long and Wang~\cite{HLW}, Long and Zhu~\cite{LZ} and others. See also the recent work by Abreu and Macarini~\cite{AM} using contact homology to study the existence of elliptic orbits for dynamically convex Reeb flows. Work by Benedetti~\cite{ben} explores consequences of dynamical convexity for certain energy levels of magnetic flows on the two-sphere. Dynamical convexity has also been explored in the context of Arnol'd diffusion by Gidea~\cite{gidea}, where  the existence of diffusing orbits with prescribed qualitative behavior described in terms of invariant tori and Aubry-Mather sets is established. These dynamical elements live within a disk-like global section in a normally hyperbolic invariant dynamically convex $3$-sphere inside the energy level.

\subsection{Application to Celestial Mechanics}

Here we apply our results to the PCR3BP for energies below the first Lagrange critical value, as explained in the beginning of this introduction.

Consider two particles, called primaries, moving on $\R^2$ with masses $m_1,m_2$ according to Newton's gravitational law, and a third massless particle, called satellite, moving in the same plane under the influence of the primaries. Assume that the primaries move in circular trajectories about their center of mass ${\bf c}$ in the counter-clockwise direction. Then one can always find a rotating system based on ${\bf c}$ so that the primaries stay at rest. After some suitable normalizations the motion of the satellite is described by the time-independent Hamiltonian
\begin{equation}\label{eq_hamilton1}
H_\mu(q,p)= \frac{|p|^2}{2} +q_1p_2-p_1q_2 - \mu p_2 - \frac{1-\mu}{|q|} -\frac{\mu}{|q-1|}.
\end{equation}
Here $q=q_1 + i q_2 \in \C \setminus \{0,1\}$ and $p=p_1+ip_2\in \C$ play the role of the position and the momentum of $m_3$; the mass ratio defined as $0<\mu:=\frac{m_1}{m_1+m_2}<1$ is assumed to be close to $1$ ($m_1 \gg m_2$). The points $0\in \C$ and $1\in \C$  are the normalized rest positions of the lighter primary $m_2$ and the heavier primary $m_1$, respectively.

One can directly verify that the Hamiltonian \eqref{eq_hamilton1} has five critical points $L_1,\ldots, L_5$ (all of them depending on $\mu$) which are monotonically ordered by their critical values. Moreover, $H_\mu(L_1) \to -\frac{3}{2}$ as $\mu \to 1^-$ and, for all $-c <H_\mu(L_1)$, the energy level $H_\mu^{-1}(-c)$ contains three components $\mathcal{C}_{\mu,c}^0,\mathcal{C}_{\mu,c}^1$ and $\mathcal{C}_{\mu,c}^2$. The projections of $\mathcal{C}_{\mu,c}^0$ and $\mathcal{C}_{\mu,c}^1$ to the $q$-plane are non-overlapping bounded regions whose closures contain $0$ and $1\in \C$, respectively. The projection of the third component $\mathcal{C}_{\mu,c}^2$ into the $q$-plane is unbounded. In the following we focus on the component $\mathcal{C}_{\mu,c}^0$ around the lighter primary at $0\in \C$.

The component $\mathcal{C}_{\mu,c}^0\subset H_\mu^{-1}(-c)$ is non-compact and contains those orbits colliding with  $0\in \C$. It is always possible to regularize $\mathcal{C}_{\mu,c}^0$ so that these collision trajectories become true trajectories. To do that, one may consider Levi-Civita coordinates $q=2v^2$ and $p=\frac{u}{\bar v}$, where $v=v_1+iv_2$ and $u=u_1 + iu_2$. The coordinates $(v,u)$ are symplectic with respect to $(q,p)$ up to a constant factor. Since $(v,u)$ and $-(v,u)$ correspond to the same point $(q,p)$, these new coordinates  admit a natural $\Z_2$-symmetry given by the antipodal map.

Following~\cite{AFFHvK}, the new Hamiltonian is written as $$ \begin{aligned} K_{\mu,c}(v,u)& := |v|^2(H_\mu(q,p)+c) = \frac{|u|^2}{2} + 2|v|^2(v_1u_2-v_2u_1) \\ & -\mu(u_1v_2+u_2v_1) -\frac{1-\mu}{2}-\mu\frac{|v|^2}{|2v^2-1|}+c|v|^2, \end{aligned} $$ where both $\mu$ and $c$ are now seen as parameters.

In $(v,u)$ coordinates, the component $\mathcal{C}_{\mu,c}^0$ corresponds to a smooth compact component $\widetilde \Sigma_{\mu,c}\subset K_{\mu,c}^{-1}(0)$, which is diffeomorphic to $S^3$ and is invariant under the $\Z_2$-symmetry. The quotient $\Sigma_{\mu,c}:= \widetilde \Sigma_{\mu,c} / \Z_2$ is diffeomorphic to  $\R P^3 = L(2,1)$.

\begin{definition}[\cite{AFFHvK}]
The convexity range is the set of pairs $(\mu,c)$ for which $\widetilde \Sigma_{\mu,c}\subset \R^4$ is a strictly convex hypersurface.
\end{definition}

\begin{theorem}[Albers, Fish, Frauenfelder, Hofer and van Koert~\cite{AFFHvK}]\label{teo_AFFHK}
For all $c>\frac{3}{2}$, there exists $\mu_0(c) \in (0,1)$ such that $(\mu,c)$ lies on the convexity range for all $\mu \in (\mu_0(c),1)$.
\end{theorem}

Observe that if $(\mu,c)$ is in the convexity range, then the Liouville form $\lambda_0$ on $\R^4$ restricts to a contact form $\tilde \lambda_{\mu,c}$ on $\widetilde \Sigma_{\mu,c}$. Due to the invariance of $\tilde \lambda_{\mu,c}$ under the antipodal map, it descends to a contact form $\lambda_{\mu,c}$ on $\Sigma_{\mu,c}$.  Theorem 3.4 in \cite{convex} implies that   $\tilde \lambda_{\mu,c}$ and $\lambda_{\mu,c}$, whose Reeb flows reparametrize the Hamiltonian flows on $\widetilde \Sigma_{\mu,c}$ and $\Sigma_{\mu,c}$, respectively, are dynamically convex. Hofer, Wysocki and Zehnder  \cite{convex} show the existence of an open book decomposition with disk-like pages adapted to $\tilde \lambda_{\mu,c}$ on $\widetilde \Sigma_{\mu,c}$ so that each page is a global surface of section. However, it is not clear whether such an open book descends to a rational open book decomposition adapted to $\lambda_{\mu,c}$ on $\Sigma_{\mu,c}$, or what precise dynamical and contact-topological properties the projection of the binding orbit has. In fact, there is no evidence for the binding orbit of this open book to be symmetric under the antipodal map. The following result is a direct consequence of a combination of Theorem~\ref{teo_AFFHK}, Theorem~\ref{main1}, Theorem~\ref{thm_22} and~\cite[Theorem~3.4]{convex}.

\begin{theorem}
For all $(\mu,c)$ in the convexity range the following assertions hold:
\begin{itemize}
\item[(i)] $\Sigma_{\mu,c}$ admits an elliptic-parabolic $2$-unknotted periodic orbit $P$ which is the binding of a rational open book decomposition with disk-like pages. Each page of this open book is a rational global surface of section for the Hamiltonian flow on $\Sigma_{\mu,c}$. Moreover, the rotation number of $P$ lies in $(\frac{1}{2},1]$ and its self-linking number is $\frac{-1}{2}$.
\item[(ii)] The first return map associated to any page of the open book in (i) admits at least one fixed point, and any such fixed point corresponds to a $2$-unknotted periodic orbit $P'\subset \Sigma_{\mu,c}$ with self-linking number $\frac{-1}{2}$.
\item[(iii)] Given any periodic orbit $P'' \subset \Sigma_{\mu,c}$ which is  $2$-unknotted with self-linking number $\frac{-1}{2}$, there exists a rational open book decomposition with disk-like pages with binding $P''$ such that each page is a rational global surface of section for the Hamiltonian flow on $\Sigma_{\mu,c}$.
\end{itemize}
\end{theorem}

The relevance of the above result is that the orbits $P,P'$ obtained in (i)-(ii) form a Hopf link and have the same contact-topological properties of the orbits obtained by Hill~\cite{hill} in the lunar problem and by Conley~\cite{conley}, both for large negative energies.

\subsubsection{Birkhoff's retrograde orbit}
In \cite{B}, Birkhoff proves that for all mass ratios $0<\mu <1$ and all energies $c$ below the first Lagrange value, there exists a periodic orbit $P_{\mu,c}$ inside the bounded component near each primary. Such a periodic orbit, called retrograde by Birkhoff, projects onto the $q$-plane as a simple closed curve surrounding the primary in the clockwise direction.  Hence $P_{\mu,c}$, seen as a periodic orbit on $(\R P^3,\xi_0)$, is $2$-unknotted and has self-linking number $\frac{-1}{2}$. Birkhoff raises the question whether this retrograde orbit is the boundary of a disk-like global surface of section. This global surface of section may be used to find a direct orbit near the light primary. The direct orbit near the heavy primary was previously found by Poincar\'e for $\mu$ close to $1$ and a dense set of  energies. It is worth mentioning that the motion of the moon in our Sun-Earth-Moon system is direct.

Theorem 1.12 answers Birkhoff's question affirmatively on the bounded component near the light primary and $(\mu,c)$ in the convexity range. Hence Birkhoff's retrograde orbit $P_{\mu,c}$ is the binding of a rational open book decomposition with disk-like pages and each page is a rational  global surface of section. In particular, a periodic orbit $P'_{\mu,c}$, corresponding to a fixed point of the first return map, is again the binding of a rational open book decomposition with similar dynamical properties.

Finally we remark that in his thesis  \cite{McG} McGehee finds disk-like global surfaces of section for the lifted flow of the bounded component near the heavy primary, $\mu \sim 1$ and energies below the first Lagrange value.

\subsection{Application to Finsler geodesic flows}

Let $F:TS^2 \to [0, +\infty)$ be a Finsler metric on the $2$-sphere $S^2$ and denote by $T^1S^2 = F^{-1}(1)$ its unit tangent bundle.  The reversibility $r\in [1,+\infty)$ of $F$, defined by $$ r :=  \sup \left\{\frac{F(v)}{F(-v)}:v\in TS^2\setminus \{0\} \right\}, $$  is a measure of how the fibers of $T^1S^2$ deviate from being symmetric relative to the zero vector. It attains its minimum value $r=1$ if and only if  all of them are symmetric. In this case, the metric is called reversible, otherwise it is called irreversible.

Existence and multiplicity of closed geodesics for Finsler metrics on $S^2$ is a very delicate question. While in the Riemannian case one always has infinitely many closed geodesics, in the Finsler case one may have finitely many as in Katok's examples \cite{kat} with only two of them. In these examples the reversibility plays a crucial role; all Katok's examples are irreversible. A sharp lower bound on the number of closed geodesics in the Finsler case has only been proved recently.

\begin{theorem}[Bangert, Long \cite{BL}]\label{theo_bl} A Finsler metric on $S^2$ admits at least two prime closed geodesics. \end{theorem}

One might ask if at least one of the closed geodesics in Theorem \ref{theo_bl} is simple. This can be  affirmatively answered if the flag curvature of $F$ is positive and pinched by a suitable constant depending on the reversibility.

\begin{theorem}[Rademacher \cite{rad}]\label{theo_rad1}
Let $F$ be a Finsler metric with reversibility $r$ and flag curvature $K$. Assume that
\begin{equation*}
\left(\frac{r}{r+1}\right)^2 < \delta \leq K \leq 1,
\end{equation*}
for some $\delta$. Then its geodesic flow admits two distinct closed geodesics $\gamma_1$ and $\gamma_2$ with lengths $l(\gamma_1) \leq l(\gamma_2)$, such that $\gamma_1$ is simple and $l(\gamma_1) \leq \frac{2\pi}{\sqrt{\delta}}$.
\end{theorem}

The geodesic flow restricted to $T^1S^2 \simeq L(2,1)$ is the Reeb flow of the universally tight contact form $\lambda_F$ obtained by pulling-back the tautological $1$-form on $T^*S^2$ to $TS^2$ under the Legendre transform induced by $F$, and restricting it to $T^1S^2$. In this case, the Reeb vector field $X_F$ of $\lambda_F$ coincides with the geodesic vector field restricted to $T^1S^2$.

A regular simple closed curve $S^1 \ni t \mapsto \gamma(t) \in S^2$ represents the generator $t \mapsto \dot \gamma(t)/F(\dot \gamma(t))$ of $\pi_1(T^1S^2) \simeq \Z_2$, which is a $2$-unknot. We may assume that $F(\dot \gamma)=1$. This means that there exists an immersion $u:\D \to T^1S^2$ so that $u|_{{\rm int}(\D)}$ is an embedding and $u|_{\partial \D}$ is a $2$-covering map over the knot $\dot \gamma(\R)\subset T^1S^2$. Moreover,  $\dot \gamma$ is transverse to the contact structure $\xi_F=\ker \lambda_F$ and its rational self-linking number  is $\frac{-1}{2}$. The disk $u$ is called a $2$-disk for the $2$-unknot $\dot \gamma$. Any knot on $T^1S^2$ which is transverse to $\xi_F$ and transversally isotopic to $\dot \gamma$ is called a Hopf fiber. A pair of geometrically distinct simple closed geodesics $\gamma_1$ and $\gamma_2$ which intersect at precisely two distinct points in $S^2$, lift to a pair of Hopf fibers $\dot \gamma_1$ and $\dot \gamma_2$ in $T^1S^2$ forming a link $l:=\dot \gamma_1 \cup \dot \gamma_2\subset T^1S^2$. Any link on $T^1S^2$ which is transverse to $\xi_F$  and transversally isotopic to $l$ is called a Hopf link.



Our main result follows from Theorem~\ref{main1} and Corollary~\ref{coro_main_2}. It is a new proof and an improvement of Theorem~\ref{theo_rad1}.

\begin{theorem}\label{theo_finsler}Let $F$ be a Finsler metric with reversibility $r$ and flag curvatures $K$ satisfying
\begin{equation}\label{eq_flag1}
\left(\frac{r}{r+1}\right)^2 <\delta \leq K \leq 1,
\end{equation}
for some $\delta \in (0,1]$. Then its geodesic flow admits two distinct closed geodesics $\gamma_1$ and $\gamma_2$ so that  the following properties hold:
\begin{itemize}
\item[(i)] The lifts $\dot \gamma_1$ and $\dot \gamma_2$ of $\gamma_1$ and $\gamma_2$ to $T^1S^2$, respectively, form a Hopf link.
\item[(ii)] Each $\dot \gamma_i(\R)\subset T^1S^2$, $i=1,2$, is the binding of a rational open book decomposition such that each page is a disk-like global surface of section for the geodesic flow.
\item[(iii)] $\gamma_1$ is simple, elliptic-parabolic and its transverse rotation number belongs to $(\frac{1}{2},1]$.
\item[(iv)] If the lift $\dot \gamma$ of some unit speed prime closed geodesic $\gamma$ is a Hopf fiber, then $\dot \gamma$ is the binding of a rational open book decomposition with disk-like pages which are global surfaces of section for the geodesic flow. In particular, this applies to the lift of the shortest closed geodesic provided by Theorem \ref{theo_rad1}.
\end{itemize}
\end{theorem}

\begin{remark}
Since $F$ may be irreversible, a simple closed geodesic $\gamma$ traversed in the opposite direction may not be a geodesic. Hence Birkhoff annuli with boundary  $\dot \gamma \cup -\dot \gamma$ do not necessarily exist for irreversible Finsler metrics. Thus the natural global surfaces of section for the geodesic flow on $T^1S^2$ which arise in such cases, are rational disk-like global surfaces of section with a single boundary component, as those given in Theorem \ref{theo_finsler}. On the other hand, the Hopf link $\dot \gamma_1 \cup \dot \gamma_2$, obtained in Theorem \ref{theo_finsler}, is a candidate for an annulus-like global surface of section. This question is left for a future work.
\end{remark}

\begin{remark}
Non-reversible Finsler metrics on $S^2$ satisfying the pinching condition~\eqref{eq_flag1} were studied in~\cite{global}, where it is shown that closed geodesics satisfying certain contact-topological properties cannot exist. The proof is based on dynamical characterizations of the tight three-sphere from~\cite{hryn,HLS} and on a theorem due to Harris and Paternain~\cite{HP} stating that~\eqref{eq_flag1} implies dynamical convexity. Moreover, in~\cite{global} it is shown that the pinching condition~\eqref{eq_flag1} is sharp for dynamical convexity.
\end{remark}

\begin{proof}[Proof of Theorem \ref{theo_finsler}]Under the hypothesis \eqref{eq_flag1} on the flag curvature, we know from \cite{HP} that $\lambda_F$ is dynamically convex. From Theorem \ref{main1} there exists a $2$-unknotted non-hyperbolic periodic orbit $\dot \gamma_1$ of $\lambda_F$ corresponding to a closed geodesic $\gamma_1$ so that $\mu_{CZ}(\dot \gamma_1^2) =3$ and $\sl(\dot \gamma_1) = \frac{-1}{2}.$ Moreover, its rotation number  satisfies $\rho(\dot \gamma_1) \in \left(\frac{1}{2},1 \right]$.

Corollary~\ref{coro_main_2} implies that $\dot \gamma_1$ is the binding orbit of a rational open book decomposition. Each page $\Sigma$ of this open book is a rational global surface of section. The dynamics of the geodesic flow on $T^1S^2$ is represented by an area-preserving first return disk map $\psi:\Sigma \to \Sigma$. Since $\psi$ preserves a finite area, it has a fixed point which corresponds to a $2$-unknotted closed geodesic $\dot \gamma_2$ with $\sl(\dot \gamma_2)=\frac{-1}{2}$. The link $\dot \gamma_1 \cup \dot \gamma_2$ is a Hopf link. Corollary~\ref{coro_main_2} also implies that $\dot \gamma_2$ is the binding of an open book decomposition.

Multiplying the Finsler metric by a suitable constant we may assume that \begin{equation}\label{eq_flag2} 1 \leq K < \left(\frac{r+1}{r}\right)^2. \end{equation} To see that $\gamma_1$ is simple we use  the following estimate due to Rademacher \cite[Proposition 1]{rad2} on the length $\ell$ of a geodesic loop \begin{equation}\label{eq_rad2}\ell > \pi . \end{equation} Now the linearized flow transverse to the Reeb vector field along a geodesic $\gamma$ is represented by solutions of the following equation \begin{equation}\label{eq_lineari}\left(\begin{array}{c} \dot a(t)\\\dot b(t) \end{array} \right) = \left(\begin{array}{cc}0 & - K(t)\\ 1 & 0 \end{array} \right)\left(\begin{array}{c} a(t)\\ b(t) \end{array} \right), \end{equation} where $K(t)$ is the flag curvature of $T_{\gamma(t)} S^2$ in the direction of $\dot \gamma(t)$. Writing a linearized solution as $a(t)+ ib(t) = r(t)e^{i\theta(t)}$ for continuous functions $r(t)>0$ and $\theta(t)\in \R,$ we obtain from \eqref{eq_flag2} and \eqref{eq_lineari} that  \begin{equation}\label{eq_lineari2} \dot \theta(t) = K(t) \cos^2 \theta(t) + \sin^2 \theta(t)\geq 1, \forall t.\end{equation} If $\gamma_1$ is not simple, then it has at least two distinct loops and, from \eqref{eq_rad2} and \eqref{eq_lineari2} we obtain the following estimate for $\dot \gamma_1^2$ $$\theta(2l(\gamma_1))-\theta(0) > 4 \pi.$$ However, according to the definition of the Conley-Zehnder index discussed in Section~\ref{sec_prelim}, this implies that $\mu_{CZ}(\dot \gamma_1^2) \geq 5$, contradicting $\mu_{CZ}(\dot \gamma_1^2)=3$. Thus $\gamma_1$ must be simple. \end{proof}

\noindent {\bf Acknowledgements.} We became aware of Birkhoff's question on the existence of disk-like global surfaces of section bound by retrograde orbits after a talk given by Urs Frauenfelder in the Workshop on Conservative Dynamics and Symplectic Geometry held at IMPA -- Rio de Janeiro, August 2015. We would like to thank Urs Frauenfelder for his inspiring lecture. UH was partially supported by CNPq grant 309983/2012-6. PS was partially supported by FAPESP grants 2011/16265-8 and 2013/20065-0, and by CNPq grant 301715/2013-0.

\section{Preliminaries}\label{sec_prelim}

Here we review a couple of basic facts about contact geometry, Conley-Zehnder indices, pseudo-holomorphic curves in symplectizations etc. Throughout this section we let $(M,\xi)$ be a closed co-oriented contact three-manifold and $\alpha$ be a defining contact form for $\xi$ inducing the given co-orientation. The Reeb flow associated to $\alpha$ will be denoted by $\phi_t$. Later we will need to consider the projection
\begin{equation}\label{proj_along_reeb_direction}
\pi_\alpha : TM \to \xi
\end{equation}
along the Reeb direction.

\subsection{Invariants of linearized dynamics and asymptotic operators}

Consider a closed $\alpha$-Reeb orbit $P=(x,T) \in \P(\alpha)$ and fix a homotopy class $\beta$ of trivializations of the rank-$2$ symplectic vector bundle $(x(T\cdot)^*\xi,d\alpha)$. Any trivialization in class $\beta$ can be used to represent the linearized Reeb flow $d\phi_t:\xi|_{x(0)} \to \xi|_{x(Tt)}$ as a path of symplectic $2\times 2$ matrices $\varphi(t)$. For any $v\in\R^2\setminus0$ there is a smooth function $\theta_v(t)$ determined by $\varphi(t)v\in \R^+e^{i\theta_v(t)}$ up to a number in $2\pi\Z$. Set $\Delta(v)=(\theta_v(1)-\theta_v(0))/2\pi$ and consider the rotation interval $J=\{\Delta(v):v\in\R^2\setminus0\}$. The interval $J$ always has length strictly less than $1/2$. For $\epsilon>0$ small set $J_\epsilon = J-\epsilon$ and
\begin{equation*}
\mu_{CZ}(P,\beta) = \left\{ \begin{aligned} & \text{$2k$ if $k\in J_\epsilon$}, \\ & \text{$2k-1$ if $J_\epsilon\subset (k-1,k)$} \end{aligned} \right.
\end{equation*}
with $k\in\Z$.
This is the so-called Conley-Zehnder index of $P$ with respect to $\beta$, see~\cite{fols}. This definition does not depend on the choice of $\epsilon>0$ small enough and of the $d\alpha$-symplectic trivialization in the class $\beta$. The transverse rotation number of $P$ with respect to $\beta$ is defined as
\begin{equation*}
\rho(P,\beta) = \lim_{t\to+\infty} \frac{\theta_v(t)}{2\pi t}.
\end{equation*}
This is independent of $v\in\R^2\setminus 0$ and of the choice of $d\alpha$-symplectic trivialization in the class $\beta$.

If $P$ is contractible and $c_1(\xi)|_{\pi_2(M)}=0$ then we denote by $\mu_{CZ}(P)$ and $\rho(P)$ the Conley-Zehnder index and the transverse rotation number of $P$ computed with respect to a $d\alpha$-symplectic trivialization of $x(T\cdot)^*\xi \to \R/\Z$ which extends to a capping disk for $P$. Then, again assuming $P$ is contractible, we have
\begin{equation}\label{iteration_formula_rho}
\rho(P^n) = n\rho(P) \ \ \ \forall n\geq1
\end{equation}
and if, in addition, $P^n$ is non-degenerate then
\begin{equation}\label{iteration_formula_mu}
\begin{aligned}
& \text{$P^n$ hyperbolic} \Rightarrow \mu_{CZ}(P^n)=n\mu_{CZ}(P) \\
& \text{$P^n$ elliptic} \Rightarrow \mu_{CZ}(P^n) = 2\lfloor n\rho(P)\rfloor + 1.
\end{aligned}
\end{equation}

\begin{remark}\label{rmk_global_triv_SO(3)}
The standard (universally tight) contact structure $\xi_{\rm std}$ on $L(2,1)$ is a trivial symplectic vector bundle, and any two choices of global symplectic trivializations compatible with its standard co-orientation are homotopic. Thus we have well-defined invariants $\mu_{CZ}(P)$ and $\rho(P)$ for every closed Reeb orbit $P$, contractible or not, of any contact form defining $\xi_{\rm std}$ with its standard co-orientation. The corresponding iteration formulae~\eqref{iteration_formula_rho} and~\eqref{iteration_formula_mu} remain equally valid for non-contractible closed Reeb orbits in this special case.
\end{remark}

We denote by $\J_+(\xi)$ the set of complex structures on $\xi$ which are compatible with $d\alpha$. This set depends only on $\xi$ and its given co-orientation, and is non-empty and contractible when equipped with the $C^\infty$-topology. We may write $\J$ for simplicity when the context is unambiguous. Associated to $J \in \J$ and $P=(x,T) \in \P(\alpha)$ there is a so-called asymptotic operator $A:W^{1,2}(x(T\cdot)^*\xi) \to L^2(x(T\cdot)^*\xi)$ defined by
\begin{equation}
A:\eta\mapsto -J(\nabla_t\eta-T\nabla_\eta X_\alpha)
\end{equation}
where $\nabla$ is any symmetric connection on $TM$. Note that $A$ is, in fact, independent of a choice of such connection. In a $(d\alpha,J)$-unitary frame this operator gets represented as $v(t) \mapsto -i\dot v(t)-S(t)v(t)$ where $S(t)$ is a path of symmetric matrices. In fact, if the linearized Reeb flow is represented in this frame as the path of symplectic $2\times 2$ matrices $\varphi(t)$ then $S(t) = -i\dot\varphi(t)\varphi^{-1}(t)$. It follows that $A$ inherits a number of properties studied in~\cite{props2}, namely, its spectrum $\sigma(A)$ is discrete, real, accumulates at $\pm\infty$ and consists of eigenvalues of multiplicity at most two. If $\beta$ is an arbitrary homotopy class of $d\alpha$-symplectic trivializations of $x(T\cdot)^*\xi$ then we can use a given trivialization in class $\beta$ to represent an eigensection associated to some $\nu\in\sigma(A)$ as a path $t\in\R/\Z \mapsto v(t) \in \R^2$. Moreover, $v(t) \in \R^+ e^{i\theta(t)}$ is nowhere vanishing, the winding number $(\theta(1)-\theta(0))/2\pi\in\Z$ is independent of the choice of eigensection in the eigenspace of $\nu$ and, consequently, will be simply denoted by $\wind(\nu,\beta)$. Note that $A$ depends on $J$ but its eigenvalues and their winding numbers do not. They satisfy $\nu\leq \nu' \Rightarrow \wind(\nu,\beta) \leq \wind(\nu',\beta)$ and for any $k\in\Z$ there are two eigenvalues (multiplicites accounted) with winding equal to $k$. We single out two special eigenvalues
\begin{equation*}
\begin{array}{ccc} \nu^{<0} = \sup \{\nu\in\sigma(A) \mid \nu<0\}, & & \nu^{\geq0} = \inf \{\nu\in\sigma(A) \mid \nu\geq0\}, \end{array}
\end{equation*}
denote the associated winding numbers by
\begin{equation}
\begin{array}{ccc} \wind^{<0}(A,\beta) = \wind(\nu^{<0},\beta), & & \wind^{\geq0}(A,\beta) = \wind(\nu^{\geq0},\beta) \end{array}
\end{equation}
and recall here the formula
\begin{equation}
\begin{aligned}
\mu_{CZ}(P,\beta) &= 2\wind^{<0}(A,\beta) + \wind^{\geq0}(A,\beta)-\wind^{<0}(A,\beta) \\
&= \wind^{<0}(A,\beta) + \wind^{\geq0}(A,\beta)
\end{aligned}
\end{equation}
which is studied in~\cite{props2}.

\subsection{Rational self-linking number of rational unknots}

Let $K\subset M$ be a knot transverse to $\xi$, which we orient by the co-orientation of $\xi$. Assume that $K$ is $k$-unknotted, for some $k\geq1$, and consider an oriented $k$-disk $u:\D\to M$ for $K$. Choose a non-vanishing section $Z$ of $u^*\xi$, any exponential map $\exp$ on $M$ and, for $\epsilon>0$ small, consider the immersion $\gamma_\epsilon : t \in \R/\Z \mapsto \exp_{u(e^{i2\pi t})}(\epsilon Z(e^{i2\pi t})) \in M \setminus K$. The rational self-linking number $\sl(K,u) \in \Q$ is defined as
\begin{equation}
\sl(K,u) = \frac{1}{k^2} \ (\text{algebraic intersection number of $\gamma_\epsilon$ with $u$}).
\end{equation}
Here the domains $\R/\Z$ of $\gamma_\epsilon$ and $\D$ of $u$ are given their standard orientations, and $M$ is oriented by $\xi$. It can be shown that if $c_1(\xi)$ vanishes on $\pi_2(M)$ then $\sl(K,u)$ does not depend on the choice of $u$, and in this case we might denote this invariant simply by $\sl(K)$.

\subsection{Pseudo-holomorphic curves}\label{par_hol_curves}

Any $J\in\J$ uniquely determines an almost complex structure $\jtil$ on $\R\times M$ by
\begin{equation}\label{defn_adapted_J}
\begin{array}{ccc} \jtil(\partial_a)=X_\alpha, & & \jtil|_\xi=J. \end{array}
\end{equation}
Here and in the following we might see $\xi$, $\alpha$, $X_\alpha$ as $\R$-invariant objects on the symplectization $$ (\R\times M,d(e^a\alpha)). $$ We may briefly say that $\jtil$ is determined, or induced, by $(\alpha,J)$.

Recall from~\cite{93} that if $(S,j)$ is a closed Riemann surface and $\Gamma\subset S$ is a finite set, then a map $\util:S\setminus \Gamma \to \R\times M$ is a finite-energy curve if it is pseudo-holomorphic $$ \bar\partial(\util) = \frac{1}{2} \left( d\util + \jtil(\util)\circ d\util \circ j \right) = 0 $$ and has finite Hofer energy $$ E(\util) = \sup_{\phi\in\Lambda} \int_{S\setminus\Gamma} \util^*d(\phi\alpha). $$ Here $\Lambda$ is the set of smooth functions $\phi:\R\to[0,1]$ satisfying $\phi'\geq 0$. Points in $\Gamma$ are called punctures. A puncture $z_0$ is positive, or negative, if $a(z)\to+\infty$, or $a(z)\to-\infty$, as $z\to z_0$ respectively. If $z_0$ is neither positive nor negative then it is called removable. The analysis from~\cite{93} shows that finite-energy curves can be smoothly continued across a removable puncture.

If $z_0$ is a puncture and $\psi:(V,z_0)\to (B,0)$ is a holomorphic chart, where $B\subset \C$ is the open unit ball centered at the origin, then we consider positive cylindrical coordinates $(s,t)\in (0,+\infty)\times\R/\Z$ on $V\setminus \{z_0\}$ defined by $z\simeq \psi^{-1}(e^{-2\pi(s+it)})$. Similarly one defines negative cylindrical coordinates $(s,t) \in (-\infty,0)\times \R/\Z$ by $z\simeq \psi^{-1}(e^{2\pi(s+it)})$. In both cases we call these cylindrical coordinates centered at $z_0$. If $z_0$ is non-removable then we set its sign to be $+1$ or $-1$ depending whether it is positive or negative, respectively.

\begin{theorem}[Hofer~\cite{93}]
Let $z_0$ be a non-removable puncture, choose $(s,t)$ positive cylindrical coordinates centered at $z_0$, and let $\epsilon$ be the sign of $z_0$. Write $\util(s,t)=(a(s,t),u(s,t)) \in \R\times M$ with respect to these coordinates. Then for any sequence $s_n\to+\infty$ there exists a subsequence $s_{n_k}$, a closed Reeb orbit $P=(x,T)$ and $d\in\R$ such that the sequence of loops $t\mapsto u(s_{n_k},t)$ $C^\infty$-converges to $t\mapsto x(\epsilon Tt+d)$ as $k\to\infty$.
\end{theorem}

In order to handle degenerate contact forms we need the following definition, taken from~\cite{hryn}.

\begin{definition}[Non-degenerate punctures]\label{defn_non_deg_punct}
Let $z$ be a non-removable puncture of the finite-energy curve $\util=(a,u)$, $\epsilon\in\{1,-1\}$ be the sign of the puncture $z$ and $(s,t)$ be positive cylindrical coordinates centered at $z$. We call $z$ a non-degenerate puncture of $\util$ if the following holds:
\begin{itemize}
\item There exists a periodic orbit $P=(x,T)$ and $c,d\in\R$ such that $u(s,t) \to x(\epsilon Tt+d)$ and $|a(s,t)-\epsilon Ts-c| \to 0$ as $s\to+\infty$, uniformly in $t$.
\item If $\zeta(s,t)\in \xi|_{x(\epsilon Tt+d)}$ is defined by $u(s,t)=\exp_{x(\epsilon Tt+d)}(\zeta(s,t))$ for $s\gg1$, then $\exists b>0$ such that $e^{bs}|\zeta(s,t)| \to 0$ as $s\to+\infty$, uniformly in $t$.
\item If $\pi_\alpha \circ du$ does not vanish identically then it does not vanish when $s$ is large enough.
\end{itemize}
The orbit $P$ is called the asymptotic limit $\util$ at $z$.
\end{definition}

A partial description of the results from~\cite{props1} reads as follows.

\begin{theorem}[Hofer, Wysocki and Zehnder~\cite{props1}]\label{thm_partial_asymp}
If $\alpha$ is non-degenerate then every non-removable puncture of a finite-energy curve is non-degenerate in the sense of Definition~\ref{defn_non_deg_punct}.
\end{theorem}

\begin{definition}\label{defn_martinet}
A Martinet tube for an orbit $P=(x,T)\in \P(\alpha)$ is a pair $(U, \Psi)$ where $U$ is a neighborhood of $x(\R)$ in $M$ and $\Psi:U \to \R/\Z \times B$ is a diffeomorphism ($B \subset \R^2$ is an open ball centered at the origin and $\R/\Z\times B$ is equipped with coordinates $(\theta,x_1,x_2)$) satisfying
\begin{itemize}
  \item[a)] $\Psi^*(f(d\theta+x_1dx_2)) = \alpha$ where the smooth function $f:\R/\Z \times B \to \R^+$ satisfies $f|_{\R/\Z \times 0} \equiv T_{\rm min}$ and $df|_{\R/\Z \times 0} \equiv 0$.
  \item[b)] $\Psi(x(T_{\rm min}t)) = (t,0,0)$.
\end{itemize}
Here $T_{\rm min}$ is the minimal positive period of $x$.
\end{definition}

\begin{remark}
The existence of such Martinet tubes follows from a Moser-type argument, see~\cite{props1}. In the notation above, we may assume that $\partial_{x_1}$ along $t\mapsto x(T_{}t)$ is any non-vanishing section of $(x(T_{\rm min}\cdot)^*\xi$.
\end{remark}

Let $\util=(a,u):(S\setminus\Gamma,j)\to(\R\times M,\jtil)$ be a non-constant finite-energy pseudo-holomorphic map as above, where $\Gamma\subset S$ is the set of non-removable punctures and $\jtil$ is induced by $(\alpha,J)$, with $J\in\J_+(\xi)$. Let $z_0\in\Gamma$ be a puncture, let $\epsilon$ be its sign, and choose $(s,t)$ holomorphic cylindrical coordinates which are positive if $z_0$ is positive, or negative if $z_0$ is negative. Assume that $\alpha$ is non-degenerate and let $P=(x,T)$ be the asymptotic limit of $\util$ at $z_0$, which exists in view of Theorem~\ref{thm_partial_asymp}. Choose a Martinet tube $(U,\Psi)$ for $P$. If $\epsilon s\gg1$ there are well-defined functions $\theta(s,t)\in\R/\Z$, $x_1(s,t),x_2(s,t) \in \R$, $z(s,t)\in\R^2$ by $$ \begin{array}{ccc} \Psi \circ u(s,t) = (\theta(s,t),x_1(s,t),x_2(s,t)), & & z(s,t)=(x_1(s,t),x_2(s,t)). \end{array} $$ We continue to denote by $\theta(s,t)$ a lift to the universal covering $\R$ of $\R/\Z$. Then $\theta,z$ can be seen as functions of $(s,t)\in\R^2$ such that $z$ is $1$-periodic in $t$ and $\theta(s,t+1)=\theta(s,t)+k$, where $k\geq 1$ is the covering multiplicity of $P$. Up to a rotation we may assume that $\theta(s,0)\to 0$.

\begin{theorem}[Hofer, Wysocki and Zehnder~\cite{props1}, Siefring~\cite{sie1}]\label{thm_asymptotics}
There exists an eigenvalue $\nu$ of the asymptotic operator $A$ at $P$ associated to $(\alpha,J)$ satisfying $\epsilon\nu<0$, a non-zero $v(t) \in \ker (A-\nu I)$, a function $R(s,t)\in\R^2$ defined for $\epsilon s\gg1$ and $t\in\R/\Z$, and constants $c\in\R$, $r>0$, such that the following hold. The functions $a(s,t),\theta(s,t)$ satisfy
\begin{align}
& \lim_{\epsilon s\to+\infty} e^{r|s|} \|\partial_s^{\beta_1}\partial_t^{\beta_2}[a(s,t)-Ts-c](s,\cdot)\|_{L^\infty(\R/\Z)} = 0, \label{rep_a_asymptotics} \\
& \lim_{\epsilon s\to+\infty} e^{r|s|} \|\partial_s^{\beta_1}\partial_t^{\beta_2}[\theta(s,t)-kt](s,\cdot)\|_{L^\infty(\R/\Z)} = 0 \label{rep_theta_asymptotics}
\end{align}
for all $\beta_1,\beta_2\geq 0$. Moreover, if $t\in\R/\Z\mapsto e(t) \in \R^2$ is the representation of $v(t)$ in the frame $\{\partial_{x_1},\partial_{x_2}\}$ along $x(Tt)$ induced by the Martinet tube $(U,\Psi)$, then
\begin{align}
& z(s,t) = e^{\nu s}(e(t)+R(s,t)) & & \lim_{\epsilon s\to+\infty} \|\partial_s^{\beta_1}\partial_t^{\beta_2}R(s,\cdot)\|_{L^\infty(\R/\Z)} = 0. \label{rep_z_asymptotics}
\end{align}
\end{theorem}

The eigenvalue $\nu$ and eigensection $v(t)$ as in the above theorem will be referred to as the asymptotic eigenvalue and asymptotic eigensection of $\util$ at the puncture~$z_0$. The following lemma is proved in~\cite{hryn}.

\begin{lemma}\label{lem_asymp_formula_nondeg_punct}
The conclusions of Theorem~\ref{thm_asymptotics} continue to hold if we replace the assumption that $\alpha$ is non-degenerate by the assumption that $z_0$ is a non-degenerate puncture of $\util$ in the sense of Definition~\ref{defn_non_deg_punct}.
\end{lemma}

We pause our discussion of the theory to prove a lemma which will be crucial for our later analysis.

\begin{definition}[Relatively prime punctures]
Let $z_0$ be a non-degenerate puncture of the finite-energy curve $\util$, let $P=(x,T)$ be the asymptotic limit of $\util$ at $z_0$, with minimal period $T_{\rm min}>0$ and covering multiplicity $k=T/T_{\rm min}$, and let $v$ be the asymptotic eigensection of $\util$ at~$z_0$. Choose a homotopy class $\beta$ of $d\alpha$-symplectic trivializations of $x(T_{\rm min}\cdot)^*\xi$. We call $z_0$ a relatively prime puncture if $\wind(v,\beta^k)$ and $k$ are relatively prime integers.
\end{definition}

\begin{remark}
Note that $\wind(v,\beta^k) \mod k$ is independent of the choice of $\beta$. Thus the above definition is independent of the choice of $\beta$.
\end{remark}

\begin{lemma}\label{lem_crucial_behavior_at_punct}
Let $z_0\in\Gamma$ be a non-degenerate relatively prime puncture of the finite-energy curve $\util=(a,u):S\setminus\Gamma\to \R\times M$. Then there exists a neighborhood $U$ of $z_0$ in $S$ such that $u|_{U\setminus\{z_0\}}$ is an embedding.
\end{lemma}

\begin{proof}
We only treat the case $z_0$ is a positive puncture, negative punctures are handled analogously. Let $(s,t)$ be positive holomorphic cylindrical coordinates centered at $z_0$. By means of a Martinet tube we can represent $u(s,t)$ in coordinates with the aid of functions $\theta(s,t)\in\R$, $z(s,t),e(t),R(s,t)\in\R^2$ of $(s,t)\in\R^2$, and $\nu<0$, as in Theorem~\ref{thm_asymptotics}, so that $$ u(s,t+\Z) \simeq (\theta(s,t)+\Z,z(s,t)) \in \R/\Z \times \R^2, \ \ \ \ \forall (s,t) \in \R^2. $$ We assume, by contradiction, the existence of $(s_n,t_n),(s'_n,t'_n)$ satisfying
\begin{equation}
\begin{array}{ccc} (s_n,t_n+\Z)\neq(s'_n,t'_n+\Z) \text{ in } \R\times\R/\Z, & & \min\{s_n,s'_n\}\to+\infty \end{array}
\end{equation}
and
\begin{equation}\label{eq_identities_un_s}
\begin{array}{ccc} \theta(s_n,t_n) \in \theta(s_n',t_n') + \Z & & z(s_n,t_n)=z(s_n',t_n'). \end{array}
\end{equation}
Without loss of generality we may assume that $t_n\to t_*$, $t_n'\to t_*'$ with $t_*,t_*'\in [0,1)$. We denote $R_n=R(s_n,t_n)$, $R_n'=R(s_n',t_n')$.

If $|s_n-s_n'|$ is not bounded then, up to relabeling and selection of a subsequence, we may assume that $s_n-s'_n\to+\infty$. Thus $\nu(s_n-s_n')\to-\infty$ and from~\eqref{rep_z_asymptotics} we have
\begin{equation}\label{eq_e_tn}
e^{\nu(s_n-s_n')}(e(t_n)+R_n)=e(t_n')+R_n'.
\end{equation}
Passing to the limit we get $e(t_*')=0$, absurd. This proves that $|s_n-s_n'|$ is bounded. Up to a subsequence we may assume that $s_n-s_n'$ is convergent. Again taking the limit in~\eqref{eq_e_tn} we find that
\begin{equation}\label{eq_e_t*}
ce(t_*)=e(t_*') \ \ \ \text{for $c=\lim_{n\to\infty} e^{\nu(s_n-s_n')} >0$.}
\end{equation}
Let $m_n$ be integers defined by
\begin{equation}\label{eq_identities_thetas}
\theta(s_n,t_n) = \theta(s_n',t_n') + m_n.
\end{equation}
Using~\eqref{rep_theta_asymptotics} we find real numbers $\epsilon_n\to0$ such that $$ kt_n + \epsilon_n=kt_n'+m_n. $$ Note that $|t_*-t_*'|<1 \Rightarrow |m_n|\leq k-1$, for all large $n$. Up to selection of a subsequence, we can pass to the limit to get
\begin{equation}
t_* = t_*' + \frac{m}{k} \ \ \ \text{for some $m\in\Z\cap[-k+1,k-1]$, $m_n\to m$.}
\end{equation}

Set $F(s,t) = (\theta(s,t),z(s,t)) \in \R\times \R^2$, and suppose that $m=0$. Hence $c=1$, $s_n-s_n' \to 0$ as $n\to\infty$ in view of~\eqref{eq_e_t*} and, in this case, $|s_n-s_n'|+|t_n-t_n'| \to 0$. Moreover, $m_n=0$ for large $n$, \eqref{eq_identities_un_s} and~\eqref{eq_identities_thetas} give $F(s_n,t_n)=F(s_n',t_n')$ for large $n$. Consider the sequence of functions $$ F_n(s,t) := M_nF(s+s_n,t) \ \ \text{where $M_n$ is the matrix} \ \ \begin{pmatrix} 1 & 0 \\ 0 & e^{-\nu s_n}I \end{pmatrix} $$ When $n$ is large $F_n(s,t)$ is defined on $[0,+\infty)\times\R$ and is explicitly given by $$ F_n(s,t) = (\theta(s+s_n,t),e^{\nu s}(e(t)+R(s+s_n,t))). $$ Then $$ F_n \to F_\infty(s,t) = (kt,e^{\nu s}e(t)) \ \ \text{as $n\to\infty$ in $C^\infty_{\rm loc}([0,+\infty)\times\R)$} $$ in view of Theorem~\ref{thm_asymptotics}. Note that the identities $F(s_n,t_n)=F(s_n',t_n')$ are equivalent to $F_n(0,t_n)=F_n(s'_n-s_n,t_n')$, for each $n$. Since $(0,t_n)$ and $(s'_n-s_n,t_n')$ both converge to $(0,t_*)=(0,t_*')$, we get that the rank of $DF_\infty(0,t_*)$ is at most $1$. But $$ DF_\infty(0,t_*) = \begin{pmatrix} 0 & k \\ \nu e(t_*) & \dot e(t_*) \end{pmatrix} \ \ \text{clearly has rank two because $e(t_*)\neq0$.} $$ This contradiction shows that $m\neq 0$, $t_*\neq t_*'$.
Up to relabeling,~\eqref{eq_e_t*} now reads $$ ce\left(t_*'+\frac{m}{k}\right) = e(t_*') \ \ \ \text{for some $1\leq m\leq k-1$}. $$ Since $e(t)$ solves a linear $1/k$-periodic ODE, we find that
\begin{equation}\label{final_conclusion_e(t)}
\text{$ce\left(t+\frac{m}{k}\right) = e(t)$ holds for all $t$.}
\end{equation}
Let $\varphi(t)$ be a choice of argument of $e(t)$, that is, $e(t) \in \R^+e^{i\varphi(t)}$, and set $$ j=\varphi(m/k)-\varphi(0). $$ Note that $j$ is an integer in view of~\eqref{final_conclusion_e(t)}. Denoting $\wind(e)=\varphi(1)-\varphi(0)$ we compute
\begin{align*}
m \ \wind(e) &= m(\varphi(1)-\varphi(0)) \\
&= \varphi(m)-\varphi(0) \\
&= \sum_{l=1}^k \varphi(lm/k)-\varphi((l-1)m/k) \\
&= kj.
\end{align*}
This shows that $k$ divides $m\wind(e)$. But we have assumed that $k$ and $\wind(e)$ are relatively prime, so $k$ must divide $m$ which is impossible because $m<k$. This contradiction concludes the argument.
\end{proof}

Finally we introduce the basic algebraic invariants of curves in symplectizations, following~\cite{props2}. Let $\util = (a,u) : (S\setminus\Gamma,j) \to (\R\times M,\jtil)$ be a finite-energy curve, where $\Gamma$ are non-removable punctures on the closed Riemann surface $(S,j)$. We assume that every puncture in $\Gamma$ is non-degenerate in the sense of Definition~\ref{defn_non_deg_punct} so that, according to Lemma~\ref{lem_asymp_formula_nondeg_punct}, the conclusions of Theorem~\ref{thm_asymptotics} hold at every puncture. If $u^*d\lambda$ does not vanish identically Hofer, Wysocki and Zehnder define
\begin{equation*}
\wind_\pi(\util) = \text{algebraic count of zeros of $\pi_\alpha \circ du$}
\end{equation*}
where $\pi_\alpha$ is the projection~\eqref{proj_along_reeb_direction}. Since $\pi_\alpha \circ du$ solves a Cauchy-Riemann type equation, the similarity principle will tell us that its zeros are isolated and contribute positively to the above count. Moreover, since $\pi_\alpha \circ du$ does not vanish near the punctures (see Definition~\ref{defn_non_deg_punct}) we can conclude that the above count is finite. Now let us symplectically trivialize $u^*\xi$ using a trivialization $\sigma$. Then for every $z\in\Gamma$ such $\sigma$ induces a homotopy class $\beta_{z,\sigma}$ of $d\alpha$-symplectic trivializations of $x_z(T_z\cdot)^*\xi$, where $P_z=(x_z,T_z)$ is the asymptotic limit of $\util$ at $z$. We now set $\wind(\util,z,\sigma):=\wind(v_z,\beta_{z,\sigma})$ where $v_z$ is the asymptotic eigenvalue of $\util$ at $z$, split $\Gamma = \Gamma^+ \sqcup \Gamma^-$ into positive and negative punctures, and define
\begin{equation}
\wind_\infty(\util) = \sum_{z\in\Gamma^+} \wind(\util,z,\sigma) - \sum_{z\in\Gamma^-} \wind(\util,z,\sigma).
\end{equation}
The above sum does not depend on the choice of $\sigma$. It relates to $\wind_\pi(\util)$ according to the formula
\begin{equation}
\wind_\pi(\util) = \wind_\infty(\util) - \chi(S) + \#\Gamma
\end{equation}
proved in~\cite{props2}.

\begin{definition}\label{defn_fast_planes}
A finite-energy plane $\util:\C\to \R\times M$ will be called fast if $\infty$ is a non-degenerate puncture and $\wind_\pi(\util)=0$.
\end{definition}

\section{Proof of Theorem~\ref{main1}}

If $\lambda$ is a dynamically convex contact form as in Theorem \ref{main1}, then the lifted contact form $\pi_{2,1}^* \lambda$ on $S^3$  is also dynamically convex and hence tight, see~\cite[Theorem~1.5]{char2}. Thus $\xi = \ker \lambda$ is a universally tight contact structure on $L(2,1)$ and, as such, coincides with $\xi_0$ up to a diffeomorphism. Therefore, we may assume that $\lambda = f\lambda_0$ for some smooth function $f:L(2,1) \to (0,+\infty)$, without any loss of generality.

Given $0<r_1<r_2$, let $\lambda_E = f_E \lambda_0$ be the contact form  on $S^3$ associated to an ellipsoid, where $f_E$ is given by
\begin{equation}\label{lambdaE}
f_E = \left\{\frac{x_1^2+y_1^2}{r_1^2} + \frac{x_2^2+y_2^2}{r_2^2} \right\}^{-1}.
\end{equation}
Since  $g_{2,1}^*\lambda_E = \lambda_E$, $\lambda_E$ descends to a contact form on $L(2,1)$, also denoted by $\lambda_E$. If
\begin{equation}\label{hip1}
\left(\frac{r_2}{r_1}\right)^2 \in \R \setminus \Q,
\end{equation}
then the Reeb flow of $\lambda_E$ on $L(2,1)$ contains precisely two nondegenerate elliptic simply covered periodic orbits $P_1 = \pi_{2,1}(S^1 \times \{0\})$ and $P_2=\pi_{2,1}(\{0\} \times S^1)$ forming a so-called Hopf link on $L(2,1)$. Both $P_1$ and $P_2$ are $2$-unknotted and have self-linking number $\frac{-1}{2}$. Their (prime) periods satisfy $$ T_1 = \frac{\pi r_1^2}{2} \ \ \mbox{ and } \ \ T_2 = \frac{\pi r_2^2}{2}, $$ respectively.  The Conley-Zehnder indices of the contractible periodic orbits $P_1^2$ and $P_2^2$ are given by $$ \begin{aligned} \mu_{CZ}(P_1^2)  =3 \ \ \mbox{ and } \ \ \mu_{CZ}(P_2^2) = 2k+1, \end{aligned} $$  where $k\in \Z$, $k>1$, is such that $r_2^2/r_1^2 \in (k-1,k)$. Given a contact form $\lambda = f\lambda_0$ on $L(2,1)$ we choose $0<r_1<r_2$ large enough so that
\begin{equation}\label{hip2}
f < f_E \ \ \ \text{pointwise on $L(2,1)$.}
\end{equation}

\subsection{The non-degenerate case}\label{subsec_nondeg_case}

The following proposition implies Theorem \ref{main1} under a non-degeneracy assumption on $\lambda$.

\begin{proposition}\label{prop1}
Let $\lambda=f\lambda_0$ where $f:L(2,1) \to (0,+\infty)$ is smooth. Choose $0<r_1<r_2$ with $r_2^2 / r_1^2 \in \R \setminus \Q$ in such a way that $f_E$ defined in equation \eqref{lambdaE} satisfies~\eqref{hip2}.  Assume that every $P\in\P(\lambda)$ with period less than or equal to $\pi r_1^2$ satisfies both conditions
\begin{itemize}
\item[a)] $P$ is non-degenerate.
\item[b)] If $P$ is contractible then $\mu_{CZ}(P)\geq 3$.
\end{itemize}
Let $J\in \J$ be a $d\lambda$-compatible complex structure on $\xi=\ker \lambda$ and let $\jtil$ be the cylindrical almost complex structure on $\R \times L(2,1)$ induced by $(\lambda,J)$. Then there exists a finite-energy $\jtil$-holomorphic plane $\tilde u=(a,u):\C \to \R \times L(2,1)$ satisfying:
\begin{itemize}
\item[i)] $\tilde u$ is an embedding.
\item[ii)] The asymptotic limit of $\tilde u$ is a non-degenerate periodic orbit $P_0^2=(x_0,2T_0)$ satisfying $\mu_{CZ}(P_0^2)=3$, where $P_0=(x_0,T_0)$ is a simply covered orbit with $1/2<\rho(P_0)<1$. In particular, $P_0$ and $P_0^2$ are elliptic and $\mu_{CZ}(P_0)=1$.
\item[iii)] $u$ is an embedding transverse to $X_\lambda$ and $u(\C) \cap x_0(\R) = \emptyset$. In particular, $x_0(\R)$ is a $2$-unknot with self-linking number $-1/2$.
\item[iv)] $E(\tilde u) = 2T_0 \leq \pi r_1^2$.
\end{itemize}
\end{proposition}

The remainder of subsection~\ref{subsec_nondeg_case} is devoted to the proof of Proposition~\ref{prop1}. From now on we assume that $r_1,r_2$ satisfy~\eqref{hip1}, that $f_E$ defined as in~\eqref{lambdaE} satisfies~\eqref{hip2}, and that every closed $\lambda$-Reeb orbit with period less than or equal to $\pi r_1^2$ satisfies conditions a) and b) from Proposition~\ref{prop1}.

\subsubsection{Conley-Zehnder indices}\label{sec_CZ_iteration}

The universally tight contact structure on $L(2,1)$ is a trivial vector bundle. Hence we can find a global $d\lambda$-symplectic frame of $\xi$, with respect to which we compute Conley-Zehnder indices and transverse rotation numbers of all closed $\lambda$-Reeb orbits, contractible or not. Under the standing assumptions of Proposition~\ref{prop1} the Conley-Zehnder indices of closed $\lambda$-Reeb orbits have certain iteration properties summarized in the following statements.

\begin{lemma}\label{lem_fundamental_CZ}
Let $P=(x,T)$ be a closed $\lambda$-Reeb orbit satisfying $\mu_{CZ}(P)=1$ and $T\leq\pi r_1^2/2$. Then $P$ is prime, non-contractible, elliptic and $1/2 < \rho(P) < 1$.
\end{lemma}

\begin{proof}
$P$ is non-contractible since otherwise $\mu_{CZ}(P)\geq 3$ by assumption. Assume that $P$ is not prime, so $P=P'^m$ for some $m\geq2$ and some prime closed $\lambda$-Reeb orbit $P'=(x,T/m)$. Since $\mu_{CZ}(P)=1$ we have $0<\rho(P')<1/m\leq1/2$ because $0<\rho(P)<1$ and $m\geq 2$. But $P'^2$ is contractible and, as such $\mu_{CZ}(P'^2)\geq 3 \Rightarrow 2\rho(P')=\rho(P'^2)>1$, contradiction. We have proved so far that $P$ is prime and non-contractible. Since $P^2$ is contractible we get $\mu_{CZ}(P^2)\geq 3 \Rightarrow 2\rho(P)=\rho(P^2)>1 \Rightarrow \rho(P)>1/2$. Since $\mu_{CZ}(P)=1 \Rightarrow \rho(P)<1$ one gets that $P$ is elliptic and $1/2<\rho(P)<1$.
\end{proof}

\begin{lemma}\label{lem_CZ_iteration}
Let $P=(x,T)$ be a contractible closed $\lambda$-Reeb orbit satisfying both $\mu_{CZ}(P)\geq3$ and $T\leq\pi r_1^2$. The following assertions hold.
\begin{itemize}
\item[i)] If we write $P=P'^m$ for some prime closed $\lambda$-Reeb orbit $P'$ and some $m\geq 1$, then $\mu_{CZ}(P'^k)\geq 1$ for all $1\leq k\leq m$, and $\rho(P')>1/2$.
\item[ii)] If $\mu_{CZ}(P)=3$ then either $P$ is prime, or $P=P'^2$ for some elliptic non-contractible prime closed $\lambda$-Reeb orbit $P'$ satisfying $1/2 < \rho(P') < 1$.
\end{itemize}
\end{lemma}

\begin{proof}
Item i) follows from $\mu_{CZ}(P)\geq 3 \Rightarrow m\rho(P')=\rho(P)>1$ and $\rho(P'^j)=j\rho(P') > j/m > 0$ for all $j$. Inequality $\rho(P')>1/2$ holds because either $m=1$ and $\rho(P')=\rho(P)>1$ or $m\geq 2$ so we can estimate $\mu_{CZ}(P'^2)\geq 3 \Rightarrow 2\rho(P')>1$.

To prove ii), assume that $P$ is not prime, so $P=P'^m$ for some $m\geq2$ and some prime closed $\lambda$-Reeb orbit $P'=(x,T/m)$. Since $\mu_{CZ}(P)=3$ we have $1/m<\rho(P')<2/m$ because $1<\rho(P)<2$. In particular $\mu_{CZ}(P')=1$. All the desired conclusions about $P'$ now follow from the previous lemma. It only remains to be shown that $m=2$, but do note that $2/m>1/2 \Rightarrow m\in\{2,3\}$ and that if $m=3$ then $P$ would not be contractible because $P'$ is not contractible. Hence $P=P'^2$.
\end{proof}

\subsubsection{Symplectic cobordisms}\label{sec_non_cyl_cob}

Following \cite{convex}, choose $h:\R \times L(2,1) \to \R^+$ smooth and satisfying
\begin{itemize}
\item $h(a,p) = f(p)$ if $a\leq -2$.
\item $h(a,p) = f_E(p)$ if $a\geq 2$.
\item $\frac{\partial h}{\partial a} \geq 0$ and $\frac{\partial h}{\partial a} > \sigma >0$ on $[-1,1] \times L(2,1)$, for some constant $\sigma>0$.
\end{itemize}
The family of contact forms $\lambda_a:= h(a, \cdot) \lambda_0$, $a\in \R$, interpolates $\lambda$ and $\lambda_E$ in a monotonic way and $\xi=\ker \lambda_a$ does not depend on $a$. Choose $J_E \in \J(\lambda_E)$ and let $J_a \in \J(\lambda_a)$, $a\in \R$, be a smooth family of $d\lambda_a$-compatible complex structures on $\xi$ so that $J_a= J$ if $a \leq -2$ and $J_a = J_E$ if $a \geq 2$. We consider smooth almost complex structures $\bar J$ on the symplectization $\R \times L(2,1)$ with the following properties. On $(\R \setminus [-1,1]) \times L(2,1)$ we consider the usual recipe given by
$$ \begin{array}{cccccc} \bar J|_\xi=J_a & \text{and} & \bar J \cdot \partial_a = X_a & \text{on} & T(\R\times L(2,1))|_{\{a\}\times L(2,1)}, & \text{for} \ |a| > 1. \end{array} $$
where $X_a$ is the Reeb vector field of $\lambda_a$. On $[-1,1] \times L(2,1)$ we only require $\bar J$ to be compatible with the symplectic form $d(h\lambda_0)$. The space of such almost complex structures on $\R \times L(2,1)$ is non-empty and contractible in the $C^\infty$-topology and will be denoted by $\J(\lambda_E,J_E,\lambda,J)$.

\subsubsection{Generalized finite energy curves}\label{sec_generalized}

Let $\Gamma \subset \C$ be a finite set. We shall consider non-constant maps $$ \tilde u:\C \setminus \Gamma \to \R \times L(2,1) $$ satisfying
\begin{equation}\label{eq_pseudo}
d\tilde u \circ i = \bar J(\tilde u) \circ  d \tilde u
\end{equation}
for some $\bar J \in \J(\lambda_E,J_E,\lambda,J)$, with finite Hofer energy $$ E(\tilde u) < \infty. $$ The energy $E(\tilde u)$ is defined as follows: let $\Lambda$ be the collection of smooth functions $\phi:\R \to [0,1]$ satisfying $\phi' \geq 0$ and $\phi= 1/2 $ on $[-1,1]$. Then $$ E(\tilde u) := \sup_{\phi \in \Lambda} \int_{\C \setminus \Gamma} \tilde u^* d(\phi(a)\lambda_a), $$ where $\lambda_a$ is seen as a $1$-form on $\R \times L(2,1)$ depending on $a\in \R$. These maps are called generalized finite energy spheres and if $\Gamma =\emptyset$  they are called generalized finite energy planes.

As in the case of cylindrical almost complex structures, the set of punctures of a generalized finite energy sphere $\tilde u=(a,u)$ is non-empty and if $\tilde u$ has a non-removable puncture $z_0\in \Gamma$, then either $a(z)\to +\infty$ or $a(z) \to -\infty$ as $z \to z_0\in \Gamma$. Thus we can talk about positive and negative punctures. Moreover, any such $\tilde u$ has at least one positive puncture, in view of the exact nature of the symplectic forms taming $\bar J$. The following proposition follows from fundamental results of Hofer~\cite{93}.

\begin{proposition}\label{step1}
If $\tilde u=(a,u):\C \to \R \times L(2,1)$ is a generalized finite energy plane, then $a(z) \to + \infty$ as $|z| \to + \infty$, and $u(Re^{2\pi it}) \to x(Tt)$ in $C^\infty$ as $R \to +\infty$, where $x(t)$ is a contractible periodic orbit of $\lambda_E$ with period equal to $E(\tilde u)$.
\end{proposition}

\subsubsection{Fredholm theory}

The following theorems are important pieces of our analysis.

\begin{theorem}[Hofer, Wysocki and Zehnder~\cite{props3}]\label{theo_transversal}
There exists a dense subset $\J_{\rm reg} \subset \J(\lambda_E,J_E,\lambda,J)$ with the following property: if $\tilde u:\C \setminus \Gamma \to \R \times L(2,1)$ is a somewhere injective generalized immersed finite-energy sphere with one positive puncture, which is pseudo-holomorphic with respect to $\bar J \in \J_{\rm reg}$, then
\begin{equation}\label{eq_fredholm}
{\rm Fred}(\tilde u):= \mu_{CZ}(P_\infty) - \sum_{z \in \Gamma} \mu_{CZ}(P_z)+\#\Gamma -1 \geq 0.
\end{equation}
Here $\Gamma\subset \C$ is a non-empty finite set of negative punctures of $\tilde u$,  $P_\infty$ is the asymptotic limit of $\tilde u$ at the positive puncture $\infty$, and $P_z$ is the asymptotic limit of $\tilde u$ at  $z\in\Gamma$. Conley-Zehnder indices are computed using a global symplectic (with respect to $d\lambda$ or, equivalently, to $d\lambda_E$) trivialization of $\xi = \ker \lambda$.
\end{theorem}

From now on we fix $\bar J\in \J_{\rm reg}$ as in Theorem \ref{theo_transversal} so that \eqref{eq_fredholm} always holds. This will play an important role in the compactness argument in \S~\ref{sec_compactness} below.

We shall focus on the space of generalized finite energy $\bar J$-holomorphic planes $\tilde u:\C \to \R \times L(2,1)$  which are asymptotic at $\infty$  to the periodic orbit $P_1^2$ of $\lambda_E$. Recall that $P_1\subset L(2,1)$ is a non-contractible elliptic periodic $\lambda_E$-Reeb orbit with smallest period $\pi r_1^2/2$, and that $\mu_{CZ}(P_1^2)=3$. In particular $E(\tilde u)=\pi r_1^2$.

\begin{theorem}[Hofer, Wysocki and Zehnder~\cite{props3}]\label{thm_fred_theory_gen_planes}
Let $\util_0:\C\to \R\times L(2,1)$ be an embedded finite-energy $\bar J$-holomorphic plane, asymptotic to $P_1^2$ at the positive puncture $\infty$. Then there exists a smooth embedding $$ \Phi:\C\times B \to \R\times L(2,1) $$ where $B\subset \R^2$ is an open ball centered at the origin, with the following properties:
\begin{itemize}
\item[(i)] $\Phi(z,0) = \util_0(z)$, $\forall z\in\C$.
\item[(ii)] For every $\tau\in B$ the map $z\mapsto \Phi(z,\tau)$ is a generalized finite-energy $\bar J$-holomorphic plane asymptotic to $P_1^2$.
\item[(iii)] If $\vtil_n:\C\to \R\times L(2,1)$ are generalized finite-energy $\bar J$-holomorphic planes asymptotic to $P_1^2$ such that $\vtil_n \to \util_0$ in $C^\infty_{\rm loc}$, then we can find $A_n,B_n\in\C$, $\tau_n\in\R^2$, such that $A_n\to 1$, $B_n\to 0$, $\tau_n\to(0,0)$ and $$ \vtil_n(z) = \Phi(A_nz+B_n,\tau_n) \ \forall z\in\C $$ for all $n$ large enough.
\end{itemize}
\end{theorem}

The above theorem is valid for any $\bar J \in \J(\lambda_E,J_E,\lambda,J)$, even if $\bar J$ does not belong to $\J_{\rm reg}$.

\subsubsection{Existence of generalized finite energy planes}

Denote by $\Theta$ the space of generalized finite-energy $\bar J$-holomorphic planes asymptotic to $P_1^2$, modulo holomorphic reparametrizations. Observe that since $P_1$ is non-contractible, every plane $\tilde u\in \Theta$ must be somewhere injective.

\begin{theorem}\label{existence}
The set $\Theta$ is non-empty. More precisely, there exists a finite energy $\jtil_E$-holomorphic plane $\util=(a,u) : \C \to \R \times L(2,1)$ asymptotic to $P_1^2$ at $\infty$ such that $\tilde u$ is an embedding, where $\jtil_E$ is the cylindrical almost complex structure induced by $(\lambda_E,J_E)$. In particular, after translating in the $\R$-direction, we may assume that $\inf a(\C) > 2$, so that $\util$ represents an element of $\Theta$.
\end{theorem}

\begin{proof}
Consider the cylindrical almost complex structure $\tilde J_E$ induced by $(\lambda_E,J_E)$ on $\R \times L(2,1)$. Recall that $\lambda_E$ is dynamically convex, nondegenerate and admits the $2$-unknotted periodic orbit $P_1$ with period $\pi r_1^2/2$.  Proposition 6.8 from \cite{HLS} implies that for some $J_0 \in \J(\lambda_E)$  there exists an embedded (fast) $\tilde J_0$-holomorphic plane $\tilde u_0=(a_0,u_0):\C \to \R \times L(2,1)$ where $\jtil_0$ is the cylindrical almost complex structure induced by $(\lambda_E,J_0)$. This plane is asymptotic to $P_1^2$ at $\infty$ and has energy $E(\tilde u_0)= \pi r_1^2$. It follows from Proposition~\ref{prop_compactness_1} that there exists some embedded fast $\tilde J_E$-holomorphic plane $\tilde u:\C \to \R \times L(2,1)$ asymptotic to $P_1^2$. Now we may suitably translate $\tilde u$ in the $\R$-direction so that $\min_{z\in \C} a(z) > 2$. In this way $\tilde u$ also solves equation \eqref{eq_pseudo} and hence is a generalized finite energy plane. By construction, it is embedded, asymptotic to $P_1^2$ at $\infty$ and  $E(\tilde u)=\pi r_1^2$.
\end{proof}

\subsubsection{Embedding controls}\label{sec_embedding_controls}

The set $\Theta$ is a non-empty smooth $2$-dimensional manifold. In fact, Theorem~\ref{existence} gives slightly more information: some connected component $\Theta' \subset \Theta$ contains embedded planes $\util=(a,u)$ satisfying $\inf a(\C) > 2$ which, consequently, can be seen as $\jtil_E$-holomorphic.

In this section  we use Siefring's intersection theory, see~\cite{sie1,sie2}, to prove some important embedding properties of the elements of the distinguished connected component $\Theta'$.

\begin{proposition}\label{prop_embed}
If $\tilde u,\tilde v\in \Theta'$ then the following hold true.
\begin{itemize}
\item[(i)] $\tilde u$ is an embedding.
\item[(ii)] Either $\tilde u(\C) \cap \tilde v(\C) = \emptyset$ or $\tilde u(\C) = \tilde v(\C)$.
\end{itemize}
\end{proposition}

\begin{proof}
In~\cite{sie2} Siefring introduces the generalized intersection number $[\tilde u]*[\tilde v]$ between two generalized finite energy planes $\tilde u,\tilde v\in \Theta$ with distinct images, and the generalized self-intersection number $[\tilde u]*[\tilde u]$ of a generalized finite energy plane $\tilde u\in \Theta$. Such numbers compute actual intersections, or self-intersections and critical points, plus `intersections at infinity', or `self-intersections at infinity', which are contributions  given by the asymptotic behavior of the planes near the puncture at $\infty$.  It turns out that the these additional contributions at infinity make these integers invariant under homotopies through so-called asymptotically cylindrical maps, see~\cite[Theorem 2.1]{sie2}. Hence, these numbers do not depend on $\tilde u,\tilde v\in \Theta'$. For us, the importance of these numbers  relies on the following implications
\begin{equation}\label{intuv}
[\tilde u]*[\tilde v]=0 \Rightarrow \tilde u(\C) \cap \tilde v(\C)=\emptyset,
\end{equation}
and
\begin{equation}\label{intuu}
[\tilde u]*[\tilde u] =0 \Rightarrow \tilde u \mbox{ is an embedding}.
\end{equation}
For proofs of \eqref{intuv} and \eqref{intuu} see Theorems 2.2 and 2.3 in \cite{sie2}, respectively.

If we verify the identities $[\tilde u]*[\tilde u]=0$ and $[\tilde u]*[\tilde v]=0$ for special choices of elements in $\Theta'$, then we conclude from the homotopy invariance of these numbers that these identities hold for any pair of elements in $\Theta'$. Using the implications \eqref{intuv} and \eqref{intuu} the proof of Proposition~\ref{prop_embed} will be complete.

By results of~\cite{HLS}, there exists a smooth one parameter family of embedded finite energy $\tilde J_E$-holomorphic planes $\tilde u_\tau=(a_\tau,u_\tau)$, $\tau\in S^1$, such that $\{u_\tau\}$ determines a rational open book decomposition of $L(2,1)$ with disk-like pages and binding~$P_1$. All of such planes are asymptotic to $P_1^2$ at $\infty$, and automatically fast since $\mu_{CZ}(P_1^2)=3$, see below. Translating each plane in the $\R$-direction, if necessary, we may assume that  $\min_{z\in \C} a_\tau(z) >2, \forall \tau$. This implies, in particular, that we can see $\tilde u_\tau \in \Theta, \forall \tau$, and $\Theta'$ is the connected component of $\Theta$ containing these planes.

For every $\tau$ we have $u_\tau(\C) \cap P_1 = \emptyset$. This follows, for instance, from arguing as in Lemma~\ref{lemma_intersection} below since $P_1$ is $2$-unknotted and has self-linking number $-1/2$. Here we just need to note that each $\util_\tau$ can be thought of as a plane which is pseudo-holomorphic with respect to the cylindrical almost complex structure $\jtil_E$ induced by $(\lambda_E,J_E)$, which allows us to use tools that are special to the cylindrical case, such as $\wind_\pi$, $\wind_\infty$ etc. Note that fast planes must be somewhere injective, in fact, fast planes must be immersions in view of $\wind_\pi=0$, and if they were not somewhere injective then we would be able to factor them through somewhere injective planes using a polynomial of degree at least two, thus forcing the existence of critical points, absurd. By~\cite[Theorem~2.3]{props2} we conclude that each map $u_\tau:\C\to L(2,1)\setminus P_1$ is an embedding.

As observed above, the integers $\wind_\pi(\tilde u_\tau)$ and $\wind_\infty(\tilde u_\tau)$ are well-defined tools at our disposal since $\util_\tau$ can be thought of as $\jtil_E$-holomorphic planes. From $\mu_{CZ}(P_1^2)=3$ we obtain
\begin{equation}\label{windinfi}
\wind_\infty(\tilde u_\tau)=\wind^{<0}(A_{P_1^2})=1.
\end{equation}
Using that each $\tilde u_\tau$ is $\tilde J_E$-holomorphic, we compute $[\tilde u]*[\tilde u]$ as in the case of cylindrical almost complex structures. Corollary 5.17 from~\cite{sie2} applied to our situation states that
\begin{equation}\label{genself1}
[\tilde u_\tau]*[\tilde u_\tau]=0 \Leftrightarrow u_\tau(\C)\cap P_1= \emptyset \mbox{ and } d_0(\tilde u_\tau)=0.
\end{equation}
Here $d_0(\tilde u_\tau) = \wind^{<0}(A_{P_1^2}) - \wind_\infty(\tilde u_\tau)\geq 0$ and~\eqref{windinfi} implies that $d_0(\tilde u_\tau)=1-1=0$. Since $u_\tau(\C)\cap P_1=\emptyset$, we conclude from~\eqref{genself1} that $[\tilde u_\tau]*[\tilde u_\tau]=0, \forall \tau$. Thus, by homotopy invariance, $[\tilde u]*[\tilde u]=0$ for all $\tilde u\in \Theta'$ and by~\eqref{intuu}, $\tilde u$ is an embedding for all $\tilde u\in \Theta'$.

Now consider $\tau_1\neq \tau\in S^1\Rightarrow \tilde u_{\tau_1}(\C)\cap \tilde u_\tau(\C)=\emptyset$. Corollary 5.9 from \cite{sie2} applied to our case asserts that \begin{equation}\label{genint2} [\tilde u_{\tau_1}]*[\tilde u_{\tau}]=0 \Leftrightarrow u_{\tau_1}(\C)\cap P_1= u_\tau(\C)\cap P_1=\emptyset \mbox{ and } d_0(\tilde u_\tau)=0.\end{equation} Again, since the right hand side of \eqref{genint2} is already verified, we get that $[\tilde u_{\tau_1}]*[\tilde u_{\tau}]=0$. By the homotopy invariance we have $[\tilde u]*[\tilde v]=0$ for any $\tilde u,\tilde v\in \Theta'$. We conclude from \eqref{intuv} that $\tilde u(\C) \cap \tilde v(\C)=\emptyset$ for any $\tilde u,\tilde v\ \in\Theta'$ satisfying $\tilde u(\C)\neq \tilde v(\C)$, completing the proof of Proposition~\ref{prop_embed}.
\end{proof}

\subsubsection{Compactness}\label{sec_compactness}

In this section we study compactness properties of the set $\Theta$. Let $\alpha < \beta$ be real numbers and let $\Theta_{\alpha,\beta}\subset \Theta$ be the subset of equivalence classes represented by maps $\tilde u=(a,u) \in \Theta$ satisfying
\begin{equation}\label{eq_minu}
\alpha \leq \min \tilde u:= \min_{z\in \C} a(z) \leq \beta.
\end{equation}

\begin{proposition}\label{lem_compact}
Given a sequence $[\tilde u_n]$ in $\Theta_{\alpha,\beta}$, there exists $[\tilde u] \in \Theta_{\alpha,\beta}$ and $n_j \to \infty$ such that, up to holomorphic reparametrizations and $\R$-translations, we have $\tilde u_{n_j} \to \tilde u$ in $C^\infty_{\rm loc}$ as $j\to\infty$.
\end{proposition}

In order to prove Proposition~\ref{lem_compact},  we shall make use of the SFT compactness theorem from~\cite{sftcomp}, see the recent book of C. Abbas~\cite{Abb} for a nice exposition. To state this theorem in our quite simple situation, we introduce the notion of a bubbling-off tree as follows. Consider a finite, rooted and oriented (away from the root) tree $\T$, and a finite set $\U$ of finite energy holomorphic spheres
so that the following properties hold.
\begin{itemize}
\item[(i)] There is a bijective correspondence between vertices $q\in \T$ and finite-energy punctured spheres $\tilde u_q\in \U.$ Each  $\tilde u_q:\C \setminus \Gamma_q \to \R \times L(2,1)$ is pseudo-holomorphic with respect to either $\tilde J$, $\tilde J_E$ or $\bar J$. Moreover, each ordered path $(q_1,\dots,q_N)$ from the root $q_1=r$ to a leaf $q_N$, where $q_{k+1}$ is a direct descendant of $q_k$, contains at most one vertex $q_i$ such that $\util_{q_i}$ is $\bar J$-holomorphic, in which case $\util_{q_j}$ is $\jtil_E$-holomorphic $\forall 1\leq j < i$, and $\util_{q_j}$ is $\jtil$-holomorphic $\forall i< j \leq N$.
\item[(ii)] Each sphere $\tilde u_q$ has exactly one positive puncture at $\infty$ and $0\leq \#\Gamma_q<+\infty$ negative punctures, where $\Gamma_q$ is the set of negative punctures of $\util_q$.
\item[(iii)] If the vertex $q$ is not the root then $q$ has an incoming edge $e$ from a vertex $q'$, and $\#\Gamma_q$ outgoing edges $f_1,\ldots, f_{\#\Gamma_q}$ to vertices $p_1,\ldots,p_{\#\Gamma_q}\in \T$, respectively. The edge $e$ is associated to the positive puncture of $\tilde u_q$ and the edges $f_1,\ldots,f_{\#\Gamma_q}$ are associated to the negative punctures of $\tilde u_q$. 
The asymptotic limit of $\tilde u_q$ at its positive puncture coincides with the asymptotic limit of $\tilde u_{q'}$ at its negative puncture associated to $e$. In the same way, the asymptotic limit of $\tilde u_q$ at a negative puncture corresponding to $f_i$ coincides with the asymptotic limit of $\tilde u_{p_i}$ at its unique positive puncture. If  $\tilde u_q$ is $\tilde J_E$-holomorphic then $\tilde u_{p_i}$ is either $\tilde J_E$ or $\bar J$-holomorphic. If $\tilde u_q$ is either $\bar J$ or $\tilde J$-holomorphic then $\tilde u_{p_i}$ is necessarily $\tilde J$-holomorphic.
\item[(iv)] If the contact area of $\tilde u_q$ vanishes and $\util_q$ is $\tilde J_E$-holomorphic or $\tilde J$-holomorphic then $\#\Gamma_q \geq 2$.
\end{itemize}
We will denote by $\B=(\T,\U)$ the bubbling-off tree where $\T$ and $\U$ satisfy all properties above.

\begin{remark}
In (i) for a path $(q_1,\dots,q_N)$ from the root to a leaf it might be the case that no $\util_{q_i}$ is $\jbar$-holomorphic. In (iii) the vertices $p_1,\dots,p_{\#\Gamma_q}$ are the direct descendants of $q$.
\end{remark}

Now consider a sequence $\tilde u_n=(a_n,u_n)$ of generalized finite energy planes representing elements of $\Theta$. Clearly $E(\tilde u_n)$ is uniformly bounded by a constant $\pi r_1^2$. Assume moreover that all periodic orbits of $\lambda$ having period $\leq \pi r_1^2$ are nondegenerate. A corollary of the SFT compactness theorem is the following statement.

\begin{theorem}\label{thm_general_compactness}
Up to a subsequence of $\util_n$, still denoted by $\util_n$, there exists a bubbling-off tree $\B=(\T,\U)$ with the following properties.
\begin{itemize}
\item For every vertex $q$ of $\T$ there exist sequences $z^q_n,\delta^q_n\in\C$ and $c_n^q\in \R$ such that
\begin{equation}\label{renormaliza}
\tilde u_n(z^q_n + \delta^q_n\cdot)+c^q_n \to \tilde u_q(\cdot) \mbox{ in } C^\infty_{\rm loc}(\C\setminus\Gamma_q) \mbox{ as } n\to \infty.
\end{equation}
Here $\tilde u+c := (a+c,u)$ where $\tilde u=(a,u)$ and $c\in\R$.
\item The curve $\tilde u_r$ is asymptotic to $P_1^2$ at $\infty$, and every asymptotic limit of all planes $\util_q$ are contractible closed orbits with periods $\leq \pi r_1^2$ of the Reeb flow of either $\lambda_E$ or $\lambda$.
\end{itemize}
\end{theorem}


In order to prove Proposition~\ref{lem_compact} we first establish two auxiliary lemmas, which exhibit more refined properties of the SFT-convergence of a sequence $\util_n=(a_n,u_n) \in \Theta$ to a bubbling-off tree $\B=(\T,\U)$ described above.

\begin{lemma}\label{lem_aux_compactness}
Let $z_n^q,\delta_n^q,c_n^q$ be sequences such that~\eqref{renormaliza} holds for all vertices $q$ of~$\T$. Then we can assume, up to selection of a subsequence still denoted by $\util_n$, that one of the three mutually excluding possibilities holds for every vertex $q$.
\begin{itemize}
\item[I)] $c_n^q$ is bounded, $a_n(z_n^q+\delta_n^q\cdot)$ is $C^0_{\rm loc}(\C\setminus\Gamma_q)$-bounded and $\util_q$ is a $\bar J$-holomorphic curve.
\item[II)] $c_n^q \to -\infty$, $a_n(z_n^q+\delta_n^q\cdot)\to+\infty$ in $C^0_{\rm loc}(\C\setminus\Gamma_q)$ as $n\to\infty$, and $\util_q$ is a $\jtil_E$-holomorphic curve.
\item[III)] $c_n^q \to +\infty$, $a_n(z_n^q+\delta_n^q\cdot)\to-\infty$ in $C^0_{\rm loc}(\C\setminus\Gamma_q)$ as $n\to\infty$, and $\util_q$ is a $\jtil$-holomorphic curve.
\end{itemize}
Moreover, if $q$ is a vertex for which III) holds then $\util_q$ is asymptotic at its positive puncture to a closed $\lambda$-Reeb orbit having period strictly less than $\pi r_1^2$. In particular, III) does not hold for the root $r$.
\end{lemma}

\begin{proof}
Fix a vertex $q$ of $\T$, write $\util_q=(a_q,u_q)$ for the components of $\util_q$, and choose any point $z_0 \in \C\setminus \Gamma_q$. Up to choice of a subsequence we may assume that either $a_n(z_n^q+\delta_n^qz_0)$ is bounded, or $a_n(z_n^q+\delta_n^qz_0)\to+\infty$ or $a_n(z_n^q+\delta_n^qz_0)\to-\infty$. If $a_n(z_n^q+\delta_n^qz_0) \to \pm\infty$ then $a_n(z_n^q+\delta_n^q\cdot) \to \pm\infty$ on compact subsets of $\C\setminus\Gamma_q$ since the derivatives of $a_n(z_n^q+\delta_n^q\cdot)$ are uniformly bounded on compact subsets of $\C\setminus\Gamma_q$. If $a_n(z_n^q+\delta_n^qz_0)$ is bounded then $a_n(z_n^q+\delta_n^q\cdot)$ is bounded on compact subsets of $\C\setminus\Gamma_q$ for the same reason. From~\eqref{renormaliza} we get $a_n(z_n^q+\delta_n^qz_0) + c_n^q \to a_q(z_0)$ as $n\to\infty$, so we can conclude that if $a_n(z_n^q+\delta_n^qz_0) \to +\infty$ then $c_n^q\to-\infty$, if $a_n(z_n^q+\delta_n^qz_0)$ is bounded then so is $c_n^q$, and if $a_n(z_n^q+\delta_n^qz_0) \to -\infty$ then $c_n^q\to+\infty$. In the first case $\util_q$ is $\jtil_E$-holomorphic, in the second case $\util_q$ is $\jbar$-holomorphic, and in the third case $\util_q$ is $\jtil$-holomorphic. This proves that I) or II) or III) holds, up to a subsequence.

Fix a vertex $q$ for which III) holds, and denote by $P_q=(x_q,T_q)$ the closed $\lambda$-Reeb orbit which is the asymptotic limit of $\util_q$ at the positive puncture $\infty$. For every $\epsilon>0$ there exists $R_\epsilon\gg1$ such that the loop $t\mapsto u_q(R_\epsilon e^{i2\pi t})$ has $\lambda$-action $\geq T_q-\epsilon/2$. Hence we find $n_\epsilon\gg1$ such that $$ n\geq n_\epsilon \ \Rightarrow \ \text{the loops $t\mapsto u_n(z_n^q+\delta_n^qR_\epsilon e^{i2\pi t})$ have $\lambda$-action $\geq T_q-\epsilon$.} $$ We can also assume that
\begin{equation*}
n\geq n_\epsilon \ \Rightarrow \ \sup_{t\in\R/\Z} a_n(z_n^q+\delta_n^qR_\epsilon e^{i2\pi t}) \leq -2
\end{equation*}
since III) holds. Consider the of $\jbar$-holomorphic finite-energy half-cylinder
\[
\vtil_\epsilon:[0,+\infty)\times \R/\Z \to \R\times L(2,1)
\]
defined by
\[
\vtil_\epsilon(s,t) = (b_\epsilon(s,t),v_\epsilon(s,t)) = \util_{n_\epsilon}(z_{n_\epsilon}^q+\delta_{n_\epsilon}^qR_\epsilon e^{2\pi(s+it)}).
\]
It satisfies
\begin{equation}\label{boundary_values_b_k}
\begin{array}{ccc} b_\epsilon(0\times\R/\Z) \subset (-\infty,-2] & & \inf_t b_\epsilon(s,t) \to +\infty \text{ as } s\to+\infty \end{array}
\end{equation}
and
\begin{equation}\label{areas_going_to_zero}
\int_{[0,+\infty) \times \R/\Z} \vtil_\epsilon^*d(h\lambda_0) \leq \pi r_1^2 - T_q + \epsilon.
\end{equation}
where $h$ is the function fixed in \S~\ref{sec_non_cyl_cob}.

At this point we wish to apply the monotonicity lemma since $d(h\lambda_0)$ is symplectic on $[-1,1]\times L(2,1)$ and $\jbar$ is $d(h\lambda_0)$-compatible. We can either use a version of the monotonicity lemma for manifolds with boundary, or argue with the usual monotonicity lemma for closed manifolds as follows. Consider the diffeomorphism $\Phi:\C^2\setminus0 \to \R\times S^3$ given by $\Phi(z)=(a,p)$ where $|z|=e^a$ and $p=z/|z|$. Denoting by $\alpha_0$ the standard Liouville form~\eqref{std_liouville_form_S3} on $\C^2$, and by $\tilde\lambda_0$ its pull-back to $S^3$ via the inclusion $S^3\hookrightarrow \C^2$, then $\pi_{2,1}^*\lambda_0 = \tilde\lambda_0$ and $\Phi^*(e^a\tilde\lambda_0)=\alpha_0$. Hence $\Phi$ defines a symplectomorphism $(\C^2\setminus0,\omega_{\rm std}) \simeq (\R\times S^3,d(e^a\tilde\lambda_0))$, where $\omega_{\rm std} = d\alpha_0$ is the standard symplectic form on $\C^2$. Choose $M>0$ such that $e^{-M+1}<f<f_E<e^{M-1}$ holds pointwise on $L(2,1)$, and a smooth function $\tilde h:[-M,M]\times S^3 \to \R^+$ satisfying
\begin{itemize}
\item $\tilde h(a,p) = h(a,\pi_{2,1}(p))$ for all $(a,p) \in [-1,1]\times S^3$.
\item $\tilde h(a,p)=e^a$ on $[-M,-M+1]\times S^3 \sqcup [M-1,M]\times S^3$.
\item $\partial_a\tilde h>0$ on $[-M,M]\times S^3$.
\end{itemize}
Here $h$ is the function used in~\S~\ref{sec_non_cyl_cob}. Then $\Phi^*d(\tilde h \tilde\lambda_0)$ is a symplectic form on $K := \Phi^{-1}([-M,M]\times S^3)$ coinciding with $\omega_{\rm std}$ near $\partial K$. The open ball $B = \{z\in\C^2:|z|\leq e^{2M}\}$ equipped with $\omega_{\rm std}$ is symplectomorphic with an open set in $\C P^2$ equipped with a positive constant multiple of the Fubini-Study symplectic form. We have proved that $([-1,1]\times S^3,(id_\R\times\pi_{2,1})^*d(h\lambda_0)) = ([-1,1]\times S^3,d(\tilde h\tilde\lambda_0))$ can be symplectically embedded into $(\C P^2,\Omega)$ where $\Omega$ is a symplectic form coinciding with a multiple of the Fubini-Study form near some complex line. Since the loops $t\mapsto v_\epsilon(s,t)$ are contractible in $L(2,1)$, the cylinders $\vtil_\epsilon$ can be lifted to cylinders $\tilde V_\epsilon:[0,+\infty)\times\R/\Z \to \R\times S^3$. These lifted cylinders are pseudo-holomorphic with respect to $\jbar' := (id_\R\times \pi_{2,1})^*\jbar$. Note that $\jbar'$ is compatible with $(id_\R\times\pi_{2,1})^*d(h\lambda_0)$ on $[-1,1]\times S^3$ since $\jbar$ is compatible with $d(h\lambda_0)$ on $[-1,1]\times L(2,1)$. Take $-1+\epsilon>a_\epsilon^- > -1$ and $1-\epsilon < a_\epsilon^+ < 1$ regular values of $b_\epsilon(s,t)$. Hence by~\eqref{boundary_values_b_k} each $S_\epsilon = \tilde V_\epsilon^{-1}([a_\epsilon^-,a_\epsilon^+]\times S^3)$ is a smooth domain with boundary in $(0,+\infty)\times\R/\Z$. The $\tilde V_\epsilon$ maps $\partial S_\epsilon$ into $\partial([a_\epsilon^-,a_\epsilon^+]\times S^3)$ and $\tilde V_\epsilon(S_\epsilon) \cap \{0\}\times S^3 \neq \emptyset$. Using the monotonicity lemma on $(\C P^2,\Omega)$ we find $\epsilon_0>0$ independent of $\epsilon$ such that
\begin{equation}
\begin{aligned}
\pi r_1^2 - T_q + \epsilon \geq \int_{S_\epsilon} \vtil_\epsilon^*d(h\lambda_0) &= \int_{S_\epsilon} \tilde V_\epsilon^*(id_\R\times \pi_{2,1})^*d(h\lambda_0) \geq \epsilon_0.
\end{aligned}
\end{equation}
Making $\epsilon\to0^+$ we get $T_q<\pi r_1^2$, as desired.
\end{proof}

The following lemma, left without proof, is standard and proved using the maximum principle combined with estimates for cylinders with small area, see~\cite{convex}.

\begin{lemma}\label{lem_uniform_ends_vertices}
Let $z_n^r,\delta_n^r,c_n^r$ be sequences such that~\eqref{renormaliza} holds for the root $r$. Then, by Theorem~\ref{thm_general_compactness}, $P_1^2=(x_1,\pi r_1^2)$ is the asymptotic limit of $\util_r$ at the positive puncture $\infty$. For every $\R/\Z$-invariant neighborhood $\W$ of $t\mapsto x_1(\pi r_1^2t)$ in $C^\infty(\R/\Z,L(2,1))$ and for every number $M>0$, there exist $R_0>0$ and $n_0$ such that if $R>R_0$ and $n>n_0$ then the loop $t\mapsto u_n(z_n^r+\delta_n^rRe^{i2\pi t})$ belongs to $\W$, and $a_n^r(z_n^r+\delta_n^rRe^{i2\pi t})+c_n^r > M$.
\end{lemma}

\begin{proof}[Proof of Proposition~\ref{lem_compact}]
Let $\tilde u_n=(a_n,u_n)$ be a sequence in $\Theta_{\alpha,\beta}$. By the SFT compactness theorem, we find a bubbling-off tree $\B=(\T,\U)$ which is the SFT-limit of a subsequence of $\tilde u_n$, which we still denote by $\util_n$.

We claim that $\T$ contains only one vertex, namely, its root~$r$. To prove this claim we start by recalling that for every vertex $q$ we find $z_n^q,\delta_n^q\in\C$ and $c_n^q\in\R$ such that~\eqref{renormaliza} holds. We may assume, without any loss of generality and up to choice of a subsequence, that one of the three possibilities I), II) and III) listed in Lemma~\ref{lem_aux_compactness} holds for each vertex $q$. Note, however, that III) is excluded since the $\R$-components of the planes $\util_n$ are bounded from below by $\alpha>-\infty$. Hence I) or II) of Lemma~\ref{lem_aux_compactness} holds for all vertices.

In case II) holds for the root $r$, we claim that $\util_r$ can not have negative punctures. In fact, in this case either its contact area vanishes or not; if it did then the defining properties of a bubbling-off tree tell us that it can not be a trivial cylinder over $P_1^2$, so the action of its asymptotic limits at negative punctures would be strictly less than $\pi r_1^2$, contradicting the fact that these asymptotic limits are contractible closed $\lambda_E$-Reeb orbits. If the contact area does not vanish then again asymptotic limits at negative punctures of $\util_r$ would be contractible closed $\lambda_E$-orbits with action strictly less than $\pi r_1^2$, again impossible. We conclude that $\util_r$ does not have negative punctures in this case. In case I) holds for $r$ we conclude from $\util_n\in\Theta_{\alpha,\beta}$ and~\eqref{renormaliza} that the $\R$-component of $\util_r$ is bounded from below, again excluding the possibility of negative punctures.

Since $\util_r$ does not have negative punctures in all cases, the tree has only one vertex. A combination of lemmas~\ref{lem_aux_compactness} and~\ref{lem_uniform_ends_vertices} will show that II) does not hold for the root $r$. In fact, assume that II) holds for the root $r$. By Lemma~\ref{lem_uniform_ends_vertices} we find $R_0>0$ such that $|z|>R_0 \Rightarrow a_n(z_n^r+\delta_n^rz) > \beta - c_n^r > \beta$ for all $n$ large enough. But $\inf \{ a_n(z_n^r+\delta_n^rz) : |z|\leq R_0 \} \to +\infty$ as $n\to\infty$ in this case. Hence, if II) holds for $r$ then $\inf \{a_n(z):z\in\C\} > \beta$ when $n$ is large, contradicting $[\util_n]\in\Theta_{\alpha,\beta}$.

We have proved that $\T$ has only one vertex (the root $r$) and that I) holds for the root $r$. The desired conclusion now follows from~\eqref{renormaliza}.
\end{proof}

\begin{proposition}\label{prop_inf}
There exists a sequence $\{[\tilde u_k=(a_k,u_k)]\} \subset \Theta'$ satisfying $$ \min_{z\in \C} a_k(z) \to -\infty \mbox{ as } k \to + \infty. $$
\end{proposition}

\begin{proof}
Adapting the argument of Proposition 4.17 in~\cite{convex} to our setup, we argue indirectly assuming that $$ \inf_{\tilde u = (a,u) \in \Theta'} \left[\inf a(\C)\right] = m > -\infty. $$ Recall that $\Theta'$ is the distinguished connected component of $\Theta$ described in~\S~\ref{sec_embedding_controls}, containing embedded finite-energy $\jtil_E$-holomorphic planes asymptotic to $P_1^2$.

For any $\alpha\in\R$ we denote $\Theta'_\alpha = \{ [\util=(a,u)]\in\Theta' : \inf a(\C) \geq\alpha \}$. Take a sequence $\tilde u_n = (a_n, u_n) \in \Theta'_m$ satisfying $\lim_{n \to \infty} \inf a_n(\C) = m$. From Theorem~\ref{existence} we necessarily have $m\leq 1$.

By Proposition~\ref{lem_compact}, we find a plane $[\tilde u_\infty = (a_\infty,u_\infty)] \in \Theta'_m$ so that up to a subsequence and possibly reparametrizations, $\tilde u_n \to \tilde u_\infty$ in $C^\infty_{\rm loc}(\C)$ as $n \to \infty$. For each $n$ choose $z_n\in\C$ such that $a_n(z_n) = \inf a_n(\C) \to m$. Using Lemma~\ref{lem_uniform_ends_vertices} we find that $|z_n|$ can be assumed to be bounded up to a subsequence, and up to a further subsequence we can also assume that $z_n\to z_\infty$ for some $z_\infty \in \C$. This gives $a_\infty(z_\infty)=m$ and, consequently, $\inf a_\infty(\C)=m$.

By Proposition~\ref{prop_embed}, $\tilde u_\infty$ is an embedding, so we can apply Theorem~\ref{thm_fred_theory_gen_planes} to find planes $[\util=(a,u)] \in \Theta'$ satisfying $\inf a(\C)<m$, a contradiction.
\end{proof}

Now we take a sequence $[\tilde u_n=(a_n,u_n)] \in \Theta' \subset \Theta$ satisfying $\min \tilde u_n \to -\infty$ as $n \to + \infty$. The existence of such a sequence is guaranteed by Proposition \ref{prop_inf}. By the SFT compactness theorem, there exists a bubbling-off tree $\B=(\T,\U)$ which is, up to a subsequence, the SFT-limit of the sequence $\tilde u_n$.

\begin{theorem}\label{theo_tree}
Under the assumptions of Proposition~\ref{prop1} the bubbling-off tree $\B=(\T,\U)$ obtained as an SFT-limit of the sequence $\{[u_n]\}_{n\in \N}$, contains only $2$ vertices, namely, its root $r$ and a second vertex $q$. The root $r$ corresponds to a $\bar J$-holomorphic finite energy cylinder $\tilde u_r: \C\setminus\{0\}\to \R \times L(2,1)$ which is asymptotic to $P_1^2$ at the positive puncture $\infty$ and to some $P \in \P(\lambda)$ at the negative puncture~$0$ satisfying $\mu_{CZ}(P)=3$. The vertex $q$ corresponds to an embedded finite energy $\tilde J$-holomorphic plane $\tilde u_q:\C\to \R \times L(2,1)$  asymptotic to $P$ at $\infty$.
\end{theorem}

\begin{proof}
Let $\tilde u_r=(a_r,u_r):\C \setminus \Gamma_r \to \R \times L(2,1)$ be the finite energy sphere associated to the root $r\in \T$. Then $P_1^2$ is the asymptotic limit of $\tilde u_r$ at its positive puncture. We claim that $\tilde u_r$ is $\bar J$-holomorphic, i.e., it is a generalized finite-energy punctured sphere. The argument to prove this is very similar to the one given in the proof of Proposition~\ref{lem_compact}. We include it here for completeness.

There are sequences $z_n^r,\delta_n^r,c_n^r$ such that~\eqref{renormaliza} holds for the root $r$. We can apply Lemma~\ref{lem_aux_compactness} to conclude that I) or II) of Lemma~\ref{lem_aux_compactness} holds for the root $r$.

Suppose that II) holds for the root $r$. Then $\util_r$ is $\jtil_E$-holomorphic and there are two cases: either its contact area vanishes or not. If the contact area of $\util_r$ vanishes then it has negative punctures and it is not a trivial cylinder over $P_1^2$; consequently the asymptotic limits of $\util_r$ at negative punctures are contractible closed $\lambda_E$-Reeb orbits with action strictly less than $\pi r_1^2$, absurd. If the contact area of $\util_r$ does not vanish and there are negative punctures then again we find contractible closed $\lambda_E$-Reeb orbits with action strictly less than $\pi r_1^2$, absurd. It remains to handle the case where $\util_r$ is a $\jtil_E$-holomorphic plane, i.e. no negative punctures. In this case using Lemma~\ref{lem_uniform_ends_vertices} we find $R_0>0$ such that $a_n(z_n^r+\delta_n^rz) > -c_n^r$ for all $z$ satisfying $|z| > R_0$ and $a_n(z_n^r+\delta_n^rz) > 0$ for all $z$ satisfying $|z|\leq R_0$, when $n$ is large enough. Since $c_n^r \to -\infty$ we find $\inf a_n(\C)\geq 0$ when $n$ is large enough, contradicting the fact that $\inf a_n(\C) \to -\infty$ as $n\to\infty$. Thus I) holds for the root.

Now we show that $\tilde u_r$ has negative punctures. Arguing indirectly assume that $\Gamma_r=\emptyset$. Then we know $\tilde u_r$ is a $\jbar$-holomorphic finite-energy plane. Since $c_n^r$ is bounded, using Lemma~\ref{lem_uniform_ends_vertices} we find $R_0>0$ such that $a_n(z_n^r+\delta_n^r\cdot)$ maps $\{|z| > R_0\}$ into $[0,+\infty)$ for all $n$ large enough. But $a_n(z_n^r+\delta_n^r\cdot)$ is uniformly bounded on compact subsets of $\C$, in particular on $\{|z|\leq R_0\}$. This is in contradiction to the fact that $\inf a_n(\C) \to -\infty$ as $n\to\infty$.

Next we show that $\#\Gamma_r=1$. We know that all asymptotic limits $P_z$, $z\in \Gamma_r$, of $\util_r$ are contractible and have period $\leq \pi r_1^2$. Hence each $P_z$ is nondegenerate and $\mu_{CZ}(P_z)\geq 3$, $\forall z\in \Gamma_r$.

Assume first that $\tilde u_r$ is not somewhere injective. Then $\tilde u_r$ factors as $\tilde u_r=\tilde v_r \circ Q$, where $\tilde v_r:\C \setminus \Gamma' \to \R\ \times L(2,1)$ is a somewhere injective generalized finite energy sphere with $1\leq \# \Gamma'$ and $Q$ is a polynomial with degree ${\rm deg} (Q)\geq 2$ satisfying $Q^{-1}(\Gamma')=\Gamma_r$. The set $\Gamma'$ consists of negative punctures of $\tilde v_r$, all of its asymptotic limits are periodic $\lambda$-Reeb orbits. The asymptotic limit $P_\infty$ of $\tilde v_r$ at $\infty$ must satisfy $P_\infty^{{\rm deg}(Q)} = P_1^2$. Since ${\rm deg}(Q)\geq 2$ and $P_1$ is simply covered it follows that
\begin{equation}\label{eq_gener0}
{\rm deg}(Q) =2 \mbox{ and } P_\infty = P_1.
\end{equation}
Recall that our choice of almost complex structure $\bar J\subset \J_{\rm reg}$ is generic as stated in Theorem \ref{theo_transversal}. Using a global $d\lambda_E$-symplectic (or $d\lambda$-symplectic) trivialization of the contact structure on $L(2,1)$ in order to compute the Conley-Zehnder indices, we know from \eqref{eq_fredholm} that \begin{equation}\label{eq_gener1}\mu_{CZ}(P_\infty) - \sum_{z'\in \Gamma'} \mu_{CZ}(P_{z'})+\#\Gamma' -1 =- \sum_{z'\in \Gamma'} \mu_{CZ}(P_{z'}) + \# \Gamma'  \geq 0,  \end{equation} where $P_{z'}$ is the asymptotic limit of $\tilde v_r$ at $z'\in \Gamma'$. Here we used that $\mu_{CZ}(P_\infty) = \mu_{CZ}(P_1) = 1$. The periods of all $P_{z'}$, $z'\in\Gamma'$, are less than the period of $P_\infty=P_1$, that is, less than $\pi r_1^2/2$; this follows from Stokes theorem. All contractible periodic $\lambda$-Reeb orbits with action $\leq \pi r_1^2$ are non-degenerate and have Conley-Zehnder index $\geq 3$. From Lemma~\ref{lem_CZ_iteration} i) we get $\mu_{CZ}(P_{z'})\geq 1$ for all $z'\in\Gamma'$.
It follows from \eqref{eq_gener1} that $\mu_{CZ}(P_{z'})=1$ for all $z'\in\Gamma'$, in particular, from Lemma~\ref{lem_fundamental_CZ} we conclude that all $P_{z'}$, $z'\in\Gamma'$, are prime, non-contractible, elliptic and $\rho(P_{z'})\in(1/2,1)$. Hence if $w\in \Gamma_r = Q^{-1}(\Gamma')$ then $Q'(w)=0$ because otherwise $P_{z'}$ is also an asymptotic limit of $\util_r$ and, consequently contractible, contradiction. Since $\deg(Q)=2$ there is only one point $w$ in $\C$ with $Q'(w)=0$, and we conclude that $\#\Gamma = \#\Gamma' = 1$ and we can assume that $\Gamma = \Gamma'=\{z_0=0\}$. Since $\mu_{CZ}(P_{z_0})=1$,  the asymptotic limit at the unique negative puncture of $\tilde u_r$ is the contractible periodic orbit $P:=P_{z_0}^2$ satisfying $\mu_{CZ}(P) = 3$.

Now assume that $\tilde u_r$ is somewhere injective. The asymptotic limit $P_z$ of $\tilde u$ at any negative puncture $z\in \Gamma_r$ is contractible and hence $\mu_{CZ}(P_z) \geq 3$. From \eqref{eq_fredholm} we have
\begin{equation}\label{eq_gener2}
\mu_{CZ}(P_1^2) - \sum_{z\in \Gamma_r} \mu_{CZ}(P_z)+\#\Gamma_r -1 =2- \sum_{z\in \Gamma_r} \mu_{CZ}(P_z) + \# \Gamma_r  \geq 0,
\end{equation}
where we have used that $\mu_{CZ}(P_1^2) =3$. Since $\Gamma_r \neq \emptyset$, we conclude from~\eqref{eq_gener2} that $\#\Gamma = 1$ and we can assume that $\Gamma_r = \{0\}$. Hence $\Gamma_r$ consists of a single negative puncture of $\tilde u_r$ whose asymptotic limit $P$ is contractible and has Conley-Zehnder index $3$.

Since $\tilde u_r$ is $\bar J$-holomorphic, we conclude that the finite energy sphere $\tilde u_q=(a_q,u_q):\C\setminus \Gamma_q\to \R \times L(2,1)$, corresponding to the second vertex $q\in \T$, the unique direct descendant of the root $r$, is a finite energy $\tilde J$-holomorphic sphere. Moreover, $\tilde u_q$ is asymptotic to $P=(x,T)$  at its positive puncture and since $\mu_{CZ}(P)=3$, either $P$ is prime or it is the double cover of the prime non-contractible periodic orbit $P'=(x,T/2)$; we used Lemma~\ref{lem_CZ_iteration}. All asymptotic limits at its negative punctures are contractible and have period $\leq \pi r_1^2$, hence are nondegenerate and have Conley-Zehnder index $\geq 3$. Next we show that $\Gamma_q = \emptyset$ and that $\tilde u_q$ is an embedding.

We argue indirectly to see that $\Gamma_q = \emptyset$. If not, denote by $P_{z}$ the asymptotic limit of $\tilde u_q$ at $z \in \Gamma_q$. If $\int_{\C \setminus \Gamma_q} u_q^*d\lambda =0$, then $\#\Gamma_q \geq 2$. In this case, from the classification of finite energy curves in symplectizations with vanishing $d\lambda$-area given in \cite[Theorem 6.11]{props2}, $\tilde u_q$ maps $\C \setminus \Gamma_q$ onto $\R \times P$ and we find  integers $k_{z}\geq 1$ so that $\tilde u_q$ is asymptotic to $(P')^{k_z}$ at each $z \in \Gamma_q$, where $P'$ is a periodic orbit of $\lambda$. We have two possibilities: either $P'=P$ is simple or $P'$ is simple, non-contractible and $(P')^2=P$.  If $P'=P$, then $\sum_{z\in \Gamma_q} k_z=1,$ contradicting $\#\Gamma_q \geq 2$. If $(P')^2=P$ then $\sum_{z \in \Gamma_q} k_{z} =2$. Using that $\#\Gamma_q\geq 2$, we obtain $\#\Gamma_q=2$, $k_z=1 \ \forall z\in \Gamma_q$, and $\tilde u_q$ is asymptotic to the non-contractible periodic orbit $P'$ at its two negative punctures, a contradiction.

We have concluded that $\int_{\C \setminus \Gamma_q} u_q^*d\lambda >0$.  In this case, the integers $\wind_\pi(\tilde u_q)$, $\wind_\infty(\tilde u_q,\infty)$ and $\wind_\infty(\tilde u_q,z),z\in \Gamma_q,$ are well defined and
\begin{equation}\label{eq_wind1}
0\leq \wind_\pi(\tilde u_q) = \wind_\infty(\tilde u_q,\infty) - \sum_{z \in \Gamma_q} \wind_\infty(\tilde u_q,z) - 1 + \# \Gamma_q.
\end{equation}
Using that $\mu_{CZ}(P_{z}) \geq 3$, we obtain
\begin{equation}\label{eq_windinf}
\wind_\infty(\tilde u_q,z) \geq 2.
\end{equation}
Now since $\mu_{CZ}(P) =3$, we have $\wind_\infty(\tilde u_q,\infty) \leq 1$. Using \eqref{eq_wind1} and \eqref{eq_windinf}, we obtain
$$
0 \leq 1 - 2\# \Gamma_q -1 + \#\Gamma_q = - \#\Gamma_q,
$$
and, therefore, $\Gamma_q = \emptyset$. We conclude that $\tilde u_q$ is a finite energy $\tilde J$-holomorphic plane and that $\wind_\pi(\tilde u_q) = \wind_\infty(\tilde u_q, \infty) -1 =0$. In particular, $u_q$ is an immersion transverse to the Reeb vector field $X_\lambda$. This also implies that $\tilde u_q$ is somewhere injective, otherwise it would factor as $\tilde u_q=\tilde v_q \circ Q$ for a finite energy $\tilde J$-holomorphic plane and a polynomial $Q:\C\to \C$ with degree $\geq 2$, forcing $\tilde u_q$ to have critical points, a contradiction.

At this point we know that the second vertex $q\in \T$ corresponds to a finite energy $\tilde J$-holomorphic plane asymptotic to $P$ and $\mu_{CZ}(P)=3$. In particular, $\T$ has only two vertices $\{r,q\}$.

We still need to show that $\tilde u_q:\C\to\R\times L(2,1)$ is an embedding. We know $\util_q$ is an immersion since $\wind_\pi(\util_q)=0$. If $\util_q$ is not an embedding we find that the set
\begin{equation*}
D = \{(z_1,z_2) \in \C\times\C \setminus\Delta \mid \util_q(z_1)=\util_q(z_2)\}
\end{equation*}
is non-empty, where $\Delta$ is the diagonal in $\C\times \C$. $D$ must be discrete, since a limit point of $D$ in $\C\times\C\setminus\Delta$ would force $\util_q$ to be multiply covered; this follows from the similarity principle. Thus, self-intersections of $\util_q$ are isolated. In view of~\eqref{renormaliza}, and of positivity and stability of intersections of pseudo-holomorphic immersions, we find self-intersections of the maps $\util_n$ for $n$ large enough, but this is impossible since the $\util_n$ are embeddings.
\end{proof}

\subsubsection{Proof of Proposition~\ref{prop1}}

By Theorem~\ref{theo_tree} there exists an embedded finite-energy $\jtil$-holomorphic plane $\util =(a,u)$ asymptotic to $P=(x,T)\in\P(\lambda)$ with $\mu_{CZ}(P)=3$. In particular, $\wind_\pi(\util)=0$ and
\begin{equation}\label{eq_tr}
\text{$u$ is an immersion transverse to the Reeb vector field $X_\lambda$.}
\end{equation}
We claim that
\begin{equation}\label{eq_uembd}
\text{$u|_{\C \setminus B_R(0)}$ is an embedding if $R>0$ is sufficiently large.}
\end{equation}
To prove this, we again use the identity $\mu_{CZ}(P)=3$. In fact, in this case, since $P$ is contractible, Lemma~\ref{lem_CZ_iteration} implies that either $P$ is simply covered and~\eqref{eq_uembd} is trivial, or $P=P'^2$ for some non-contractible simply covered periodic $\lambda$-orbit $P'=(x,T/2)$. In the latter case, since $\wind_\infty(\tilde u,\infty)=1$ is odd and the covering multiplicity of $P$ is $2$ we can apply Lemma~\ref{lem_crucial_behavior_at_punct} to obtain~\eqref{eq_uembd}. Thus~\eqref{eq_uembd} is proved in all cases.

Before proceeding we pause to draw some consequences of the Fredholm theory developed in~\cite{props3}. Since $\tilde u$ is an embedded fast plane, by~\cite[Theorem~A.1]{HLS} there is a smooth two-dimensional foliation $F$ of a neighborhood of $\util(\C)$ in $\R\times L(2,1)$ by embedded finite-energy $\tilde J$-holomorphic planes asymptotic to $P$, such that $\util(\C)$ is a leaf of $F$. Moreover, by the completeness statement contained in~\cite[Theorem~A.1]{HLS}, this family contains the $\R$-translations $\tilde u_\epsilon=(a+\epsilon,u)$ of $\tilde u$ for all $\epsilon>0$ small enough. This proves that
\begin{equation}\label{conseq_fred_theory_HLS}
\text{$\util_\epsilon(\C) \cap \util(\C) = \emptyset$ for all $\epsilon>0$ small enough.}
\end{equation}

Now we use~\eqref{conseq_fred_theory_HLS} in order to show that
\begin{equation}\label{eq_u_not_int_P}
u(\C) \cap x(\R) = \emptyset
\end{equation}
We argue indirectly, assuming that $u(\C) \cap x(\R) \neq \emptyset$. From \eqref{eq_tr} and from the fact that $\tilde u$ is asymptotic to $P$, we find $z,z'\in \C$ and $c>0$ such that $\tilde u(z') = \tilde u_c(z)$, where $\tilde u_c=(a+c,u)$. Fixing such $c>0$, we claim that for every $0<\epsilon<c$ there exists $R>0$ so that if $b\in [\epsilon,c]$ and $\tilde u(z) =\tilde u_b(w)$ then $|z|,|w|<R$. To see this, we argue indirectly and assume the existence of sequences $b_n\in [\epsilon,c]$, $z_n,w_n\in \C$ satisfying $\tilde u_{b_n}(z_n) = \tilde u(w_n)$ and $\max \{|z_n|,|w_n|\} \to +\infty$ as $n\to \infty$. Then
\begin{equation}\label{eq_azn}
\begin{aligned}
u(z_n)& =u(w_n), \\
a(z_n)+b_n& = a(w_n).
\end{aligned}
\end{equation}
from where we see that $b_n\neq0 \Rightarrow z_n \neq w_n$. Up to selection of a subsequence, we may assume that one of the following alternatives holds: (a) $|z_n|,|w_n| \to +\infty$ as $n\to \infty$; (b) $z_n$ is bounded and $|w_n|\to +\infty$ as $n\to \infty$; (c) $|z_n|\to +\infty$  as $n\to +\infty$ and $w_n$ is bounded. In view of $z_n \neq w_n$, case (a) is excluded by~\eqref{eq_uembd} and the first equation in~\eqref{eq_azn}.  The second equation in~\eqref{eq_azn} and the fact that $b_n$ is bounded implies that either both $z_n$ and $w_n$ are bounded or both $z_n$ and $w_n$ are unbounded. Hence cases (b) and (c) are also excluded and our claim follows. We conclude by stability and positivity of intersections of pseudo-holomorphic curves that $\tilde u(\C)$ intersects $\tilde u_\epsilon(\C)$ for all $\epsilon>0$ small enough, contradicting~\eqref{conseq_fred_theory_HLS}. The proof of~\eqref{eq_u_not_int_P} is complete.

A direct application of~\cite[Theorem~2.3]{props2} proves that $u:\C\to L(2,1)\setminus x(\R)$ is injective. This fact and~\eqref{eq_tr} together show that $u$ determines an oriented $m$-disk for $x(\R)$, where $m$ is the covering multiplicity of $P$. Again using~\eqref{eq_tr} we can compute the self-linking number of $P$ to be $-1/m$ as in~\cite[Lemma~3.10]{HLS}. By purely topological considerations, or by an application of Lemma~\ref{lem_CZ_iteration}, we already know that $m\in\{1,2\}$. To conclude that $m=2$ we need to invoke the following characterization of the tight $3$-sphere:

\begin{theorem}[Hofer, Wysocki and Zehnder~\cite{char1,char2}]\label{thm_char2}
Let $\lambda$ be a contact form on a closed connected $3$-manifold $M$. If the Reeb flow of $\lambda$ admits an unknotted periodic orbit, with prime period $T>0$, Conley-Zehnder index $3$, self-linking number $-1$, and so that any other contractible periodic $\lambda$-Reeb orbit with period $\leq T$ is nondegenerate and has Conley-Zender index $\geq 3$, then $(M,\xi = \ker \lambda)$ is contactomorphic to the sphere $S^3$ equipped with the standard tight contact structure $\xi_0 = \lambda_0$.
\end{theorem}

In fact, if $m=1$ then the above result would imply that $L(2,1)$ is diffeomorphic to $S^3$, absurd. Thus $x(\R)$ is a $2$-unknot with self-linking number $-1/2$. By Lemma~\ref{lem_CZ_iteration} the orbit $P'=(x,T/2)$ is elliptic and $1/2<\rho(P')<1$. The proof of Proposition~\ref{prop1} is now complete.

\subsection{The degenerate case}

Here we give a complete proof of Theorem~\ref{main1}. As observed before, since $\ker\lambda$ is universally tight, there is no loss of generality to assume that $\lambda=f\lambda_0$ for some $f:L(2,1)\to (0,+\infty)$ smooth. Choose $r_1,r_2$ satisfying~\eqref{hip1} and such that the function $f_E$ defined as in~\eqref{lambdaE} satisfies $f<f_E$ pointwise. Choose a $C^\infty$-neighborhood $\V$ of the constant function $1$ with the following property: if $h \in \V$ then $hf<f_E$ pointwise, and if the contractible closed Reeb orbit $P'=(x',T')\in \P(h\lambda)$ satisfies $T'\leq\pi r_1^2$, then $\mu_{CZ}(P')\geq 3$. Since $\lambda$ is dynamically convex, an easy application of the Arzel\`a-Ascoli theorem combined with the lower-semicontinuity of the Conley-Zehnder index will tell us that such neighborhood $\V$ exists.

Let $h\in\V$ and choose a sequence $h_n \to h$ in $C^\infty$ such that $\lambda_n := h_n\lambda = h_nf\lambda_0$ is non-degenerate, for all $n$. The existence of such a sequence is proved in \cite[Proposition 6.1]{convex}. Then $h_nf < f_E$ pointwise, for all large $n$.  We claim that if $n$ is large enough then $\lambda_n$ satisfies the hypotheses of Proposition~\ref{prop1}. Otherwise, there exists a subsequence, also denoted $\lambda_n$, so that each $\lambda_n$ admits a  contractible periodic orbit $Q_n$ with period $0<T_n \leq \pi r_1^2$ and index $\mu_{CZ}(Q_n) \leq 2$. Using Arzel\`a-Ascoli theorem we may assume, up to selection of a subsequence, that $Q_n$ converges in $C^\infty$ as $n \to +\infty$ to  a contractible periodic orbit $Q$ of $\lambda$ with period $\leq \pi r_1^2$. Due to the lower semi-continuity of the generalized Conley-Zehnder index we have $\mu_{CZ}(Q) \leq 2$, contradicting $h\in\V$.  Hence, by Proposition \ref{prop1}, $\lambda_n$ admits a $2$-unknotted simply covered periodic orbit $P_n$ with self-linking number $\frac{-1}{2}$ and $\mu_{CZ}(P_n^2) = 3$, $\forall n$. Moreover, their periods are uniformly bounded by $\frac{\pi r_1^2}{2}$. Therefore, using Arzel\`a-Ascoli theorem, we have $P_n \to P$ as $n \to +\infty$ in $C^\infty$, up to extraction of a subsequence, where $P$ is a periodic orbit of $\lambda$ with period $\leq \frac{\pi r_1^2}{2}$. Since $P_n$ is $2$-unknotted we know that $P_n^2$ is contractible, and hence $P^2$ is also contractible. Again due to the lower semi-continuity of the generalized Conley-Zehnder index, we must have $\mu_{CZ}(P^2)\leq 3$. Since $f\in\V$, $\mu_{CZ}(P^2)=3$ and this implies that $P$ is simply covered. In fact, assume $P$ factors as $P=P_0^k$ for some $P_0\in \P(\lambda)$ and an integer $k\geq2$. Since $\pi_1(L(2,1)) \simeq \Z_2,$  $P_0^2$ is contractible and, therefore, $\mu_{CZ}(P_0^2) \geq 3$. Since $P^2=(P_0^k)^2=(P_0^2)^k$ and $k\geq 2$, we obtain $\mu_{CZ}(P^2) \geq 5$, a contradiction. Since $P$ is simple and $P_n \to P$ as $n \to +\infty$, $P$ is transversally isotopic to each $P_n$ for all large $n$ and, therefore, $P$ is also $2$-unknotted and also has self-linking number $\frac{-1}{2}$.  The proof of Theorem \ref{main1} is finished.

\section{Proof of Theorem~\ref{thm_22}}

Throughout this section we fix a contact form $\lambda$ on $L(p,q)$ defining the standard contact structure $\xi=\ker\lambda$. We assume that $X_\lambda$ is tangent to an order $p$ rational unknot $K$ with self-linking number $-1/p$, which we orient by $X_\lambda$. Let $(x,T_{\rm min})$ be the prime closed $\lambda$-Reeb orbit determined by~$K$, that is, $K=x(\R)$ where $x:\R\to L(p,q)$ is a periodic $\lambda$-Reeb trajectory with minimal positive period $T_{\rm min}$. Set $T=pT_{\rm min}$, denote $P = (x,T)$ and assume that
\begin{equation}
\rho(P) > 1.
\end{equation}
We make no genericity assumptions on $\lambda$.

Consider $\P^*(\lambda) \subset \P(\lambda)$ the set defined in the statement of Theorem~\ref{thm_22}, consisting of closed orbits $P'\subset L(p,q)\setminus K$ which are contractible in $L(p,q)$ and have transverse rotation number (with respect to the Reeb flow of $\lambda$) equal to $1$. From now on we fix an oriented $p$-disk $u_0$ for $K$ obtained from Proposition~\ref{prop_nice_disk}, which is special robust for $(\lambda,K)$.

The remainder of this section is devoted to the proof of Theorem~\ref{thm_22}, so from now on we assume that~\eqref{linked_conv} holds for every orbit $P'\in\P^*(\lambda)$. We may write $\P^*$ instead of $\P^*(\lambda)$ for simplicity. The next statement is~\cite[Proposition 7.1]{HLS}.

\begin{proposition}\label{prop_global_sections}
Let $e$ be the singular point of the characteristic foliation of $u_0$, and fix any open neighborhood $V$ of $e$. Suppose that every orbit $P'\in\P^*$ satisfies~\eqref{linked_conv}. Then for every sequence of smooth functions $f_n:L(p,q)\to(0,+\infty)$ satisfying $f_n|_K\equiv1$, $df_n|_K\equiv0$, $f_n|_V\equiv 1$, $f_n\to 1$ in $C^\infty$ and such that $\lambda_n:=f_n\lambda$ is nondegenerate $\forall n$, one finds $n_0$ with the following property: for all $n\geq n_0$ there exists a rational open book decomposition $(K,\pi_n)$ with disk-like pages of order $p$ adapted to~$\lambda_n$ in the sense that all pages are (rational) disk-like global surfaces of section for the $\lambda_n$-Reeb flow.
\end{proposition}

The existence of sequences $f_n$ as above is standard. From now on we fix the choice of such a sequence, and denote
\begin{equation}\label{sequence_lambda_n}
\lambda_n := f_n\lambda.
\end{equation}

\begin{remark}\label{rmk_pages}
The pages of the rational open book decompositions whose existence is guaranteed by Proposition~\ref{prop_global_sections} above are images of embedded fast finite-energy planes under the projection $\R\times L(p,q) \to L(p,q)$. These planes are leaves of a finite-energy foliation, see~\cite[Section 7]{HLS}.
\end{remark}

The first step in the proof of Theorem~\ref{thm_22} is that for every $J\in\J_+(\xi)$ and every $(a,y)\in \R\times (L(p,q)\setminus K)$ there exists an embedded fast finite-energy $\jtil$-holomorphic plane asymptotic to $P$ passing through $(a,y)$, where $\jtil$ is determined by $(\lambda_n,J)$, for all $n\gg1$. This follows immediately from combining the above statement with the proposition below, which is a generalization of~\cite[Lemma~3.13]{openbook}. Its proof is found in the appendix.

\begin{proposition}\label{prop_compactness_1}
Let $\alpha$ be a non-degenerate defining contact form for the closed co-oriented contact $3$-manifold $(M,\xi)$. Assume that $c_1(\xi)|_{\pi_2(M)}=0$, and that there exists a periodic $\alpha$-Reeb trajectory $x:\R\to M$ such that $x(\R)$ is $p$-unknotted and has self-linking number $-1/p$. Fix a point $y\in M\setminus x(\R)$ and consider $P=(x,T=pT_{\rm min}) \in \P(\alpha)$ where $T_{\rm min}>0$ is the minimal positive period of $x$. Let $\J_{\rm fast}(\alpha,y)\subset \J_+(\xi)$ be the set of $d\alpha$-compatible complex structures $J:\xi\to\xi$ for which there exists some embedded fast finite-energy $\jtil$-holomorphic plane $\util=(a,u):\C\to\R\times M$ asymptotic to $P$ and satisfying $y\in u(\C)$; here $\jtil$ is determined by $(\alpha,J)$.

If $\rho(P)>1$ and every contractible orbit $P'=(x',T')\in \P(\alpha)$ in $M\setminus x(\R)$ and satisfying $\rho(P')=1$ also satisfies at least one of the following conditions
\begin{itemize}
\item[a)] $T'>T$,
\item[b)] $P'$ is not contractible on $M\setminus x(\R)$,
\end{itemize}
then either $\J_{\rm fast}(\alpha,y)=\emptyset$ or $\J_{\rm fast}(\alpha,y)=\J_+(\xi)$.
\end{proposition}

It follows from Proposition~\ref{prop_global_sections} and Remark~\ref{rmk_pages} that for every $y\in L(p,q)\setminus K$ and $n$ large enough the set $\J_{\rm fast}(\lambda_n,y)$ is non-empty. Thus, as a consequence of Proposition~\ref{prop_compactness_1} we know, in fact, that $\J_{\rm fast}(\lambda_n,y) = \J_+(\xi)$.

Fix $J\in\J_+(\xi)$, and for every $n$ denote by $\jtil_n$ the almost complex structure determined by $(\lambda_n,J)$. Our arguments so far show that there is no loss of generality to assume the existence of finite-energy $\jtil_n$-holomorphic embedded fast planes
\begin{equation}
\util_n = (a_n,u_n) : \C \to \R \times L(p,q)
\end{equation}
asymptotic to $P$ such that $y\in u_n(\C)$. Up to reparametrization there is no loss of generality to assume that
\begin{equation}\label{normalization_u_n}
\begin{array}{ccc} a_n(0) = \min_\C a_n & & \int_{\C\setminus\D} u_n^*d\lambda_n = \sigma \end{array}
\end{equation}
where $\sigma>0$ satisfies
\begin{equation*}
2\sigma < T', \ \ \forall \ P'=(x',T') \in \P(\lambda) \cup \bigcup_n\P(\lambda_n).
\end{equation*}
The existence of $\sigma$ follows from standard arguments.

Note that $\jtil_n \to \jtil$ in $C^\infty$ (weak or, equivalently in this case, strong) as $n\to\infty$, where $\jtil$ is determined by $(\lambda,J)$. Up to translating in the $\R$-direction and selecting a subsequence, we may assume in addition that we find a finite set $\Gamma \subset \D$ and a finite-energy smooth $\jtil$-holomorphic map
\begin{equation*}
\util = (a,u) : \C\setminus \Gamma \to \R\times L(p,q)
\end{equation*}
such that $$ \util_n \to \util \ \ \ \text{in} \ \ \ C^\infty_{\rm loc}(\C\setminus\Gamma). $$ Also, up to extraction of a further subsequence, we may assume in addition that all points in $\Gamma$ are negative punctures of $\util$. We claim that $\util$ is non-constant. This is obvious if $\Gamma \neq \emptyset$, and if $\Gamma=\emptyset$ this follows from $0 < T - \sigma = \int_\D u^*d\lambda$. Identifying $\C\setminus\Gamma$ with a punctured Riemann sphere, where $\infty$ is a puncture, it follows that $\infty$ must be a positive puncture.

Our next goal is to show that fast planes do not intersect the trivial cylinder over its asymptotic orbit. The argument for this important step is contained in~\cite{props2}, but the statement we need, which can be found in~\cite{HLS} in the non-degenerate case, is not explicitly published. The case $p=1$, however, can be found in~\cite{hryn}.

\begin{lemma}\label{lemma_intersection}
If $\vtil=(b,v)$ is a fast finite-energy plane asymptotic to $P$ pseudo-holomorphic with respect to $\jtil$, or to $\jtil_n$ for some $n$, then $v(\C) \cap K = \emptyset$.
\end{lemma}

\begin{proof}
Let $U\simeq \R/\Z\times B$ be a Martinet tube around $x(\R)$, set $V=L(p,q)\setminus K$ and denote by $D$ the closure of $\{{\rm pt}\}\times (1/2)B$. We orient $D$ by $d\lambda$, denote by $a$ the homology class of $\partial D$ in $H_1(U\cap V)$ and by $a_0$ the homology class of $\partial D$ in $H_1(V)$. Since $U\cap V$ is homotopy equivalent to a $2$-torus, $a$ is a generator of $H_1(U\cap V)$ and we may choose another generator $b$ in order to identify $H_1(U\cap V) \simeq \Z\times \Z$ via the isomorphism $sa+tb \simeq (s,t)$. Such a choice of $b$ can be identified with a choice of a $d\lambda$-symplectic trivialization of $(x_{T_{\rm min}})^*\xi$, up to homotopy. Hence, the homotopy classes of $d\lambda$-symplectic trivializations of $(x_T)^*\xi$ can be identified with the set of homology classes $\{ka+pb\}_{k\in\Z} \subset H_1(U\cap V)$.

Fix $R\gg1$ and orient $D_R=\{z\in\C: |z|\leq R\}$ by the standard complex orientation of $\C$. Under our assumption that $\vtil$ is fast there is an asymptotic formula for the approach of $\vtil$ to $P$ as described in \S~\ref{par_hol_curves}, even if the contact form is degenerate. This is Lemma~\ref{lem_asymp_formula_nondeg_punct} and has many consequences to be explored below.

Let $\alpha_0$ be the class $v_*[\partial D_R]$ in  $H_1(V)$, and $\alpha$ be the class $v_*[\partial D_R]$ in $H_1(U\cap V)$. By the asymptotic formula, the algebraic intersection number ${\rm int}(\vtil)$ of $\vtil$ with $Z := \R\times K$ is well-defined, and by positivity of intersections it is a non-negative integer. The proof will be finished if we can show that ${\rm int}(\vtil)=0$. A basic computation from~\cite{props2} shows that $\alpha_0 = {\rm int}(\vtil)a_0$. Moreover, we find an eigensection $e(t)$ of the asymptotic operator at $P$ associated to a negative eigenvector, such that $\alpha$ is the class $[t\in\R/\Z\mapsto \exp(\epsilon e(t))] \in H_1(U\cap V)$, where $\epsilon>0$ is small: $e(t)$ is the asymptotic eigenvector of $\vtil$ at the puncture $\infty$.

Recall the $p$-disk $u_0$ and let $s_0\in\Z$ be uniquely determined by $[u_0((1-\epsilon)e^{i2\pi t})] = s_0a+pb \in H_1(U\cap V)$. This class generates $\ker \iota_*$ where $\iota:U\cap V \hookrightarrow V$ is the inclusion. By the definition of self-linking number, a $d\lambda$-symplectic trivialization of $(x_T)^*\xi$ which extends to $(u_0)^*\xi$ is identified with the class $(s_0+p\ \sl(K))a+pb=(s_0-1)a+pb$, according to the convention above. Thus, using the definition of $\wind_\infty$, we compute
\begin{equation*}
\begin{aligned}
{\rm int}(\util)a_0 & =\alpha_0=\iota_*\alpha=\iota_*[\exp(\epsilon e(t))] \\
&= \iota_*[(s_0-1+\wind_\infty(\vtil))a+pb] = \iota_*(s_0a+pb) = 0
\end{aligned}
\end{equation*}
thus proving that ${\rm int}(\vtil)=0$, as desired. Note that $\wind_\infty(\vtil)=1$ since $\vtil$ is fast.
\end{proof}

\begin{lemma}\label{lemma_global_linked}
If every $P'=(x',T') \in \P_*$ satisfies~\eqref{linked_conv} then every closed $\lambda$-Reeb orbit contained in $L(p,q)\setminus K$ is not contractible in $L(p,q)\setminus K$.
\end{lemma}

\begin{proof}
Let $P_0=(x_0,T_0)\in\P(\lambda)$ be contained in $L(p,q)\setminus K$ and denote $K_0=x_0(\R)$. It is a standard procedure to find positive smooth functions $g_n:L(p,q) \to \R$ such that $g_n|_{K\cup K_0} \equiv 1$, $dg_n|_{K\cup K_0} \equiv 0$, $g_n$ is identically equal to $1$ on some small neighborhood of the unique singular point $e$ of the $p$-disk $u_0$, $g_n \to 1$ in $C^\infty$, and each $g_n\lambda$ is a non-degenerate contact form. Thus, in view of Proposition~\ref{prop_global_sections}, for $n$ large enough we find a rational open book decomposition with disk-like pages of order $p$ and binding $K$ such that each page is a global surface of section for the Reeb flow of $g_n\lambda$. In particular, the intersection number of $P_0$ with a page is strictly positive because $P_0$ is a closed orbit of the Reeb flow of $g_n\lambda$. This implies that $P_0$ is not contractible in the complement of $K$, as desired.
\end{proof}

\begin{lemma}\label{lemma_asymp_limit_u}
The sequence $[t\in\R/\Z \mapsto u(Re^{i2\pi t})]$ converges to $[t\mapsto x(Tt)]$ when $R\to+\infty$ in the space $C^\infty(\R/\Z,L(p,q)) \mod \ \R/\Z$.
\end{lemma}

\begin{proof}
Consider any sequence $R_k \to +\infty$. By the fundamental results of Hofer~\cite{93} we can find a subsequence $R_{k_j}$, and some $P_0=(x_0,T_0) \in \P(\lambda)$ such that $u(R_{k_j}e^{i2\pi t})$ converges to $x_0(T_0t)$ in $C^\infty(\R/\Z,L(p,q))$ as $j\to\infty$.

If $P_0$ is geometrically distinct from $P$ then choose $j$ and $n$ large enough so that $u(R_{k_j}e^{i2\pi t})$ can be homotoped to $x_0(T_0t)$ in $L(p,q)\setminus K$, and $u_n(R_{k_j}e^{i2\pi t})$ can be homotoped to $u(R_{k_j}e^{i2\pi t})$ in $L(p,q)\setminus K$. Thus there is a homotopy from $u_n(R_{k_j}e^{i2\pi t})$ to $x_0(T_0t)$ in the complement of $K$. By Lemma~\ref{lemma_intersection} we know that $u_n(\C)\cap K=\emptyset$. Hence $x_0(T_0t)$ is contractible in $L(p,q)\setminus K$, contradicting Lemma~\ref{lemma_global_linked}. We have proved that $P_0 = (x,mT_{\rm min})$ for some $m\geq 1$.

We must have $m\leq p$ since $mT_{\rm min} = E(\util) \leq \limsup_n E(\util_n)= T = pT_{\rm min}$. Arguing as above one proves that $t\mapsto x(mT_{\rm min}t)$ is contractible in $L(p,q)$. To exclude the possibility of $m<p$ simply note that the class $\alpha$ of the loop $t\mapsto x(T_{\rm min}t)$ generates $H_1(L(p,q)) = \Z/p\Z$, so that $m\alpha \neq 0$ if $m<p$. Thus $m=p$ and the proof is complete.
\end{proof}

\begin{lemma}\label{lemma_positive_contact_area}
The section $\pi_\lambda \circ du$ does not vanish identically.
\end{lemma}

\begin{proof}
Consider the trivial cylinder $\vtil : \C\setminus0 \to \R\times L(p,q)$ given by $\vtil(e^{2\pi(s+it)}) = (T_{\rm min}s,x(T_{\rm min}t))$. Our proof is indirect. If $\pi_\lambda \circ du$ vanishes identically then, in view of the above lemma and of results from~\cite{props2}, we find a non-constant complex polynomial $P(z)$ of degree $p$ satisfying $P^{-1}(0)=\Gamma$ and $\util=\vtil \circ P$. Thus $\util$ has definite asymptotic limits at any given negative puncture, which are contractible iterates of $(x,T_{\rm min})$. Thus, if $\#\Gamma\geq 2$ then the asymptotic limit at some puncture in $\Gamma$ is $(x,mT_{\rm min})$ for some $m<p$. This is a contradiction since $H_1(L(p,q)) = \Z/p\Z$ is generated by the class $[t\mapsto x(T_{\rm min}t)]$. Thus $\Gamma=\{z_0\}$ and $P(z)=(z-z_0)^p$ for some $z_0\in\C$. We must have $z_0=0$ in view of the normalization conditions~\eqref{normalization_u_n}. Computing we obtain $$ T = \int_{\partial \D} u^*\lambda = \lim_n \int_{\partial \D} u_n^*\lambda = \lim_n \int_{\D} u_n^*d\lambda_n = T - \sigma $$ which is again contradiction.
\end{proof}

\begin{lemma}\label{lemma_no_bubb_points}
The set $\Gamma$ is empty, i.e., $\util$ is a finite-energy plane.
\end{lemma}

\begin{proof}
We argue by contradiction. Suppose that $\Gamma \neq \emptyset$ and choose $z_0\in\Gamma$. Then, by the fundamental results of Hofer~\cite{93}, we find $r_k\to0^+$ and a closed $\lambda$-Reeb orbit $P'=(x',T')$ such that the sequence of loops $t\mapsto u(z_0+r_ke^{i2\pi t})$ converges to $t\mapsto x'(T't+d)$ in $C^\infty$ as $k\to\infty$. If $P'$ is contained in $L(p,q)\setminus K$, then choosing $n$ and $k$ large enough we can find a homotopy from $u_n(z_0+r_ke^{i2\pi t})$ to $x'(T't+d)$ inside $L(p,q)\setminus K$. But since $u_n(\C)\cap K=\emptyset$ for every $n$, we conclude that $P'$ is contractible in $L(p,q)\setminus K$, contradicting Lemma~\ref{lemma_global_linked}. Thus $P'=(x,mT_{\rm min})$ for some $m>1$. However, since by Lemma~\ref{lemma_positive_contact_area} we know that $\int_{\C\setminus\Gamma}u^*d\lambda>0$, we must have $m<p$. But this implies that $t\mapsto x(mT_{\rm min}t)$ is contractible in $L(p,q)$ with $1\leq m<p$, which is again a contradiction because $H_1(L(p,q))=\Z/p\Z$ is generated by $t\mapsto x(T_{\rm min}t)$.
\end{proof}

\begin{lemma}\label{lemma_wind_infty_u}
The section $\pi_\lambda \circ du$ has no zeros.
\end{lemma}

\begin{proof}
Set $v(s,t)=u(e^{2\pi(s+it)})$, $v_n(s,t)=u_n(e^{2\pi(s+it)})$, where $(s,t)\in\R\times \R/\Z$. By Lemma~\ref{lemma_positive_contact_area} and the similarity principle we know that the set of points where $\pi_\lambda\circ du$ vanishes is discrete. Thus we can find $s_k\to+\infty$ such that $\pi_\lambda\circ du(z)\neq0$ for all $z$ satisfying $|z|=R_k := e^{2\pi s_k}$. Since $\util_n\to\util$ in $C^\infty_{\rm loc}$, we can find $n_k\to\infty$, $C^\infty$-small homotopies $h_k:[0,1]\times \R/\Z \to L(p,q)$ satisfying $h_k(0,t)=v_{n_k}(s_k,t)$, $h_k(1,t)=v(s_k,t)$, and nowhere vanishing sections $\sigma_k$ of $(h_k)^*\xi$ satisfying
\[
\begin{array}{ccc}
\sigma_k(0,t) = \pi_{\lambda_{n_k}} \cdot \partial_sv_{n_k}(s_k,t) & \text{and} & \sigma_k(1,t) = \pi_{\lambda} \cdot \partial_sv(s_k,t).
\end{array}
\]
Let $\zeta_k$ be a nowhere vanishing section of $(h_k)^*\xi$ such that its restriction to $0\times\R/\Z$ can be extended as a nowhere vanishing section of $(u_{n_k}|_{B_{R_k}})^*\xi$. Since $\wind_\pi(\util_n)=0$ for all $n$ we have that $\wind(t\mapsto \sigma_k(0,t),t\mapsto \zeta_k(0,t))=1$. By invariance of winding numbers under homotopies we get $\wind(t\mapsto \sigma_k(1,t),t\mapsto \zeta_k(1,t))=1$. This implies that the algebraic count of zeros of $\pi_\lambda\circ du$ on $B_{R_k}(0)$ is $1-1=0$; here we strongly used that $H_2(L(p,q))=0$. Since each zero of $\pi_\lambda\circ du$ counts positively to this algebraic count, we conclude that $\pi_{\lambda}\circ du$ does not vanish on $B_{R_k}$, for every~$k$. Thus $\pi_{\lambda}\circ du$ is nowhere vanishing.
\end{proof}

\begin{lemma}\label{lemma_uniform_1}
Let $G$ be any $\R/\Z$-invariant neighborhood of $t\in\R/\Z\mapsto x(Tt)$ in $C^\infty(\R/\Z,L(p,q))$. Then there exists $R_0>0$ such that $t\mapsto u_n(Re^{i2\pi t})$ is a loop in $G$ for every $R\geq R_0$ and $n$.
\end{lemma}

\begin{proof}
The argument is the same as that of the proof of Lemma~8.1 from~\cite{convex}. If the lemma is not true, then choose $R_n\to+\infty$ such that $t\mapsto u_n(R_ne^{i2\pi t})$ is not a loop in $G$. Define cylinders $\vtil_n$ by $$ \vtil_n(s,t) = (a_n(R_ne^{2\pi(s+it)})-c_n,u_n(R_ne^{2\pi(s+it)})) $$ where $c_n = a_n(R_n)$. We claim that $d\vtil_n$ is $C^0_{\rm loc}$-bounded. If not, using Hofer's lemma and standard rescaling arguments, we find a non-constant finite-energy $\jtil$-holomorphic plane that bubbles-off from the sequence $\vtil_n$ at some point $(s_*,t_*)$. In view of~\eqref{normalization_u_n} the $d\lambda$-area of this plane is $\leq \sigma$, which is impossible since by results from~\cite{93} the $d\lambda$-area of this plane is the period of some closed $\lambda$-Reeb orbit. This establishes the desired $C^0_{\rm loc}$ bounds for $d\vtil_n$. Hence, since $\vtil_n(0,0) \in 0\times L(p,q)$, we get $C^1_{\rm loc}$-bounds for the sequence $\vtil_n$. Elliptic estimates allow us to choose a subsequence, still denoted $\vtil_n$, that converges in $C^\infty_{\rm loc}(\R\times\R/\Z)$ to some $\jtil$-holomorphic finite-energy map $\vtil=(b,v)$. The map $\vtil$ is non-constant since $$ \int_{\{s=0\}} v^*\lambda = \lim_n \int_{\{s=0\}} v_n^*\lambda_n = \lim_n \int_{\{|z|=R_n\}} u_n^*\lambda_n = \lim_n \int_{\{|z|\leq R_n\}} u_n^*d\lambda_n \geq T-\sigma. $$ We want to check that $\pi_\lambda\circ dv$ vanishes identically. In fact, suppose not. Then we find $\delta>0$ and $\rho_n\to+\infty$ such that $\int_{\C\setminus B_{\rho_n}} u_n^*d\lambda_n \geq \delta$. Thus, $\int_{B_{\rho_n}} u_n^*d\lambda_n \leq T-\delta$ for all $n$ large. Since $\util_n \to \util$ in $C^\infty_{\rm loc}$, we get $\int_\C u^*d\lambda \leq T-\delta$ from Fatou's Lemma, a contradiction to a combination of lemmas~\ref{lemma_asymp_limit_u} and~\ref{lemma_no_bubb_points}. As a consequence, $\vtil$ is a trivial cylinder over some closed $\lambda$-Reeb orbit $P'=(x',T')$.

Up to a subsequence, the loops $t\mapsto u_n(R_ne^{i2\pi t})$ converge to $t\mapsto x'(T't+d)$ for some $d\in\R$, as $n\to\infty$. Note, however, that $P'$ lies in $L(p,q)\setminus K$. In fact, if not then $P'$ is the $m$-cover of $(x,T_{\rm min})$ for some $m\geq1$. Since $E(\util_n)=pT_{\rm min}$ we get $m\leq p$. However, $P'$ is obviously contractible, and consequently we get $m=p$, $P'=P$ contradicting our choice of $R_n$. Thus $P'\subset L(p,q)\setminus K$, and this is in  contradiction to Lemma~\ref{lemma_intersection} because this loop can not be contractible in the complement of $K$ (Lemma~\ref{lemma_global_linked}). The proof is complete.
\end{proof}

From now on we fix a Martinet Tube $(U,\Psi)$ for $P$, and denote by $(\theta,x_1,x_2)$ the coordinates on $\R/\Z\times B = \Psi(U)$ with respect to which $\lambda = g(d\theta+x_1dx_2)$, where $g(\theta,x_1,x_2)$ satisfies a) and b) in Definition~\ref{defn_martinet}. We denote $z=(x_1,x_2)$. Choose an $\R/\Z$-invariant neighborhood $G$ of $t\mapsto x(Tt)$ so small that loops in $G$ are contained in $U$. As a consequence of Lemma~\ref{lemma_uniform_1} we can find a uniform number $s_0>1$ such that
\begin{equation}\label{number_s_0}
s\geq s_0 \Rightarrow [t\mapsto u_n(e^{2\pi(s+it)})] \in G, \ \forall n.
\end{equation}
In particular, the functions $a_n,\theta_n,x_{1,n},x_{2,n}$ given by
\begin{equation*}
(a_n(s,t),\theta_n(s,t),x_{1,n}(s,t),x_{2,n}(s,t)) = (id_\R\times \Psi) \circ \util_n (e^{2\pi(s+it)})
\end{equation*}
are well-defined on $[s_0,+\infty) \times \R/\Z$. We may lift the $\theta$-coordinate to the universal covering $\R$ and, when convenient, think of $\theta_n$ as a real-valued function of $(s,t) \in [s_0,+\infty) \times \R$. There is a unique such lift, again denoted by $\theta_n(s,t)$ with no fear of ambiguity, satisfying $\lim_{s\to+\infty} \theta_n(s,0) \in [0,1)$. It satisfies $\theta_n(s,t+1)=\theta_n(s,t)+p$. The other functions $a_n,x_{1,n},x_{2,n}$ are also to be thought of as functions of $(s,t)\in[s_0,+\infty) \times \R$ which are $1$-periodic in $t$. We denote $z_n(s,t)=(x_{1,n}(s,t),x_{2,n}(s,t))$.

In these coordinates the contact forms $\lambda_n$ are written as $g_n(d\theta+x_1dx_2)$ where $g_n=f_ng$. By the properties of the $f_n$ we know that $g_n\to g$ in $C^\infty_{\rm loc}$, $g_n|_{\R/\Z\times 0}\equiv T_{\rm min}$ and $dg_n|_{\R/\Z\times 0}\equiv 0$ for all $n$. The Reeb vector field $X_{\lambda_n}$ of $\lambda_n$ has coordinates $$ X_{\lambda_n} = (R_n,Y_n) \in \R\times \R^2 $$ and clearly $Y_n(\theta,0,0) = 0$. Let $j:\R/\Z\times B \to \R^{2\times 2}$ be the smooth $2\times 2$ matrix-valued function representing $J:\xi\to\xi$ in the frame $e_1=\partial_{x_1}$, $e_2=-x_1\partial_\theta+\partial_{x_2}$. The Cauchy-Riemann equations can be written as the system
\begin{align}
& \partial_sa_n - g_n(\theta_n,z_n)(\partial_t\theta_n+x_{1,n}\partial_tx_{2,n}) = 0 \label{cr_a_coord_1} \\
& \partial_ta_n + g_n(\theta_n,z_n)(\partial_s\theta_n+x_{1,n}\partial_sx_{2,n}) = 0 \label{cr_a_coord_2} \\
& \partial_sz_n + j_n\partial_tz_n + S_nz_n = 0 \label{cr_z_coordinate}
\end{align}
where $$ j_n(s,t) = j(\theta_n(s,t),z_n(s,t)) $$ and
\begin{equation*}
S_n(s,t) = \left[ \partial_ta_n \ I-\partial_sa_n \ j_n \right] D_n(s,t)
\end{equation*}
with
\begin{equation*}
D_n(s,t) = \int_0^1 Y_n'(\theta_n(s,t),\tau z_n(s,t)) d\tau
\end{equation*}
where the prime denotes a derivative in the $z$-coordinate. These functions are $1$-periodic in $t$ and may be seen as defined in $\R\times\R/\Z$.

\begin{lemma}\label{lemma_unif_bounds_near_end}
The following holds:
\begin{align}
& \lim_{s\to+\infty} \sup_{n,t} \left( |D^\beta x_{1,n}(s,t)| + |D^\beta x_{2,n}(s,t)|\right) = 0 \ \ \forall \beta \label{unif_bound_1} \\
& \lim_{s\to+\infty} \sup_{n,t} \left( |D^\beta [a_n-Ts](s,t)| + |D^\beta [\theta_n-pt](s,t)| \right) = 0 \ \ \text{if} \ |\beta|\geq1. \label{unif_bound_2}
\end{align}
\end{lemma}

\begin{proof}
This proof is contained in~\cite{convex}. There is no loss of generality to assume $T_{\rm min}=1$, so that $T=p$. Note that Lemma~\ref{lemma_uniform_1} gives~\eqref{unif_bound_1} for every $\beta=(0,\beta_2)$, $\beta_2\geq0$. Set $\vtil_n(s,t) = \util_n(e^{2\pi(s+it)})$.

We claim that
\begin{equation}\label{uniform_gradient_bounds_end}
\sup_{s\geq s_0,n,t} |\nabla\vtil_n| < \infty.
\end{equation}
Suppose, by contradiction, that~\eqref{uniform_gradient_bounds_end} does not hold. Then, since $\{\vtil_n\}$ is $C^\infty_{\rm loc}$-convergent, we find $(s_n,t_n)$ satisfying $s_n\to+\infty$ such that, up to a subsequence, $|\nabla\vtil_n(s_n,t_n)|\to\infty$. Using Hofer's lemma and rescaling, exactly as in Lemma~\ref{lemma_uniform_1}, we find a non-constant finite-energy $\jtil$-holomorphic plane that bubbles-off from this sequence. The $d\lambda$-area of this plane is at least equal to some period and, as such, strictly larger than $\sigma$, contradicting~\eqref{normalization_u_n}. This proves~\eqref{uniform_gradient_bounds_end}. Differentiating the Cauchy-Riemann equations for $\vtil_n$ and using a boot-strapping argument one now concludes that
\begin{equation}\label{uniform_derivatives_bounds_end}
\sup_{s\geq s_0,t,n} |D^\alpha\vtil_n| < \infty \ \ \ \forall \alpha.
\end{equation}
The arguments just explained, which we used to arrive at~\eqref{uniform_derivatives_bounds_end}, can be found in~\cite[Lemma 7.1]{convex}.

As a consequence $S_n(s,t)$ is uniformly bounded in $n$; this fact together with~\eqref{cr_z_coordinate} implies that~\eqref{unif_bound_1} holds for $|\beta|\leq 1$.

Let $\beta$ satisfy $|\beta|\geq1$ and suppose, by contradiction, that
\begin{equation}
\limsup_{s\to\infty} \left[ \sup_{n,t} \left( |D^\beta [a_n-Ts]| + |D^\beta [\theta_n-pt]| \right) \right] > 0.
\end{equation}
Then, up to choice of a subsequence, we may assume that we can find $(s_n,t_n)$, $s_n\to\infty$, such that
\begin{equation}\label{partial_contradiction}
|D^\beta [a_n-Ts](s_n,t_n)| + |D^\beta [\theta_n-pt](s_n,t_n)| \geq \epsilon > 0.
\end{equation}
Consider
\begin{equation*}
h_n(s+it) = a_n(s+s_n,t+t_n)-a_n(s_n,t_n) + i(\theta_n(s+s_n,t+t_n)-\theta_n(s_n,t_n)).
\end{equation*}
By~\eqref{uniform_derivatives_bounds_end}, equations~\eqref{cr_a_coord_1}-\eqref{cr_a_coord_2} and by~\eqref{unif_bound_1} for $|\beta|\leq1$, $h_n$ $C^\infty_{\rm loc}$-converges to a holomorphic function $h(s+it)$ with bounded gradient, up to the choice of a subsequence. The limit $h$ can not be constant, in fact, by Lemma~\ref{lemma_uniform_1} the $t$-derivative of its imaginary part is equal to $p$. Since $h(s+i(t+1))=h(s+it)+ip$, this function must have the form $h(s+it) = Ts+c + i(pt+d)$ (Liouville's theorem was used). This contradicts~\eqref{partial_contradiction} and proves~\eqref{unif_bound_2} for $|\beta|\geq1$.

The proof can now be completed with an induction argument, differentiating equation~\eqref{cr_z_coordinate}, using the elliptic estimates for the standard $\bar\partial$-operator and~\eqref{unif_bound_2}.
\end{proof}

Up to composing each $\util_n$ with a rotation we can assume, without loss of generality, that
\begin{equation}\label{1_limit_theta_n}
\lim_{s\to+\infty} \theta_n(s,0) = 0, \ \ \ \ \text{for each fixed $n$.}
\end{equation}
Here we used the asymptotic behavior described in Theorem~\ref{thm_partial_asymp} and the fact that each $\lambda_n$ is non-degenerate. Moreover, by Lemma~\ref{lemma_unif_bounds_near_end} we know that
\begin{equation}\label{2_limit_theta_n}
\lim_{j\to\infty} \|D^\beta[\theta_{n_j}(s,t)-pt-\theta_{n_j}(s_j,0)](s_j,\cdot)\|_{L^\infty(\R/\Z)} = 0
\end{equation}
holds for every pair of sequences $n_j,s_j\to+\infty$ and for every partial derivative $D^\beta = \partial_s^{\beta_1}\partial_t^{\beta_2}$.
Set
\begin{equation}\label{asymptotic_geometric_data}
\begin{aligned}
& j^\infty(t) = j(pt,0), \\
& D^\infty(t) = Y'(pt,0), \\
& S^\infty(t) = -Tj^\infty(t)D^\infty(t)
\end{aligned}
\end{equation}
and let $M:\R/\Z\times B\to \R^{2\times 2}$ be a smooth function taking values on symplectic matrices and satisfying $Mj=J_0M$, where $J_0$ is the standard complex structure on $\R^2\simeq \C$. Then define
\begin{equation*}
M_n(s,t) = M(\theta_n(s,t),z_n(s,t))
\end{equation*}
and
\begin{equation}\label{asymptotic_geometric_data_M}
M^\infty(t) = M(pt,0).
\end{equation}

\begin{lemma}\label{lemma_geom_data_unif_compactness}
For every pair of sequences $n_l,s_l\to+\infty$ there exists $l_k \to \infty$ and a number $c \in [0,1)$ such that
\begin{equation}
\begin{aligned}
& \lim_{k\to\infty} \|\partial_s^{\beta_1}\partial_t^{\beta_2}[j_{n_{l_k}}(s,t)-j^\infty(t+c)](s_{l_k},\cdot)\|_{L^\infty(\R/\Z)} = 0 \\
& \lim_{k\to\infty} \|\partial_s^{\beta_1}\partial_t^{\beta_2}[D_{n_{l_k}}(s,t)-D^\infty(t+c)](s_{l_k},\cdot)\|_{L^\infty(\R/\Z)} = 0 \\
& \lim_{k\to\infty} \|\partial_s^{\beta_1}\partial_t^{\beta_2}[S_{n_{l_k}}(s,t)-S^\infty(t+c)](s_{l_k},\cdot)\|_{L^\infty(\R/\Z)} = 0 \\
& \lim_{k\to\infty} \|\partial_s^{\beta_1}\partial_t^{\beta_2}[M_{n_{l_k}}(s,t)-M^\infty(t+c)](s_{l_k},\cdot)\|_{L^\infty(\R/\Z)} = 0
\end{aligned}
\end{equation}
for each partial derivative $\partial_s^{\beta_1}\partial_t^{\beta_2}$, $\beta_1,\beta_2\geq0$.
\end{lemma}

\begin{proof}
We only discuss the first limit, the second and fourth limits are handled analogously, and the third follows from the first two, Lemma~\ref{lemma_unif_bounds_near_end} and the formula $S_n(s,t) = \left[ \partial_ta_n \ I-\partial_sa_n \ j_n \right] D_n(s,t)$.

We find $c\in[0,1)$, $l_k \to \infty$ such that $\theta_{n_{l_k}}(s_{l_k},0) \to c \mod 1$ since $\R/\Z$ is compact. Denote $c_k := \theta_{n_{l_k}}(s_{l_k},0)$ and consider $$ \epsilon_k(s,t) = \theta_{n_{l_k}}(s,t)-pt-c_k. $$ Since $j_n=j\circ(\theta_n,z_n)$ and $j$ is $1$-periodic in the first coordinate, we get
\begin{equation*}
\begin{aligned}
j_{n_{l_k}}(s,t) &= j(pt+c_k,0) \\
&+ \left[ \int_0^1Dj(pt+c_k+\tau\epsilon_k(s,t),\tau z_{n_{l_k}}(s,t)) \ d\tau \right] \cdot (\epsilon_k(s,t),z_{n_{l_k}}(s,t)).
\end{aligned}
\end{equation*}
The desired result follows from the above representation and~\eqref{2_limit_theta_n} since each derivative of $j$ is bounded on a neighborhood of the orbit.
\end{proof}

Defining
\begin{equation*}
\zeta_n(s,t) = M_n(s,t)z_n(s,t)
\end{equation*}
and
\begin{align*}
\Lambda_n = (M_nS_n - \partial_sM_n-J_0\partial_tM_n) M_n^{-1}
\end{align*}
we get
\begin{equation}
\partial_s\zeta_n+J_0\partial_t\zeta_n+\Lambda_n\zeta_n=0.
\end{equation}
Setting
\begin{align*}
& \Lambda^\infty(t) = \left[ M^\infty(t)S^\infty(t)-J_0\partial_tM^\infty(t) \right] M^\infty(t)^{-1}
\end{align*}
we obtain the following corollary of Lemma~\ref{lemma_geom_data_unif_compactness}.

\begin{corollary}
For every pair of sequences $n_l,s_l\to+\infty$ there exists $l_k\to\infty$ and a number $c\in[0,1)$ such that
\begin{equation*}
\lim_{k\to\infty} \|\partial^{\beta_1}_s\partial^{\beta_2}_t[\Lambda_{n_{l_k}}(s,t)-\Lambda^\infty(t+c)](s_{l_k},\cdot)\|_{L^\infty(\R/\Z)} = 0 \ \ \ \forall \ \beta_1,\beta_2\geq0.
\end{equation*}
\end{corollary}

\begin{remark}[$C^{l,\alpha,\delta}_0$-topology]\label{rmk_Holder_with_decay}
Given numbers $l\geq1$, $\alpha\in(0,1)$ and $\delta<0$, we recall here the space $C^{l,\alpha,\delta}_0([s_0,+\infty)\times\R/\Z,\R^m)$ defined in~\cite{props3}. It is defined as the set of functions $h:[s_0,+\infty)\times\R/\Z\to \R^m$ which are locally of class $C^{l,\alpha}$ and satisfy $$ \lim_{R\to+\infty} \| e^{-\delta s}D^\beta h \|_{C^{0,\alpha}([R,+\infty)\times\R/\Z)} = 0 \ \ \ \forall |\beta|\leq l. $$ This space becomes a Banach space with the norm
\begin{equation}\label{norm_C_l_alpha_delta}
\|h\|_{l,\alpha,\delta} = \|e^{-\delta s}h(s,t)\|_{C^{l,\alpha}([s_0,+\infty)\times\R/\Z)}.
\end{equation}
See~\cite{props3} for more details.
\end{remark}

The asymptotic behavior of the plane $\util$ can only be studied in view of the following proposition which is proved in Appendix~\ref{app_unif_asymp_analysis} by a small modification of the arguments from~\cite[appendix B]{openbook}. The proof is also essentially contained in~\cite{props1}.

\begin{proposition}\label{prop_unif_asymptotic_analysis}
Suppose that $K_n : [0,+\infty)\times \R/\Z \rightarrow \R^{2m\times 2m}$, $n\geq 1$, and $K^\infty : \R/\Z \rightarrow \R^{2m\times 2m}$ are given smooth maps satisfying:
\begin{enumerate}
\item[i)] $K^\infty(t)$ is symmetric $\forall t$.
\item[ii)] For every pair of sequences $n_l,s_l\to+\infty$ there exists $l_k\to+\infty$ and $c\in[0,1)$ such that
\[
\lim_{k\to\infty} \sup_{\tau\in\R/\Z} |D^\beta[K_{n_{l_k}}(s,t)-K^\infty(t+c)](s_{l_k},\tau)| = 0 \ \ \ \forall \beta.
\]
\end{enumerate}
Consider the unbounded self-adjoint operator $L$ on $L^2(\R/\Z,\R^{2m})$ defined by
\[
  Le = -J_0\dot{e} - K^\infty e.
\]
Denote $E = C^{l,\alpha,\delta}_0 \left( \left[ 0,+\infty \right) \times \R/\Z,\R^{2m} \right)$ for some $\delta < 0$ and $l\geq1$. Suppose that $\{X_n\} \subset E$ are smooth maps satisfying
\begin{equation}\label{eqn_X_n_prop}
\partial_s X_n + J_0\partial_t X_n + K_n X_n = 0  \ \forall n.
\end{equation}
If $\delta\not\in\sigma(L)$ and $X_n$ is $C^\infty_{loc}$-bounded then $\{X_n\}$ has a convergent subsequence in $E$.
\end{proposition}

Let $\kappa$ be a homotopy class of $d\lambda$-symplectic, or equivalently $d\lambda_n$-symplectic for any $n$, trivializations of $(x_T)^*\xi$ coming from a capping disk for $P$. Let $A_n$ be the asymptotic operator at the orbit $P$ associated to $(\lambda_n,J)$, and let $A$ be the corresponding operator associated to $(\lambda,J)$. By the fact that the Conley-Zehnder index of $P$ with respect to $\lambda$ is larger than or equal to $3$, and by the continuity of the spectrum, we find $\delta<0$ such that for $n$ large enough the eigenvalues of $A_n$ and $A$ with winding number $1$ with respect to $\kappa$ are smaller than $2\delta$. Moreover, we can assume that $\delta$ is not an eigenvalue of $A$ or of $A_n$ with $n$ large enough.

Choose $l\geq 1$, $\alpha\in(0,1)$ and consider the space $E=C^{l,\alpha,\delta}_0([s_0,+\infty)\times\R/\Z,\R^2)$ where $\delta$ is described as above. Since each plane $\util_n$ is fast, we can use Theorem~\ref{thm_asymptotics} to conclude that each $z_n$ belongs to $E$; this is so since the asymptotic eigenvalue of $\util_n$ is below $\delta$. Thus, also $\zeta_n\in E$ for $n$ large. Applying Proposition~\ref{prop_unif_asymptotic_analysis} to $\Lambda_n,\Lambda^\infty$ we conclude that $\zeta_n$ converges in $E$, up to selection of a subsequence. Thus, $z_n$ converges in $E$ as well. In particular, setting
\begin{equation}
(a(s,t),\theta(s,t),x_1(s,t),x_2(s,t)) = id_\R\times\Psi(\util(e^{2\pi(s+it)}))
\end{equation}
we obtain that $z(s,t)=(x_1(s,t),x_2(s,t)) \in E$ and that $z_n\to z$ in $E$, because $z_n\to z$ in $C^\infty_{\rm loc}$ and $z_n$ is convergent in $E$. Since $l\geq1$ was arbitrary, we conclude $\exists r>0$ such that
\begin{equation}\label{decay_z_n}
\lim_{s\to+\infty} \sup_{t\in\R/\Z} e^{rs} |D^\beta z(s,t)| = 0 \ \ \ \forall \beta.
\end{equation}

We know from Theorem~\ref{thm_asymptotics} that $\forall n$ $\exists c_n$ such that $a_n(s,t)-Ts-c_n \to 0$ as $s\to+\infty$. Consider
\begin{equation}
\Delta_n(s,t) = \begin{pmatrix} a_n(s,t)-Ts-c_n \\ \theta_n(s,t)-pt \end{pmatrix}.
\end{equation}
Equations~\eqref{cr_a_coord_1}-\eqref{cr_a_coord_2} can be rewritten as
\begin{equation}\label{cr_Delta_n}
\partial_s\Delta_n + \begin{pmatrix} 0 & -T_{\rm min} \\ T_{\rm min}^{-1} & 0 \end{pmatrix} \partial_t\Delta_n + B_nz_n = 0
\end{equation}
for suitable matrix-valued functions $B_n(s,t)$ satisfying
\begin{equation*}
\limsup_{s\to+\infty} \sup_{t,n} |D^\beta B_n(s,t)| < \infty \ \ \ \forall \beta.
\end{equation*}
Using this, \eqref{cr_Delta_n},~\eqref{decay_z_n}, the convergence of $z_n$ in $E$ and Lemma~\ref{lemma_unif_bounds_near_end} we see that we can apply a version of~\cite[Lemma 6.3]{hryn} for families to conclude that
\begin{equation}\label{decay_Delta_n}
\lim_{s\to+\infty} \sup_{n,t} e^{rs} |D^\beta \Delta_n(s,t)| = 0 \ \ \ \forall \beta \ \text{such that} \ |\beta|\geq1.
\end{equation}
This uniform exponential decay of the derivatives of $\Delta_n$ of order $\geq 1$ together with the fact that $\Delta_n(s,t) \to 0$ as $s\to+\infty$  will tell us that~\eqref{decay_Delta_n} also holds for $\beta=(0,0)$.
Together \eqref{decay_z_n}~\and~\eqref{decay_Delta_n} guarantee that $\infty$ is a non-degenerate puncture of $\util$.

So far we proved that $\util$ is a finite-energy plane with a non-degenerate puncture, and that $\util$ is asymptotic to $P$. It is an immersion since it satisfies $\wind_\pi(\util)=0$, this is a consequence of Lemma~\ref{lemma_wind_infty_u}. Since $y\not\in K$ and $\util_n\to\util$ in $C^\infty_{\rm loc}$ we can use Lemma~\ref{lemma_uniform_1} to conclude that $y\in u(\C)$.

\begin{remark}\label{rmk_fast_embedded}
The fast plane $\util$ must be an embedding. In fact, even $u$ is an immersion since $\pi_\lambda \circ du$ does not vanish. Consequently it must also be somewhere injective because, otherwise using results from~\cite{props2}, we would find a complex polynomial $p(z)$ of degree at least $2$ and a somewhere injective plane $\vtil$ such that $\util=\vtil\circ p$; this forces $\util$ to have critical points, which is impossible. Now, using the proof of~\cite[Theorem~2.3]{props2} combined with Lemma~\ref{lemma_intersection} above we conclude that $u$ is injective.
\end{remark}

Combining the previous arguments with the above remark, we have proved

\begin{proposition}\label{prop_existence_fast_planes}
For every $y\in L(p,q)\setminus K$, and every $J\in \J_+(\xi)$ there exists an embedded fast finite-energy $\jtil$-holomorphic plane $\util=(a,u)$ asymptotic to $P$ and satisfying $y\in u(\C)$, where $\jtil$ is the $\R$-invariant almost complex structure induced by the data $(\lambda,J)$.
\end{proposition}

With the above existence result proved, we need two additional ingredients in order to foliate $L(p,q)\setminus K$ by projections of fast finite-energy planes asymptotic to~$P$: Fredholm theory and a piece of intersection theory from~\cite{props2}.

\begin{theorem}\label{thm_two_fast_planes}
If $\util=(a,u)$ is a fast finite-energy plane asymptotic to $P$, then $u:\C\to L(p,q)\setminus K$ is an embedding. If $\vtil=(b,v)$ is another such plane then either $u(\C)=v(\C)$ or $u(\C)\cap v(\C) = \emptyset$. In the former case, there exist $A,B\in\C$, $A\neq0$, and $c\in\R$ such that $a(z)=b(Az+B)+c$ and $u(z)=v(Az+B)$ for all $z\in\C$.
\end{theorem}

The proof of the above theorem is due to Hofer, Wysocki and Zehnder, but its statement is not found in the literature. We briefly explain why is this theorem true. The first statement was already explained in Remark~\ref{rmk_fast_embedded}. In fact, the map $u$ is an immersion into $L(p,q)\setminus K$ in view of Lemma~\ref{lemma_intersection} and of the identity $\wind_\pi(\util)=0$, in particular, $\util$ is somewhere injective, see Remark~\ref{rmk_fast_embedded}. The claim that $u$ is injective follows from the proof of~\cite[Theorem~2.3]{props2} since fast planes have an asymptotic formula and are somewhere injective\footnote{To see this, note that if a fast plane is not somewhere injective then it can be factored through a somewhere injective plane and, consequently, must have a critical point. This, however, is impossible since the identity $\wind_\pi=0$ implies that a fast plane is an immersion.}. The dichotomy $u(\C)=v(\C)$ or $u(\C)\cap v(\C)=\emptyset$ follows from Lemma~\ref{lemma_intersection} and the proof of \cite[Theorem~4.11]{props2}. The constants $A,B,c$ can be found using the similarity principle.

\begin{theorem}\label{thm_fredholm_theory_fast_planes}
Let $\alpha_0$ be any contact form defining $(L(p,q),\xi_{\rm std})$, $\util_0$ be an embedded fast $\jtil$-holomorphic plane in $\R\times L(p,q)$ asymptotic to some closed $\alpha_0$-Reeb orbit~$P_0$, not necessarily prime. If $\mu_{CZ}(\util_0)\geq 3$ then there exists an embedding $f:B\times \C\to \R\times L(p,q)$, where $B\subset \R^2$ is an open ball centered at the origin, with the following properties:
\begin{itemize}
\item[i)] $\util_0(\cdot)=f(0,\cdot)$.
\item[ii)] For every $\tau\in B$ the map $z\mapsto f(\tau,z)$ is a fast finite-energy $\jtil$-holomorphic plane asymptotic to $P$.
\item[iii)] If $\{\util_k\}$ is a sequence fast finite-energy $\jtil$-holomorphic planes asymptotic to $P$ satisfying $\util_k\to\util$ in $C^\infty_{\rm loc}$, then $\exists A_k,B_k \in \C$, $\tau_k\in B$ such that $A_k\to1$, $B_k\to0$, $\tau_k\to0$ and $f(\tau_k,A_kz+B_k) = \util_k(z) \ \forall z\in\C$.
\end{itemize}
\end{theorem}

In the case $p=1$ this follows from~\cite[Theorem~2.3]{hryn}. In the non-degenerate case its proof is sketched in the appendix of~\cite{HLS}; we only note here that the same arguments can be used to prove Theorem~\ref{thm_fredholm_theory_fast_planes} since fast finite-energy planes have a nice asymptotic formula which allows for a Fredholm theory with a suitably placed exponential weight.

By Proposition~\ref{prop_existence_fast_planes} we find for every $y\not\in K$ an embedded fast finite-energy plane $\util=(a,u)$ asymptotic to $P$ such that $y\in u(\C)$. Using theorems~\ref{thm_two_fast_planes} and~\ref{thm_fredholm_theory_fast_planes} one can show that the various projections $u(\C)$ of such planes are the leafs of a smooth foliation of $L(p,q)\setminus K$, each leaf being transverse to the Reeb vector field since these planes satisfy $\wind_\pi=0$. Using the uniform asymptotic approach of these planes to its asymptotic limit $P$, one concludes that this foliation is a rational open book decomposition with binding $K$. See~\cite{HLS} for details.

The proof of Theorem~\ref{thm_22} will be completed if we can show that each page is a disk-like global surface of section for the $\lambda$-Reeb flow. Let $y\not\in K$ be fixed arbitrarily and assume that its omega-limit set $\omega(y)$ with respect to the $\lambda$-Reeb flow does not intersect $K$. Then, by compactness of the leaf-space and transversality of the Reeb vector field to the leafs, we conclude that the trajectory through $y$ will hit every page in the future infinitely many times. If $\omega(y) \cap K \neq\emptyset$ then the trajectory through $y$ will spend arbitrarily large amounts of time in the far future very close to $K$ and, consequently, can be well-approximated by the linearized $\lambda$-Reeb flow along $K$; this will force it to hit every leaf infinitely often in the future because the inequality $\mu_{CZ}(P)\geq 3$ implies that the linearized flow along $p$-iterates of $K$ will wind more than $2\pi$ (one full turn) with respect to a page. The reasoning for past hitting times is similar, replacing the omega-limit set by the alpha-limit set of $y$. Theorem~\ref{thm_22} is proved.

\appendix

\section{Proof of Proposition~\ref{prop_compactness_1}}

The proof, which consists of showing that $\J_{\rm fast}(\alpha,y)$ is open and closed in $\J$, is a generalization of the arguments contained in~\cite[appendix]{openbook}. We only include this proof here since we now make more general assumptions than in~\cite{openbook} and need to keep track of the point $y \in M\setminus x(\R)$.

\subsection{$\J_{\rm fast}(\alpha,y)$ is closed.}

Let $J\in\J$ and $J_n\in\J_{\rm fast}(\alpha,y)$ satisfy $J_n\to J$ in $C^\infty$. The data $(\alpha,J)$ and $(\alpha,J_n)$ induce almost complex structures $\jtil$ and $\jtil_n$ on $\R\times M$, respectively, as explained in \S~\ref{par_hol_curves}. Consider $\util_n=(a_n,u_n)$ embedded fast finite-energy $\jtil_n$-holomorphic planes asymptotic to $P$ satisfying $y\in u_n(\C)$. Arguing as in Lemma~\ref{lemma_intersection} we obtain
\begin{equation}\label{intersection_asymptotic_fast}
u_n(\C) \cap x(\R) = \emptyset \ \ \ \forall n.
\end{equation}

Let $\sigma>0$ be a number smaller than all positive periods of closed $\alpha$-Reeb orbits. After reparametrization and translation in the $\R$-direction, we may assume that
\begin{equation}
\begin{array}{cccc} \int_\D u_n^*d\alpha = T-\sigma, & u_n(0)=y & \text{and} & a_n(2)=0. \end{array}
\end{equation}
Up to a subsequence, we find a finite set $\Gamma\subset \D$ and a finite-energy $\jtil$-holomorphic map $\util=(a,u) : \C\setminus \Gamma \to \R\times M$ such that $\util_n \to \util$ in $C^\infty_{\rm loc}(\C\setminus \Gamma)$. There is no loss of generality to assume that $\Gamma$ consists of negative punctures of $\util$. Thus $\infty$ is a positive puncture. Using results on cylinder with small contact area it can be proved that $\util$ is non-constant and its asymptotic limit at $\infty$ is $P$. Using ``soft-rescalling'' at the negative punctures, it can be shown that the asymptotic limits of $\util$ at the punctures in $\Gamma$ are contractible closed $\alpha$-Reeb orbits.

We claim that
\begin{equation}\label{positive_contact_area_1}
\int_{\C\setminus\Gamma} u^*d\alpha > 0.
\end{equation}
We prove this indirectly. If~\eqref{positive_contact_area_1} is not true then we find a non-constant complex polynomial $Q$ such that $Q^{-1}(0)=\Gamma$ and $\util=F\circ Q$, where $F:\C\setminus\{0\}\to \R\times M$ is the map $F(e^{2\pi(s+it)})=(T_{\rm min}s,x(T_{\rm min}t))$. This implies, in particular, that $\Gamma\neq\emptyset$ and that the degree of $Q$ is precisely $p$. If $\#\Gamma>1$ then the asymptotic limit at some point of $\Gamma$ is $(x,mT_{\rm min})$ for some $1\leq m<p$ and, as observed above, this orbit must be contractible. But this is impossible: since $x(\R)$ is $p$-unknotted the loop $t\in\R/\Z\mapsto x(T_{\rm min}t)$ induces an element of $\pi_1(M)$ of order $p$ and, consequently, the loop $t\in\R/\Z\mapsto x(mT_{\rm min})$ can not be contractible because $m<p$. This proves that $\#\Gamma=1$. If $0\not\in\Gamma$ then $y=u(0)\in x(\R)$, absurd. Thus $\Gamma=\{0\}$ and we can estimate $T = \int_{\partial\D} u^*\lambda = \lim_n \int_{\D} u_n^*d\lambda = T-\sigma$, a contradiction.

Let us enumerate the negative punctures in $\Gamma$ as $z_1,\dots,z_N$, and let $P_i=(x_i,T_i)$ be the asymptotic limit of $\util$ at $z_i$. Choose a $d\alpha$-symplectic trivialization $B$ of $u^*\xi$ with the following property: it extends at every puncture $z_i$ (or $\infty$) to a $d\alpha$-symplectic trivialization of $({x_i}_{T_i})^*\xi$ (or of $(x_T)^*\xi$) coming from a capping disk. Here we used the assumption that $c_1(\xi)$ vanishes on $\pi_2(M)$.

It follows from~\cite[Lemma~4.9]{HLS} that $\mu_{CZ}(P_i)\geq 2$ for every $i$. We sketch the argument below. The reason for this is that there is a holomorphic building that arises as the limit of the planes $\util_n$, up to choice of a further subsequence. In our particular situation, the building has the simpler structure of a bubbling-off tree of finite-energy punctured spheres, see~\cite[Section~4]{HLS} for instance.

The {\it fundamental mechanism} is the following: if $\vtil$ is a punctured finite-energy sphere in the tree with precisely one positive puncture and if it is asymptotic to a closed Reeb orbit $P_+$ at its positive puncture satisfying $\mu_{CZ}(P_+)\leq 1$, then there exists at least one negative puncture where the corresponding asymptotic limit satisfies $\mu_{CZ}\leq 1$. The proof of this last claim is as follows. If, by contradiction, $\mu_{CZ}\geq 2$ for every asymptotic limit at a negative puncture of $\vtil$ then we analyze two cases. In case the contact area of $\vtil$ does not vanish we get $\wind_\pi(\vtil)<0$, absurd. In case the contact area of $\vtil$ vanishes, there are negative punctures and all asymptotic limits (including $P_+$) are iterates of a common prime orbit $P_0$. Let $m\geq 1$ be such that $P_+=P_0^m$, fix a negative puncture $z_i$ and let $j$ be such that the asymptotic limit at $z_i$ is $P_i=P_0^j$. Note that $m\geq j$, $P_0^m$ and $P_0^j$ are contractible, and if $N$ is the least common multiple of $m$ and $j$ then trivializations of $\xi$ along $P_0^m,P_0^j$ coming from capping disks iterate (by $N/m,N/j$ respectively) to trivializations of $\xi$ along $P_0^N$ which are homotopic since $c_1(\xi)$ vanishes on $\pi_2(M)$. Hence
\begin{align*}
& 2\frac{N}{j} \leq \mu_{CZ}((P_0^j)^{N/j}) = \mu_{CZ}(P_0^N) = \mu_{CZ}((P_0^m)^{N/m}) \leq 2\frac{N}{m}-1 \leq 2\frac{N}{j} - 1
\end{align*}
This contradiction concludes the proof of the {\it fundamental mechanism}.

Now, using this {\it fundamental mechanism} explained above, we will argue to conclude that $\mu_{CZ}(P_i)\geq2$ for every $i$. If not, there is $i_0$ such that $\mu_{CZ}(P_{i_0})\leq 1$. The orbit $P_{i_0}$ is the asymptotic limit at the unique positive puncture of a finite-energy punctured sphere in the building. By the fundamental mechanism, this punctured sphere must have a negative puncture where it is asymptotic to an orbit satisfying $\mu_{CZ}\leq 1$. Repeating this procedure, we go down the tree one level at a time until we reach a leaf, that is, a finte-energy plane asymptotic to an orbit satisfying $\mu_{CZ}\leq1$; this absurd concludes the argument.

Fix $R_0>1$ such that $\pi_\alpha\circ du$ does not vanish on $\C\setminus B_{R_0}$. Since $\util_n \to \util$ in $C^\infty_{\rm loc}(\C\setminus \Gamma)$, if $n$ is large enough we can find
\begin{itemize}
\item[i)] a smooth homotopy $h_n:[0,1]\times \R/\Z\to M$ satisfying $h_n(0,t)=u_n(R_0e^{i2\pi t})$ and $h_n(1,t)=u(R_0e^{i2\pi t})$,
\item[ii)] a non-vanishing section $\kappa_n$ of $(h_n)^*\xi$ satisfying $\kappa_n(0,t)=\pi_\alpha \cdot \partial_ru_n(R_0e^{i2\pi t})$ and $\kappa(1,t)=\pi_\alpha \cdot \partial_ru(R_0e^{i2\pi t})$, and
\item[iii)] a non-vanishing section $Z_n$ of $(h_n)^*\xi$ satisfying $Z_n(1,t)=B(R_0e^{i2\pi t})$.
\end{itemize}
Above $\partial_r$ denotes radial partial derivative with respect to usual polar coordinates on $\C\setminus0$. The vector field $Z_n(0,t)$ along $u_n(R_0e^{i2\pi t})$ lies on $\xi$ and extends as a non-vanishing section of $(u_n|_{B_{R_0}})^*\xi$. This is so by iii) and by the fact that $c_1(\xi)$ vanishes on $\pi_2(M)$. Hence, since $\wind_\pi(\util_n)=0$ for all $n$ we get $\wind(\kappa_n(0,t),Z_n(0,t))=1$. By the homotopy invariance of winding numbers we get $$ 1=\wind(\kappa_n(1,t),Z_n(1,t))=\wind(\pi_\alpha \cdot \partial_ru(R_0e^{i2\pi t}),B(R_0e^{i2\pi t})). $$ Since $R_0$ can be chosen arbitrarily large we conclude that
\begin{equation}\label{wind_infty_util_infty}
\wind_\infty(\util,\infty,B)=1.
\end{equation}
This follows from the asymptotic formula from Theorem~\ref{thm_asymptotics}.

The next step is to use~\eqref{wind_infty_util_infty} to conclude that $\mu_{CZ}(P_i)=2 \ \forall i$. Suppose by contradiction that $\mu_{CZ}(P_{i_0})\geq 3$ for some $i_0$. Then $\wind_\infty(\util,z_{i_0},B) \geq 2$. Since $\mu_{CZ}(P_i)\geq2 \ \forall i$ we have $\wind_\infty(\util,z_i,B)\geq 1 \ \forall i$. This leads to the following absurdity:
\begin{equation*}
\begin{aligned}
& \wind_\pi(\util) \\
& = \wind_\infty(\util) + N - 1 \\
& = \wind_\infty(\util,\infty,B) - \wind_\infty(\util,z_{i_0},B) - \sum_{i\neq i_0} \wind_\infty(\util,z_i,B) + N - 1 \\
& \leq 1 - 2 - (N-1) + N - 1 = -1.
\end{aligned}
\end{equation*}

Now we can propagate the above arguments down the bubbling-off tree in order to prove that every leaf is a finite-energy plane with asymptotic limit satisfying $\mu_{CZ}=2$. In fact, let $\vtil=(b,v)$ be a finite-energy sphere in the second level of the tree. Then $\vtil$ has exactly one positive puncture $\infty$ where it is asymptotic to the same asymptotic limit $\bar P$ as $\util$ is at one of its negative punctures. We proved above that $\mu_{CZ}(\bar P)=2$. Every asymptotic limit of $\vtil$ at a negative puncture must satisfy $\mu_{CZ}\geq 2$ since, otherwise, we can argue as before to find a leaf of the tree that is a finite-energy plane asymptotic to an orbit with $\mu_{CZ}\leq 1$, absurd. If the contact area of $\vtil$ does not vanish then $\wind_\infty(\vtil,\infty,B') = 1$ where $B'$ is a symplectic trivialization\footnote{The existence of $B'$ follows from the fact that $c_1(\xi)$ vanishes on $\pi_2(M)$.} of $v^*\xi$ that extends at the punctures to trivializations coming from capping disks for its asymptotic limits. This is so because $\mu_{CZ}(\bar P)=2$. A calculation as above, using $\wind_\infty(\vtil,\infty,B')=1$, will provide a contradiction if one the asymptotic limits of $\vtil$ at a negative puncture satisfies $\mu_{CZ}>2$. If the contact area of $\vtil$ vanishes then there is a prime closed Reeb orbit $P_0$ such that $\bar P = P_0^m$ and every asymptotic limit of $\vtil$ at a negative puncture is of the form $P_0^j$ with $j\leq m$. Suppose, again by contradiction, that an asymptotic limit $P'$ of $\vtil$ at one of its negative punctures satisfies $\mu_{CZ}(P')>2$. Define $j$ by $P' = P_0^j$. We know that $j\leq m$. Denote by $N$ the least common multiple of $j$ and $m$. Of course, every asymptotic limit of $\vtil$ is contractible and, consequently, so is $P_0^N$. We get a contradiction as follows
\begin{align*}
& 2\frac{N}{m} = \mu_{CZ}((P_0^m)^{N/m}) =\mu_{CZ}(P_0^N) = \mu_{CZ}((P_0^j)^{N/j}) \geq 2\frac{N}{j} + 1 \geq 2\frac{N}{m} + 1.
\end{align*}
We concluded that, in all cases, every asymptotic limit of $\vtil$ at a negative puncture satisfies $\mu_{CZ}=2$. This must be so for all curves in the second level of the tree. The same arguments apply to curves in the lower levels when we go down one level at a time. Thus, all leafs are planes asymptotic to orbits with $\mu_{CZ}=2$, as desired.

We are finally ready to conclude our compactness argument, this last step will make use of the linking hypotheses made in Proposition~\ref{prop_compactness_1}. The important remark is that for every plane $\vtil=(b,v):\C\to \R\times M$ associated to a leaf of the tree there are sequences $A_n,B_n \in \C$, $c_n\in\R$, such that $\util_n(A_nz+B_n) +c_n \to \vtil$ in $C^\infty_{\rm loc}$. Let us assume by contradiction that the tree has at least two levels. Then the plane $\vtil$ satisfies $\wind_\infty(\vtil)=1$ because its asymptotic limit $P_*$ satisfies $\mu_{CZ}(P_*)=2$. Consequently $\vtil$ is an immersion transverse to the Reeb flow. Suppose that $P_*$ is not geometrically distinct from $P$. Setting $P_{\rm min}=(x,T_{\rm min})$ we have $P=P_{\rm min}^p$ and find $j$ such that $P_* = P_{\rm min}^j$. If $N$ is the least common multiple of $j$ and $p$, then using $\mu_{CZ}(P)\geq 3$ and $\mu_{CZ}(P_*)=2$ we compute as above
\begin{align*}
& 2\frac{N}{j} = \mu_{CZ}(P_*^{N/j}) = \mu_{CZ}(P_{\rm min}^N) = \mu_{CZ}(P^{N/p}) \geq 2\frac{N}{p} + 1 \ \Rightarrow  \ j<p.
\end{align*}
But $P_{\rm min}$ induces an element of $\pi_1$ with order precisely $p$, contradicting contractibility of $P_*$. Hence $P_*$ is geometrically distinct from $P$. The period of $P_*$ is the Hofer energy of $\vtil$ and, consequently, is not larger than $T$. By hypothesis $P_*$ is not contractible in $M\setminus x(\R)$. This forces intersections of $\vtil$ with the $\jtil$-complex surface $\R\times x(\R)$. By positivity and stability of intersections we find intersections of the maps $\util_n$ with $\R\times x(\R)$ for $n\gg1$, contradicting~\eqref{intersection_asymptotic_fast}. We have proved that the tree has exactly one level, consisting of its root. In other words, $\Gamma=\emptyset$ and $\util$ is a plane. It is asymptotic to $P$ and~\eqref{wind_infty_util_infty} tells us that it is fast. Clearly $u(0)=y$.

It only remains to be checked that $\util$ is an embedding. It must be somewhere injective since, otherwise, we would factor it through a somewhere injective plane via a complex polynomial of degree at least $2$, forcing critical points of $\util$, but this is impossible since $\util$ is an immersion. Hence the set $A$ of self-intersection points of $\util$ is discrete. If $A\neq \emptyset$ then again by positivity and stability of intersections we would find self-intersections of $\util_n$ for $n$ large, but this is impossible because the $\util_n$ are embeddings. We have finally proved that $\J_{\rm fast}(\alpha,y)$ is $C^\infty_{\rm loc}$-closed.

\subsection{$\J_{\rm fast}(\alpha,y)$ is open.}

Most of the argument can be found in~\cite[appendix~A]{openbook}. Let $J_0\in\J_{\rm fast}(\alpha,y)$ and $\util_0=(a_0,u_0):\C\to \R\times M$ be an embedded fast finite-energy $\jtil_0$-holomorphic plane asymptotic to $P$ and satisfying $y\in u_0(\C)$. After rotating the domain we can assume that $u(Re^{i2\pi t}) \to x(Tt)$ as $R\to+\infty$.

Fix $l\geq 1$ and let $\mathcal{K}^l$ be the space of $C^l$-sections $K$ of the vector bundle $\mathcal{L}(\xi)$ satisfying $J_0K+KJ_0=0$ and $d\alpha(\cdot,K\cdot)+d\alpha(K\cdot,\cdot)=0$ on $\xi|_p$, for every $p\in M$. With the $C^l$-norm this space becomes a Banach space. If $r>0$ is small enough then every $K$ belonging to the ball of radius $r$ centered at the origin in $\mathcal{K}^l$ induces a complex structure $J=J_0\exp(-J_0K)$ on $\xi$ of class $C^l$ which is $d\alpha$-compatible. The set $\mathcal{U}_r$ of $J$ that arise in this way contains a neighborhood of $J_0$ in $\J$, and admits the structure of a (trivial) Banach manifold via the above explained identification with a ball in $\mathcal{K}^l$.

Fix any $\R$-invariant metric $g$ on $\R\times M$ for which $\jtil_0$ is a pointwise isometry, like for instance $g = da \otimes da + \lambda \otimes \lambda + d\lambda(\cdot,J_0\cdot)$. Then the normal bundle $N$ of $\util_0(\C)$ is a $\jtil_0$-invariant vector bundle over $\util_0(\C)$. Denoting by $\pi_M:\R\times M \to M$ the projection onto the second factor, clearly the bundle $\tilde\xi := \pi_M^*\xi$ is also $\jtil_0$-invariant. In view of the asymptotic formula we can find $R_0$ large enough and a $\jtil_0$-invariant subbundle $L\subset \util_0^*T(\R\times M)$ which coincides with $\util_0^*\tilde\xi = u_0^*\xi$ on $\C\setminus B_{R_0}$, and coincides with $\util_0^*N$ on $B_{R_0-1}$. Here we denoted by $B_\rho$ the ball of radius $\rho$ in $\C$ centered at the origin.

Now let $(U,\Psi)$ be a Martinet tube at $P$. As explained in Definition~\ref{defn_martinet} there are associated coordinates $(\theta,x_1,x_2)\in \R/\Z\times B$ where $\alpha \simeq g(d\theta+x_1dx_2)$ for some smooth function $g$ satisfying $g(\theta,0,0)\equiv T_{\rm min}$, $dg(\theta,0,0)=0$ for all $\theta$.

Consider the $p$-covering space $\R/p\Z\times B \to \R/\Z\times B$, and denote the $\R/p\Z$-coordinate by~$\theta'$. The projection is $(\theta',x_1,x_2) \mapsto (\theta,x_1,x_2)$ and we denote by $\alpha'$ the lift of $\alpha$ to $\R/p\Z\times B$, namely $\alpha' = g(\theta,x_1,x_2)(d\theta'+x_1dx_2)$. The lift of $J_0|_U$ to a $d\alpha'$-compatible complex structure on $\xi' := \ker \alpha'$ will be denoted by $J_0'$. Perhaps after making $R_0$ larger we can assume that $\util_0(\C\setminus B_{R_0}) \subset \R\times U$ and, consequently, we can consider a lift $(\Psi\circ u_0)' : \C\setminus B_{R_0} \to \R/p\Z\times B$ of $\Psi\circ u_0$. There exists a $(d\alpha',J_0')$-unitary frame $\{n_1',n_2'\}$ of $\xi'$ on $\R/p\Z\times B$ such that the $n_i' \circ (\Psi\circ u_0)'$ form a frame with the following important property: after identifying $(\Psi\circ u_0)'^*(\xi') \simeq u_0^*\xi|_{\C\setminus B_{R_0}} = L|_{\C\setminus B_{R_0}}$ in the obvious manner, $\{n_1'\circ (\Psi\circ u_0)',n_2'\circ (\Psi\circ u_0)'\}$ extends to a complex frame of $L$. This extended frame of $L$ will be denoted by $\{n_1,n_2\}$.

Note that for $|z|$ large enough $L_z = \xi|_{u_0(z)}$, and taking the limit as $|z|\to\infty$ the frame $\{n_1,n_2\}$ induces a $(d\alpha,J_0)$-unitary frame $\{\bar n_1,\bar n_2\}$ of $(x_T)^*\xi$. This follows from the construction of $\{n_1,n_2\}$ explained above. Let $\{e_1,e_2\}$ be a $d\alpha$-symplectic frame of $(x_T)^*\xi$ induced by some capping disk for $P$. It follows from~\cite[Theorem~1.8]{props3} that
\begin{equation}\label{winding_disk_to_normal}
\wind(t\mapsto e_1(t),t\mapsto n_1(t)) = -1.
\end{equation}

Let $\exp$ be the exponential map associated to the metric $g$. In view of the asymptotic formula and of the $\R$-invariance of $g$ we can find a small open ball $B'\subset \C$ centered at the origin such that the map
\begin{equation}\label{map_Phi_fred_theory}
\begin{array}{ccc} \Phi : \C\times B' \to \R\times M & & \Phi(z,w) = \exp_{\util_0(z)}(\Re[w]n_1(z) + \Im[w]n_2(z)) \end{array}
\end{equation}
is an immersion onto a neighborhood of $\util_0(\C)$. For every $J \in \U_r$ we can consider the $\R$-invariant almost complex structure $\jtil$ on $\R\times M$ determined by $(\alpha,J)$. Denoting $\bar J = \Phi^*\jtil$, then $\bar J$ can be seen as a smooth $\R^{4\times 4}$-valued function on $\C\times B'$ which can be written in $2\times 2$ blocks as
\begin{equation*}
\bar J(z,w) = \begin{pmatrix} j_1(z,w) & \Delta_1(z,w) \\ \Delta_2(z,w) & j_2(z,w)\end{pmatrix}.
\end{equation*}
The graph of a differentiable function $z\mapsto v(z)$ has $\bar J$-invariant tangent space if, and only if, $H(v,J)=0$ where
\begin{equation}
H(v,J) = \Delta_2(z,v) + j_2(z,v) \ dv - dv \ j_1(z,v) - dv \ \Delta_1(z,v) \ dv.
\end{equation}

Given $\gamma\in(0,1)$ and $\delta<0$ we can consider the space $C^{l,\gamma,\delta}_0(\C,\C)$ defined as the set of maps $v:\C\to\C$ of class $C^{l,\gamma}$ such that $(s,t) \mapsto v(e^{2\pi(s+it)})$ is of class $C^{l,\gamma,\delta}_0$ on $[0,+\infty)\times\R/\Z$; see Remark~\ref{rmk_Holder_with_decay}. Consider $Y = \C\times \mathcal{L}_\R(\C)$ as a trivial vector bundle over $\C$. The space $C^{l,\gamma,\delta}_0(Y)$ is defined as the space of $C^{l,\gamma}$-sections $z\mapsto A(z)$ such that $(s,t) \mapsto A(e^{2\pi(s+it)})e^{2\pi(s+it)}$ is of class $C^{l,\gamma,\delta}_0$ on $[0,+\infty)\times\R/\Z$. The splitting $\mathcal{L}_\R(\C) = \mathcal{L}^{1,0}(\C) \oplus \mathcal{L}^{0,1}(\C)$ as $\C$-linear and $\C$-antilinear maps induces splittings $Y = Y^{1,0} \oplus Y^{0,1}$ and $C^{l,\gamma,\delta}_0(Y) = C^{l,\gamma,\delta}_0(Y^{1,0}) \oplus C^{l,\gamma,\delta}_0(Y^{0,1})$. The set $\mathcal{V}\subset C^{l,\gamma,\delta}_0(\C,\C)$ of maps with image in $B'$ is clearly open. It is a standard procedure to check that
\begin{equation*}
H:(v,J) \in \mathcal{V} \times \U_r \mapsto H(v,J) \in C^{l-1,\gamma,\delta}(Y)
\end{equation*}
defines a smooth map.

Let us write
\begin{equation*}
\bar J_0 = \begin{pmatrix} j^0_1 & \Delta_1^0 \\ \Delta_2^0 & j_2^0 \end{pmatrix}
\end{equation*}
so that $j^0_1(z,0)=j^0_2(z,0)=i$ and $\Delta^0_1(z,0)=\Delta^0_2(z,0)=0$. It follows that
\begin{equation*}
D_1H(0,J_0)\zeta = i \ d\zeta - d\zeta \ i + D_2\Delta_2^0(z,0)\zeta.
\end{equation*}
Differentiating the identity $\bar J_0^2=-I$ we get that $D := D_1H(0,J_0)$ takes values on $C^{l-1,\gamma,\delta}_0(Y^{0,1})$.

\begin{theorem}[Hofer, Wysocki and Zehnder~\cite{props3}]\label{thm_index_HWZ}
If $\delta<0$ is not in the spectrum of the asymptotic operator at $P$ associated to $(\alpha,J_0)$ then $$ D : C^{l,\gamma,\delta}_0(\C,\C) \to C^{l-1,\gamma,\delta}_0(Y^{0,1}) $$ is a Fredholm operator with index $\mu_{CZ}^\delta(P)-1$.
\end{theorem}

Here $\mu_{CZ}^\delta(P)$ denotes the Conley-Zehnder index associated to the linearized $\alpha$-Reeb flow along $P$ represented by a $d\alpha$-symplectic frame induced by a capping disk, which takes the exponential weight $\delta$ into account in the following manner: $\mu_{CZ}^\delta(P) = 2k$ if $\delta$ lies between two eigenvalues of $A$ with winding number $k$ with respect to this frame, or $\mu_{CZ}^\delta(P)=2k+1$ if $\delta$ lies between eigenvalues with winding numbers $k$ and $k+1$ with respect to this frame.

From now on we fix our choice of~$\delta$: since $\rho(P)>1 \Leftrightarrow \mu_{CZ}(P)\geq 3$ we can place $\delta$ between eigenvalues with winding numbers $1$ and $2$, so that $\mu_{CZ}^\delta(P)=3$ and the index of $D$ is $2$.

The analysis of~\cite{props3} shows that there exists a (trivial) smooth Banach bundle $\mathcal{E} \to \mathcal{V} \times \U_r$ with fibers modeled on $C^{l-1,\gamma,\delta}(Y^{0,1})$ and a smooth section $\eta$ of $\mathcal{E}$ such that $\eta(v,J)=0 \Leftrightarrow H(v,J)=0$. Moreover, the partial vertical derivative $D_1\eta(0,J_0)$ at $(0,J_0)$ coincides with $D$.

\begin{lemma}\label{lemma_automatic_transv}
The operator $D$ is surjective.
\end{lemma}

\begin{proof}
Consider the map $\sigma(s,t) = e^{2\pi(s+it)}$, fix $\zeta \in C^{l,\gamma,\delta}(\C,\C)$ and denote $a(s,t)=\zeta\circ\sigma(s,t)$. If we evaluate $id\zeta - d\zeta i + D_2\Delta_2^0(z,0)\zeta$ at $z=\sigma(s,t)$ and apply it to $\partial_t\sigma$ we obtain $a_s+ia_t+C(s,t)a$ where $C(s,t)u=[D_2\Delta_2(\sigma(s,t),0)u]\partial_t\sigma(s,t)$. It is shown in~\cite{props3} that $C(s,t) \to S(t)$ as $s\to+\infty$, where $S(t)$ is a smooth $1$-periodic path of symmetric $2\times 2$ matrices such that the linearized Reeb flow along $(x_T)^*\xi$ represented in the frame $\{\bar n_1,\bar n_2\}$ yields a path $\varphi(t) \in Sp(2)$ satisfying $-i\dot\varphi-S\varphi=0$. In view of~\eqref{winding_disk_to_normal} we get $\mu_{CZ}^\delta(\varphi)=1$ and conclude that all eigenvectors of the operator $A \simeq -i\partial_t-S$ associated to eigenvalues smaller than $\delta$ have winding number smaller than or equal to $0$. Here we used the monotonicity of winding numbers along the spectrum of $A$; see~\cite{props2}.

If $\zeta\in\ker D$ is non-zero then it follows from asymptotic analysis as in~\cite{props1,sie1,hryn} that for $s$ large enough $a(s,t)$ does not vanish and and looks like $e^{\lambda s}(u(t)+\epsilon(s,t))$ for some eigenvector $u(t)$ associated to an eigenvalue $\lambda<0$ of $A$. Here $\epsilon(s,t)\to0$ as $s\to+\infty$. In particular, the winding number of $t\mapsto a(s,t)$ is equal to the winding number $m$ of $u$ when $s$ is large enough, and $\lambda<\delta$ since $\zeta \in C^{l-1,\gamma,\delta}$. Thus $m\leq0$ and we can use Carleman's similarity principle together with standard degree theory to conclude that $m=0$ and that any non-trivial $\zeta \in \ker D$ is nowhere vanishing.

We are finally in a position to conclude the argument: if there are three linearly independent vectors in $\ker D$ then a suitable combination of them will vanish at some point since $\dim_\R\C=2$. By the above analysis this linear combination is identically zero, absurd. Hence $\dim\ker D \leq 2$. Since $D$ has index $2$ we conclude that $\dim\ker D=2$ and $\dim{\rm coker}\ D=0$, as desired.
\end{proof}

The above lemma is an automatic transversality statement, it is found in~\cite{hryn}, and also proved in~\cite{wendl}. As a consequence we find an open neighborhood $\mathcal{O}$ of $J_0$ in $\U_r$ such that, perhaps after shrinking $\mathcal{V}$, the universal moduli space $$ \M^{\rm univ} = \{ (v,J) \in \mathcal{V} \times \mathcal{O} \mid \eta(v,J)=0 \} $$ is a smooth submanifold of $\mathcal{V} \times \mathcal{O}$. However, there is an important additional piece of information: perhaps after shrinking $\mathcal{V}$ and $\mathcal{O}$ even more we may assume that the projection $\Pi:(v,J) \in \mathcal{V} \times \mathcal{O} \mapsto J \in \mathcal{O}$ onto the second factor induces a submersion $$ \Pi|_{\M^{\rm univ}} : \M^{\rm univ} \to \mathcal{O} $$ such that each fiber $\M(J) := \Pi^{-1}(J) \cap \M^{\rm univ}$ is non-empty. Then each $\M(J)$ is a smooth (non-empty) manifold, and by~\cite[Lemma A.3.6]{mcdsal} and Theorem~\ref{thm_index_HWZ} we have $\dim \M(J) = \mu_{CZ}^\delta(P)-1 = 2$ for all $J\in\mathcal{O}$.

We need to argue a little more in order to keep track of the point $y$. There is no loss of generality to assume that $\util_0(0)=(0,y)$. We introduce the evaluation map $$ \begin{array}{ccc} {\rm ev} : \C\times \M^{\rm univ} \to \R\times M & & {\rm ev}(z,v,J) = \Phi(z,v(z)) \end{array} $$ where $\Phi$ is the map~\eqref{map_Phi_fred_theory}. This is easily proved to be a smooth map and $(0,0,J_0) \in {\rm ev}^{-1}(0,y)$.

\begin{lemma}\label{lemma_transv_ev_map}
The map ${\rm ev}|_{\C\times \M(J_0)}:\C\times\M(J_0) \to \R\times M$ is non-singular at the point $(0,0,J_0)$.
\end{lemma}

\begin{proof}
Since $\util_0$ is an embedding we only need to show that the partial derivative of ${\rm ev}|_{\C\times\M(J_0)}$ in the $\M(J_0)$-direction is transverse to $\util_0(\C) = {\rm ev}(\C\times\{(0,J_0)\})$. This would not be the case precisely when there is a section in $\ker D$ which vanishes somewhere, but this possibility was ruled out by the argument used to prove Lemma~\ref{lemma_automatic_transv}.
\end{proof}

It is an immediate consequence of Lemma~\ref{lemma_transv_ev_map} that the differential $d({\rm ev})|_{(0,0,J_0)}$ of the map ${\rm ev}$ at the point $(0,0,J_0)$ is onto, so there is no loss of generality to assume that ${\rm ev}^{-1}(0,y)$ is a smooth submanifold. The codimension of ${\rm ev}^{-1}(0,y)$ is $4$, and by Lemma~\ref{lemma_transv_ev_map} ${\rm ev}^{-1}(0,y)$ intersects $\C\times \M(J_0)$ transversally at the point $(0,0,J_0)$. In particular, since $\dim \C\times \M(J_0)=4$, it follows that the restriction to $T_{(0,0,J_0)}{\rm ev}^{-1}(0,y)$ of the linearization of the map $(z,v,J) \mapsto J$ at the point $(0,0,J_0)$ is surjective. This completes the proof that for every $J$ close enough to $J_0$ in $\U_r$ there exists some $(v,J) \in \M(J)$ such that $(0,y)$ belongs to the image of the map $z \mapsto {\rm ev}(z,v,J)$.

In order to finish the proof that $J_0$ is an interior point of $\J_{\rm fast}(\alpha,y)$ it only remains to be shown that the map $z\mapsto {\rm ev}(z,v,J)$ can be reparametrized as an embedded fast finite-energy $\jtil$-holomorphic plane, whenever $J$ is close enough to $J_0$ in $\U_r$ and $(v,J) \in \M^{\rm univ}$. This argument uses the analysis of~\cite[appendix]{props3} and has been spelled out in detail in~\cite[appendix~A]{openbook}. We sketch it here: using the analysis of~\cite[appendix]{props3} such a map $z\mapsto {\rm ev}(z,v,J)$ can be reparametrized as a finite-energy $\jtil$-holomorphic plane $\util$. It is clearly an embedded plane. Since $v\in C^{l,\gamma,\delta}_0$ we can use asymptotic analysis as in~\cite{props1,sie1,hryn} to show that the asymptotic eigenvalue of this plane is $\leq\delta$, and hence $<\delta$ since $\delta$ does not belong to the spectrum of the corresponding asymptotic operator. In particular, $\wind_\infty(\util)\leq 1$ and, by the similarity principle, we have $\wind_\infty(\util)=1 \Rightarrow \wind_\pi(\util)=0$, as desired.

\section{Proof of Proposition~\ref{prop_unif_asymptotic_analysis}}\label{app_unif_asymp_analysis}

The arguments here are almost identical to those of~\cite[appendix~B]{openbook}, we include them for the sake of completeness. We fix a sequence of smooth maps
\begin{equation}
K_n : [0,+\infty) \times \R/\Z \to \R^{2m\times 2m} \ \ \ (n\geq1)
\end{equation}
and a map
\begin{equation}
K^\infty : \R/\Z \to \R^{2m\times 2m}
\end{equation}
such that $K^\infty(t)$ is symmetric for all $t$, and satisfying the following property: for every pair of sequences $n_l,s_l\to+\infty$ there exists $l_k\to+\infty$ and $c\in[0,1)$ such that
\begin{equation}
\lim_{k\to\infty} \| \partial^{\beta_1}_s\partial^{\beta_2}_t[K_{n_{l_k}}(s,t)-K^\infty(t+c)](s_{l_k},\cdot) \|_{L^\infty(\R/\Z)} = 0
\end{equation}
for all $\beta_1,\beta_2\geq0$.

From now on $L$ will denote the unbounded self-adjoint operator
\begin{equation}\label{operator_L}
\begin{array}{ccc} L : W^{1,2} \subset L^2 \to L^2 & & e(t) \mapsto -J_0\dot e(t)-K^\infty(t) e(t) \end{array}
\end{equation}
where $W^{1,2} = W^{1,2}(\R/\Z,\R^{2m})$, $L^2=L^2(\R/\Z,\R^{2m})$ and $J_0$ is the standard complex structure
\begin{equation}
J_0 = \begin{pmatrix} 0 & -I \\ I & 0 \end{pmatrix}
\end{equation}
on $\R^{2m}$ (written in four $m\times m$ blocks). The first step towards the proof of the proposition is the following lemma.

\begin{lemma}\label{lemma_unif_asympt_1}
For each $s\geq0$ and $n\geq1$ consider the unbounded self-adjoint operator $L_n(s):W^{1,2}\subset L^2 \to L^2$ given by $L_n(s) = -J_0 \partial_t - S_n(s,t)$ where $S_n = \frac{1}{2}(K_n+K_n^T)$ is the symmetric part of $K_n$.
If $\delta<0$ does not lie on the spectrum of the operator $L$~\eqref{operator_L} then we can find $s_0\geq 0$, $n_0\geq 1$ and $c>0$ such that
\begin{equation*}
\begin{aligned}
& \|[L_n(s)-\delta]e\|_{L^2} \geq c\|e\|_{L^2} \\
\end{aligned}
\end{equation*}
for all $s\geq s_0$, $n\geq n_0$, $e\in W^{1,2}$.
\end{lemma}

\begin{proof}
Let us argue by contradiction, assuming that there exist $n_l,s_l \to +\infty$ and vectors $e_l \in W^{1,2}$ such that $\|e_l\|_{L^2}=1 \ \forall l$, and
\begin{equation*}
\|[L_{n_l}(s_l)-\delta]e_l\|_{L^2} \to 0 \ \ \ \text{as} \ \ \ l\to\infty.
\end{equation*}
In view of the assumptions, we find $l_k$ and $c\in[0,1)$ such that $S_{n_{l_k}}(s_{l_k},t) \to K^\infty(t+c)$ uniformly in $t$, as $k\to\infty$. Since
\begin{equation*}
\partial_te_{l_k} = J_0([L_{n_{l_k}}(s_{l_k})-\delta]e_{l_k}) + J_0S_{n_{l_k}}(s_{l_k},\cdot)e_{l_k} + \delta J_0e_{l_k}
\end{equation*}
then $e_{l_k}$ is bounded in $W^{1,2}$. Thus, up to a further subsequence, we may assume that $e_{l_k} \to e$ in $L^2$ as $k\to\infty$, for some $e\in L^2$. Appealing again to the above equation we conclude that $e_{l_k}$ is a Cauchy sequence in $W^{1,2}$, in particular $e\in W^{1,2}$, $\|e\|_{L^2}=1$ and $\|e_{l_k}-e\|_{W^{1,2}} \to 0$. But all this implies that $[L_{n_{l_k}}(s_{l_k})-\delta]e_{l_k} \to -J_0\partial_te-K^\infty(\cdot+c)e - \delta e$ as $k\to\infty$. In other words, $e_c(t)=e(t-c)$ satisfies $Le_c=\delta e_c$ and $\|e_c\|_{L^2}=1$. This is in contradiction with our hypothesis on $\delta$.
\end{proof}

\begin{lemma}\label{lemma_unif_asympt_2}
Suppose that $\delta<0$ does not lie on the spectrum of the operator $L$~\eqref{operator_L}, and let $\mu$ be largest number in $\sigma(L)$ below $\delta$. Then there exists $0<r<\delta-\mu$ and $n_1,s_1>0$ such that
\begin{equation*}
\|X(s,\cdot)\|_{L^2} \leq e^{(\delta-r)(s-s_1)} \| X(s_1,\cdot) \|_{L^2}
\end{equation*}
holds for all $s\geq s_1$ and for all solutions $X$ of
\begin{equation*}
\begin{array}{ccc} \partial_sX+J_0\partial_tX+K_nX=0 & & \lim_{s\to+\infty} e^{-\delta s} \|X(s,\cdot)\|_{L^2} = 0 \end{array}
\end{equation*}
with $n\geq n_1$.
\end{lemma}

\begin{proof}
Let $S_n=\frac{1}{2}(K_n+K_n^T)$ and $A_n=\frac{1}{2}(K_n-K_n^T)$ be the symmetric and anti-symmetric parts of $K_n$. Then we have the following property: for all $\beta_1,\beta_2\geq0$ and for all pairs of sequences $s_l,n_l\to+\infty$ we can find $l_k\to\infty$ and $c\in[0,1)$ such that
\begin{equation}\label{hypotheses_prop}
\begin{aligned}
& \lim_{k\to\infty} \| \partial_s^{\beta_1}\partial_t^{\beta_2}[S_{n_{l_k}}(s,t)-S^\infty(t+c)](s_{l_k},\cdot)\|_{L^\infty(\R/\Z)} = 0 \\
& \lim_{k\to\infty}  \| \partial_s^{\beta_1}\partial_t^{\beta_2}A_{n_{l_k}}(s_{l_k},\cdot) \|_{L^\infty(\R/\Z)} = 0.
\end{aligned}
\end{equation}
In particular, it follows that
\begin{equation}\label{conseq1_hyp_prop}
\lim_{s,n\to+\infty} \| \partial_s^{\beta_1}\partial_t^{\beta_2}S_n(s,\cdot) \|_{L^\infty(\R/\Z)} = 0, \ \ \ \forall \beta_1\geq 1,\beta_2\geq 0,
\end{equation}
in the sense that for all $\beta_1\geq1,\beta_2\geq0$ and $\epsilon>0$ there are numbers $s(\epsilon,\beta_1,\beta_2)$, $n(\epsilon,\beta_1,\beta_2)$ such that $$ s\geq s(\epsilon,\beta_1,\beta_2),n\geq n(\epsilon,\beta_1,\beta_2) \Rightarrow \| \partial_s^{\beta_1}\partial_t^{\beta_2}S_n(s,\cdot) \|_{L^\infty(\R/\Z)}\leq\epsilon. $$ Analogously,
\begin{equation}\label{conseq2_hyp_prop}
\lim_{s,n\to+\infty} \| \partial_s^{\beta_1}\partial_t^{\beta_2}A_n(s,\cdot) \|_{L^\infty(\R/\Z)} = 0, \ \ \ \forall \beta_1\geq0,\beta_2\geq 0.
\end{equation}

The function $X$ solves a partial differential equation which depends on $n$: this is not explicit in the notation $X(s,t)$ but the reader should not forget the vital role played by $n$. Consider $Y(s,t)=e^{-\delta s}X(s,t)$. Then
\begin{equation*}
Y_s - (L_n(s)-\delta)Y + A_nY = 0,
\end{equation*}
where $L_n(s)$ is the self-adjoint operator described in the statement of the previous lemma, and $\|Y(s,\cdot)\|_{L^2(\R/\Z)}^2 \to 0$ as $s\to+\infty$. Setting $g(s)=\frac{1}{2}\|Y(s,\cdot)\|_{L^2(\R/\Z)}^2$ then one quickly computes
\begin{equation}\label{expression_g''}
g''(s) = 2\|(L_n(s)-\delta)Y\|_{L^2}^2 - 2\left< (L_n(s)-\delta)Y,A_nY \right>_{L^2} - \left< (\partial_sS_n)Y,Y \right>_{L^2}.
\end{equation}
For this one uses many times that $L_n(s)$ is self-adjoint.

The following notation will simplify the exposition below: given a function $f$ of $(s,t) \in \R\times \R/\Z$ we write $\|f\|_{L^2,s}$ for the $L^2(\R/\Z)$-norm of $f(s,\cdot)$.

We now follow~\cite[Lemma~B.2]{openbook} closely, giving more details. Let $n_0$, $s_0$ and $c>0$ be given by the previous lemma. In view of~\eqref{conseq2_hyp_prop} we can find $n_1\geq n_0$, $s_1\geq s_0$ such that
\begin{equation}
n\geq n_1,s\geq s_1 \Rightarrow \|A_n(s,\cdot)\|_{L^\infty} < c.
\end{equation}
In particular
\begin{equation}
\begin{aligned}
n\geq n_1,s\geq s_1 \Rightarrow c\|Y\|_{L^2,s} > \|A_n(s,\cdot)\|_{L^\infty}\|Y\|_{L^2,s} \geq \|A_nY\|_{L^2,s}.
\end{aligned}
\end{equation}
Using~\eqref{expression_g''} and the previous lemma we estimate when $s\geq s_1,n\geq n_1$:
\begin{equation}\label{ineq_g''_first}
\begin{aligned}
g''(s) &\geq 2c\|Y\|_{L^2,s}\|(L_n(s)-\delta)Y\|_{L^2,s} - 2\|A_nY\|_{L^2,s}\|(L_n(s)-\delta)Y\|_{L^2,s} \\
&- \|\partial_sS_n(s,\cdot)\|_{L^\infty}\|Y\|_{L^2,s}^2 \\
&= 2\|(L_n(s)-\delta)Y\|_{L^2,s} \left( c\|Y\|_{L^2,s} - \|A_nY\|_{L^2,s} \right) \\
&- \|\partial_sS_n(s,\cdot)\|_{L^\infty}\|Y\|_{L^2,s}^2 \\
&\geq 2c\|Y\|_{L^2,s} \left( c\|Y\|_{L^2,s} - \|A_n(s,\cdot)\|_{L^\infty}\|Y\|_{L^2,s} \right) \\
&- \|\partial_sS_n(s,\cdot)\|_{L^\infty}\|Y\|_{L^2,s}^2 \\
&= 2 \left( c^2 - c\|A_n(s,\cdot)\|_{L^\infty} - \frac{1}{2}\|\partial_sS_n(s,\cdot)\|_{L^\infty} \right) \|Y\|_{L^2,s}^2 \\
&= 4g(s) \left( c^2 - c\|A_n(s,\cdot)\|_{L^\infty} - \frac{1}{2}\|\partial_sS_n(s,\cdot)\|_{L^\infty} \right)
\end{aligned}
\end{equation}
Choose $0<\epsilon\ll 1$. By means of~\eqref{conseq1_hyp_prop} and~\eqref{conseq2_hyp_prop} we find $n_2\geq n_1,s_2\geq s_1$ such that
\begin{equation*}
n\geq n_2,s\geq s_2 \Rightarrow c^2 - c\|A_n(s,\cdot)\|_{L^\infty} - \frac{1}{2}\|\partial_sS_n(s,\cdot)\|_{L^\infty} \geq (c-\epsilon)^2.
\end{equation*}
This and~\eqref{ineq_g''_first} together give
\begin{equation}\label{g''_second_estimate}
n\geq n_2,s\geq s_2 \Rightarrow g''(s) \geq 4(c-\epsilon)^2g(s).
\end{equation}
Now one uses the following fundamental fact about positive functions satisfying a differential inequality of the above type: if the non-negative $C^2$-function $g$ defined on $[s_2,+\infty)$ satisfies~\eqref{g''_second_estimate} and $g(s) \to 0$ as $s\to+\infty$, then
\begin{equation*}
g(s) \leq g(s_2)e^{-2(c-\epsilon)(s-s_2)} \ \ \ \forall s\geq s_2.
\end{equation*}
The conclusion of the lemma follows since $2g(s)=e^{-2\delta s}\|X(s,\cdot)\|_{L^2}^2$.
\end{proof}


\begin{lemma}\label{lemma_unif_bounds_K_n}
Under the hypotheses of Proposition~\ref{prop_unif_asymptotic_analysis}
\begin{equation}\label{bounds_on_K_n}
\limsup_{s,n\to+\infty} \|\partial_s^{\beta_1}\partial_t^{\beta_2}K_n(s,\cdot)\|_{L^\infty(\R/\Z)} < +\infty
\end{equation}
holds for all $\beta_1,\beta_2\geq 0$. In other words, for all $\beta_1,\beta_2\geq 0$ there exist numbers $s(\beta_1,\beta_2)$, $n(\beta_1,\beta_2)$ and $M(\beta_1,\beta_2)$ such that $$ s\geq s(\beta_1,\beta_2), n\geq n(\beta_1,\beta_2), t\in\R/\Z \ \ \Rightarrow  \ \ |\partial_s^{\beta_1}\partial_t^{\beta_2}K_n(s,t)| \leq M(\beta_1,\beta_2). $$
\end{lemma}

\begin{proof}
If this lemma is not true then we find $\beta_1,\beta_2\geq 0$, sequences $n_l,s_l\to+\infty$ and $t_l\in\R/\Z$ such that $|\partial_s^{\beta_1}\partial_t^{\beta_2}K_{n_l}(s_l,t_l)| \to +\infty$. By the hypotheses of Proposition~\ref{prop_unif_asymptotic_analysis}, $\exists$ $l_k$, $c\in[0,1)$ such that $|\partial_s^{\beta_1}\partial_t^{\beta_2}[K_{n_{l_k}}(s,t)-K^\infty(t+c)](s_{l_k},t_{l_k})| \to 0$, absurd.
\end{proof}

\begin{lemma}\label{lemma_unif_asympt_3}
For every $k\geq 1$ there exists $n_k$ such that
\begin{equation}\label{decay_X_n_non_uniform}
n\geq n_k, \ \beta_1+\beta_2\leq k \ \ \Rightarrow \ \ \lim_{s\to+\infty} e^{-\delta s} \|\partial_s^{\beta_1}\partial_t^{\beta_2}X_n(s,\cdot)\|_{L^\infty(\R/\Z)} = 0.
\end{equation}
\end{lemma}

\begin{proof}
Fix $p>1$. In this proof we shall need to make use of the standard elliptic estimate which holds for smooth functions $h:\R\times\R/\Z \to \R^{2m}$ with compact support
\begin{equation}\label{cal_zyg_ineq}
\|h\|_{W^{{\ell}+1,p}(\R\times\R/\Z)} \leq C_{\ell} \left( \|\bar\partial h\|_{W^{{\ell},p}(\R\times\R/\Z)} + \|h\|_{W^{{\ell},p}(\R\times\R/\Z)} \right)
\end{equation}
for the Cauchy-Riemann operator $\bar\partial = \partial_s+J_0\partial_t$; here the constant $C_\ell>0$ is independent of $h$.

Fix a smooth function $\phi:\R\to[0,1]$ satisfying $\phi|_{[-1,1]}\equiv1$ and ${\rm supp}(\phi) \subset (-2,2)$. For each $\tau\in\R$ set $\phi_\tau(s)=\phi(s-\tau)$ and $Q_\tau=[\tau-2,\tau+2]\times\R/\Z$. Now we use~\eqref{bounds_on_K_n} to find for every $\ell\geq1$ numbers $s_\ell,n_\ell$ and $M_\ell$ such that
\begin{equation}
s\geq s_\ell,n\geq n_\ell,t\in\R/\Z \ \ \Rightarrow \ \ |\partial_s^{\beta_1}\partial_t^{\beta_2}K_n(s,t)| \leq M_\ell \ \ \forall \beta_1,\beta_2 \ \text{with} \ \beta_1+\beta_2\leq \ell.
\end{equation}
Using this and~\eqref{cal_zyg_ineq} we can estimate for $\tau \geq s_{\ell}+2$ and $n\geq n_{\ell}$:
\begin{eqnarray}
\begin{aligned}
& \|X_n\|_{W^{{\ell}+1,p}([\tau-1,\tau+1]\times\R/\Z)} \leq \|\phi_\tau X_n\|_{W^{{\ell}+1,p}(\R\times\R/\Z)} \\
& \leq C_{\ell}\left( \|\bar\partial(\phi_\tau X_n)\|_{W^{{\ell},p}(\R\times\R/\Z)} + \|\phi_\tau X_n\|_{W^{{\ell},p}(\R\times\R/\Z)} \right) \\
& \leq C'_{\ell}\left( \|\bar\partial X_n\|_{W^{{\ell},p}(Q_\tau)} + \|X_n\|_{W^{{\ell},p}(Q_\tau)} \right) \\
& = C'_{\ell}\left( \|K_n X_n\|_{W^{{\ell},p}(Q_\tau)} + \|X_n\|_{W^{{\ell},p}(Q_\tau)} \right) \\
& \leq C''_{\ell} \|X_n\|_{W^{{\ell},p}(Q_\tau)}
\end{aligned}
\end{eqnarray}
where $C'_{\ell}$ depends on $C_{\ell}$ and the derivatives of $\phi$ up to order $\ell+1$, and $C''_{\ell}$ depends on $C'_{\ell}$ and $M_{\ell}$. Since $X_n\in E$ for all $n$, we find that $\|X_n\|_{W^{l,p}(Q_\tau)}$ decays like $e^{\delta \tau}$ as $\tau\to+\infty$, where $l$ is the number used in the definition of $E$. An induction argument will tell us that for any $k\geq 1$, $\|X_n\|_{W^{k,p}(Q_\tau)}$ decays like $e^{\delta\tau}$ as $\tau\to+\infty$ if $n$ is larger than some $n_k$. Then~\eqref{decay_X_n_non_uniform} follows from the Sobolev embedding theorem.
\end{proof}

With the above lemmata at our disposal, we finally turn to the proof of the proposition. The number $\delta<0$ does not lie in $\sigma(L)$, and we take $\mu\in\sigma(L)$ satisfying $\mu<\delta$ and $(\mu,\delta] \cap \sigma(L) = \emptyset$.  \\

\noindent {\bf Step 1.} For every $m\geq 0$ we can find $s_m,n_m$ and $0<r_m<\delta-\mu$ such that
\begin{equation*}
\left( \sum_{j=0}^m \| (\partial_s)^jX_n(s,\cdot) \|_{L^2(\R/\Z)}^2 \right)^{\frac{1}{2}} \leq e^{(\delta-r_m)(s-s_m)} \left(\sum_{j=0}^m \| (\partial_s)^jX_n(s_m,\cdot) \|_{L^2(\R/\Z)}^2\right)^{\frac{1}{2}}
\end{equation*}
holds if $s\geq s_m$ and $n\geq n_m$. \\

\begin{proof}[Proof of Step 1]
Differentiating~\eqref{eqn_X_n_prop} with respect to $s$ yields
\begin{equation*}
\begin{aligned}
& \partial_sX_n+J_0\partial_tX_n+K_nX_n = 0 \\
& (\partial_s)^2X_n+J_0\partial_t\partial_sX_n+\partial_s(K_nX_n) = 0 \\
& \vdots \\
& (\partial_s)^{m+1}X_n+J_0\partial_t(\partial_s)^mX_n+(\partial_s)^m(K_nX_n) = 0
\end{aligned}
\end{equation*}
This system of equations can be rewritten as a single equation in the form
\begin{equation*}
\partial_sZ_n + \hat J_0\partial_tZ_n + \hat K_nZ_n = 0
\end{equation*}
where $Z_n = \begin{bmatrix} X_n,\partial_sX_n,\dots,(\partial_s)^mX_n \end{bmatrix}^T$,
\begin{equation*}
\hat J_0 = \begin{bmatrix} J_0 & 0 & \dots & 0 \\ 0 & J_0 & \dots & 0 \\ \vdots & \vdots & \ddots & \vdots \\ 0 & 0 & \dots & J_0 \end{bmatrix} \ \ \ \hat K_n = \begin{bmatrix} K_n & 0 & \dots & 0 \\ * & K_n & \dots & 0 \\ \vdots & \vdots & \ddots & \vdots \\ * & * & \dots & K_n \end{bmatrix}
\end{equation*}
where every block-entry of the lower triangular block-matrix $\hat K_n$ below the diagonal is a term $\Delta_n(s,t)$ satisfying
\begin{equation*}
\lim_{s,n\to+\infty} \|\partial_s^{\beta_1}\partial_t^{\beta_2}\Delta_n(s,\cdot)\|_{L^\infty(\R/\Z)} = 0 \ \ \ \forall \beta_1,\beta_2\geq0.
\end{equation*}
We can now apply Lemma~\ref{lemma_unif_asympt_2} to $Z_n$ in order to obtain get $s_m,n_m,r_m$ with the properties we desired.
\end{proof}

\noindent {\bf Step 2.} For every integer $q\geq0$ we can find numbers $s_q$, $n_q$, and $0<r_q<\delta-\mu$, $c_q$ such that
\begin{equation}\label{full_step_2}
\begin{aligned}
\max_{\beta_1+\beta_2\leq q} & \|\partial_s^{\beta_1}\partial_t^{\beta_2}X_n(s,\cdot)\|_{L^\infty(\R/\Z)} \\
& \leq c_q e^{(\delta-r_q)(s-s_q)} \max_{\beta_1+\beta_2\leq q+1} \| \partial_s^{\beta_1}\partial_t^{\beta_2}X_n(s_q,\cdot) \|_{L^\infty(\R/\Z)}
\end{aligned}
\end{equation}
holds for all $s\geq s_q$ and $n\geq n_q$.

\begin{proof}[Proof of Step 2]
We will prove by induction that for every $m\geq0$ there exist $k_m,s_m,n_m$ and $0<\rho_m<\delta-\mu$ such that
\begin{equation}\label{step_2_L2}
\begin{aligned}
& \sqrt{\sum_{\beta_1+\beta_2\leq m} \|\partial_s^{\beta_1}\partial_t^{\beta_2}X_n(s,\cdot)\|_{L^2(\R/\Z)}^2} \\
& \leq k_m e^{(\delta-\rho_m)(s-s_m)} \sqrt{\sum_{\beta_1+\beta_2\leq m} \|\partial_s^{\beta_1}\partial_t^{\beta_2}X_n(s_m,\cdot)\|_{L^2(\R/\Z)}^2}
\end{aligned}
\end{equation}
holds for when $s\geq s_m$ and $n\geq n_m$. If we succeed in proving~\eqref{step_2_L2} then the proof of {\it Step 2} will be complete in view of the Sobolev embedding theorem, which tells us that $W^{1,2}(\R/\Z) \hookrightarrow L^\infty(\R/\Z)$ continuously.

Now we proceed with the proof of~\eqref{step_2_L2}. The case $m=0$ is a special case of {\it Step~1}. Assuming~\eqref{step_2_L2} holds for $m$ we now show that it also holds for $m+1$. However, this induction step will be proved by a separate induction argument: we will show, using induction in $0\leq j\leq m+1$, that for every such $j$ we can find $s',n'$, $c'$ and $0<\rho'<\delta-\mu$ such that
\begin{equation}\label{sub_induction_j}
\|\partial_s^{m+1-j}\partial_t^jX_n(s,\cdot)\|_{L^2(\R/\Z)} \leq c'e^{(\delta-\rho')(s-s')} \sqrt{\Sigma' \| \partial_s^{\beta_1}\partial_t^{\beta_2}X_n(s',\cdot) \|_{L^2(\R/\Z)}^2}
\end{equation}
holds if $s\geq s'$ and $n\geq n'$, where the sum $\Sigma'$ indicates a sum over all numbers $\beta_1,\beta_2\geq 0$ satisfying either $\beta_1+\beta_2\leq m$, or $\beta_1+\beta_2=m+1$ and $\beta_2\leq j$. The case $j=0$ follows from~{\it Step~1} and our previous induction hypothesis. Fix $b\leq m+1$ and assume that~\eqref{sub_induction_j} holds for all $0\leq j\leq b-1$. By~\eqref{eqn_X_n_prop} we get
\begin{equation*}
\begin{aligned}
\partial_s(\partial_s^{m+1-b}\partial_t^{b-1}X_n)&+J_0\partial_s^{m+1-b}\partial_t^bX_n \\
&= \partial_s^{m+1-b}\partial_t^{b-1}(\partial_sX_n+J_0\partial_tX_n) \\
&= \partial_s^{m+1-b}\partial_t^{b-1}(-K_nX_n)
\end{aligned}
\end{equation*}
Hence
\begin{equation*}
\partial_s^{m+1-b}\partial_t^bX_n = J_0(\partial_s^{m+2-b}\partial_t^{b-1}X_n + \partial_s^{m+1-b}\partial_t^{b-1}(K_nX_n)).
\end{equation*}
This equation, the uniform asymptotic bounds~\eqref{bounds_on_K_n} on derivatives of $K_n$ and the induction hypothesis together prove~\eqref{sub_induction_j} for $0\leq j\leq b$. Hence~\eqref{sub_induction_j} holds for all $0\leq j\leq m+1$, which together with~\eqref{step_2_L2} for $m$ gives~\eqref{step_2_L2} for $m+1$. The proof of {\it Step~2} is complete.
\end{proof}

We are ready for the final \\

\noindent {\bf Step 3.} Some subsequence of $X_n$ converges in $C^{l,\alpha,\delta}_0$.

\begin{proof}[Proof of Step 3]
Since $X_n$ is $C^\infty_{\rm loc}$-bounded we can find $X_\infty(s,t)$ smooth and assume, up to selection of a subsequence, that $X_n\to X_\infty$ in $C^\infty_{\rm loc}$. We will show now that
\begin{equation}\label{final_step_3_crucial}
\lim_{n\to\infty} \ \ \left[ \sup_{(s,t)\in[0,+\infty)\times\R/\Z} e^{-\delta s}| \partial_s^{\beta_1}\partial_t^{\beta_2}[X_n-X_\infty]| \right] = 0
\end{equation}
holds for every $\beta_1,\beta_2\geq0$. Let $\beta_1,\beta_2\geq 0$ and $\epsilon>0$ be fixed arbitrarily. By {\it Step~2} we find $s',n'$ such that
\begin{equation}\label{eq_step3_1}
\sup_{s\geq s', n\geq n'} e^{-\delta s}| \partial_s^{\beta_1}\partial_t^{\beta_2}X_n(s,t)| \leq \epsilon/2.
\end{equation}
For this we used formula~\eqref{full_step_2} and the fact that $X_n\to X_\infty$ in $C^\infty_{\rm loc}$. In particular, taking the limit as $n\to\infty$ one gets
\begin{equation}\label{eq_step3_2}
\sup_{s\geq s'} e^{-\delta s}| \partial_s^{\beta_1}\partial_t^{\beta_2}X_\infty(s,t)| \leq \epsilon/2.
\end{equation}
In view of the $C^\infty_{\rm loc}$-convergence we find $n''\geq n'$ such that
\begin{equation}\label{eq_step3_3}
n\geq n'' \Rightarrow \sup_{(s,t) \in [0,s']\times\R/\Z} e^{-\delta s}| \partial_s^{\beta_1}\partial_t^{\beta_2}[X_n-X_\infty]| \leq \epsilon.
\end{equation}
Putting~\eqref{eq_step3_1},~\eqref{eq_step3_2} and~\eqref{eq_step3_3} together we obtain
\begin{equation}\label{eq_step3_4}
n\geq n'' \Rightarrow \sup_{(s,t) \in [0,+\infty)\times\R/\Z} e^{-\delta s}| \partial_s^{\beta_1}\partial_t^{\beta_2}[X_n-X_\infty]| \leq \epsilon.
\end{equation}
Since $\epsilon>0$ was arbitrary, this proves~\eqref{final_step_3_crucial}. Since~\eqref{final_step_3_crucial} holds for any $\beta_1,\beta_2\geq 0$ it follows that $X_\infty \in C^{l,\alpha,\delta}_0$ and $X_n\to X_\infty$ in $C^{l,\alpha,\delta}_0$.
\end{proof}

The proof of Proposition~\ref{prop_unif_asymptotic_analysis} is now complete.


\begin{thebibliography}{cc}
\bibitem{Abb} C. Abbas. \textit{An introduction to compactness results in symplectic field theory.} Springer, 2014.
\bibitem{AM} M. Abreu and L. Macarini. {\it Dynamical convexity and elliptic periodic orbits for Reeb flows.} arXiv:1411.2543.
\bibitem{AFvKP} P.~Albers, U.~Frauenfelder, O.~van Koert, and G.~P. Paternain. {\it Contact geometry of the restricted three-body problem.} Comm. Pure Appl. Math., 65(2):229--263, 2012.
\bibitem{AFFHvK} P. Albers, J. W. Fish, U. Frauenfelder, H. Hofer, and O. van Koert. \textit{Global surfaces of section in the planar restricted 3-body problem.} Archive for Rational Mechanics and Analysis, {\bf 204} (2012), no.1, 273--284.
\bibitem{BK} K. L. Baker and J. B. Etnyre. {\it Rational Linking and contact geometry.} Perspectives in analysis, geometry, and topology, Progress in Mathematics 296 (Birkh\"auser/Springer, New York, 2012) 19--37.
\bibitem{BL} V. Bangert and Y. Long. \textit{The existence of two closed geodesics on every Finsler 2-sphere.} Mathematische Annalen {\bf 346} (2010), no.2, 335--366.
\bibitem{ben} G. Benedetti. {\it The contact property for symplectic magnetic fields on $S^2$.} Ergodic Theory and Dynamical Systems: 1--32.
\bibitem{B}G. D. Birkhoff, \textit{The restricted problem of three bodies." Rendiconti del Circolo Matematico di Palermo}, {\bf 39 1}, (1915), 265--334.
\bibitem{sftcomp}F. Bourgeois, Y. Eliashberg, H. Hofer, K. Wysocki and E. Zehnder. \textit{Compactness results in Symplectic Field Theory.} Geometry and Topology, Vol. 7 (2004), 799--888.
\bibitem{BH} B. Bramham, H. Hofer, {\it First  steps towards a symplectic dynamics}, Surveys in Differential Geometry, Vol. 17 (2012), 127--178.
\bibitem{conley} C. Conley. {\it On some new long periodic solutions of the plane restricted three body problem.} Commun. Pure Appl. Math. 16, (1963) pp. 449--467.
\bibitem{PS} Naiara V. de Paulo and Pedro A. S. Salom\~ao. {\it Systems of transversal sections near critical energy levels of Hamiltonian systems in $\mathbb{R}^4$,} to appear in Memoirs of AMS.
\bibitem{AOE} G. Dell'Antonio, B. D'Onofrio, I. Ekeland. {\it Periodic solutions of elliptic type for strongly nonlinear Hamiltonian systems.} The Floer memorial volume, 327--333, Progr. Math., 133, Birkh\"auser, Basel, 1995.
\bibitem{ekeland_elliptic} I. Ekeland. {\it An index theory for periodic solutions of convex Hamiltonian systems.} Nonlinear functional analysis and its applications, Part 1 (Berkeley, Calif., 1983), 395--423, Proc. Sympos. Pure Math., 45, Part 1, Amer. Math. Soc., Providence, RI, 1986.
\bibitem{ekeland} I. Ekeland. {\it Convexity methods in Hamiltonian mechanics.} Springer, 1990. Ergebnisse der Mathematik und ihrer Grenzgebiete, 3. Folge.
\bibitem{FS} J. W. Fish and R. Siefring. {\it Connected Sums and finite energy foliations I: Contact connected sums}, arXiv 1311.4221.
\bibitem{Fl1} A.~Floer, {\em Symplectic fixed points and holomorphic spheres.} Comm. Math. Phys. {\bf 120} (1989), no. 4, 575--611.
\bibitem{Fl2} A.~Floer, {\em Morse theory for Lagrangian intersections.} J. Differential Geom. {\bf 28} (1988), no. 3, 513--547.
\bibitem{Fl3} A.~Floer, {\em The unregularized gradient flow of the symplectic action.} Comm. Pure Appl. Math. {\bf 41} (1988), no. 6, 775--813.
\bibitem{FK} U. Frauenfelder and J. Kang. {\it Real holomorphic curves and invariant global surfaces of section.} arXiv:1501.04853.
\bibitem{gei} H. Geiges,  {\em An introduction to contact topology.} Vol. 109. Cambridge University Press, 2008.
\bibitem{gidea} M. Gidea. {\it Global Diffusion on a Tight Three-Sphere} Qual. Theory Dyn. Syst. DOI 10.1007/s12346-015-0142-3.
\bibitem{gromov}M. Gromov. \textit{Pseudoholomorphic curves in symplectic manifolds.} Invent. Math. {\bf 82} (1985), 307-347.
\bibitem{hill} G. W. Hill. {\it Researches in the lunar theory.} Amer. J. Math. 1, (1878) pp. 5--27, 129--148, 245--251.
\bibitem{HP} A. Harris and G. P. Paternain. \textit{Dynamically convex Finsler metrics and J-holomorphic embedding of asymptotic cylinders.} Annals of Global Analysis and Geometry {\bf 34} (2008), no.2, 115--134.
\bibitem{93} H. Hofer. \textit{Pseudoholomorphic curves in symplectizations with applications to the Weinstein conjecture in dimension three.} Invent. Math. {\bf 114} (1993), 515--563.
\bibitem{props1} H. Hofer, K. Wysocki and E. Zehnder. \textit{Properties of pseudoholomorphic curves in symplectisations I: Asymptotics.} Ann. Inst. H. Poincar\'e Anal. Non Lin\'eaire {\bf 13} (1996), 337--379.
\bibitem{props2} H. Hofer, K. Wysocki and E. Zehnder. \textit{Properties of pseudoholomorphic curves in symplectisations II: Embedding controls and algebraic invariants.} Geom. Funct. Anal. {\bf 5} (1995), no. 2, 270--328.
\bibitem{props3}H. Hofer, K. Wysocki and E. Zehnder. \textit{Properties of pseudoholomorphic curves in symplectizations III: Fredholm theory.} Topics in nonlinear analysis, Birkh\"auser, Basel, (1999), 381-475.
\bibitem{convex} H. Hofer, K. Wysocki and E. Zehnder. \textit{The dynamics of three-dimensional strictly convex energy surfaces.} Ann. of Math. {\bf 148} (1998), 197--289.
\bibitem{char1}H. Hofer,  K. Wysocki and E. Zender. \textit{A characterisation of the tight three-sphere.} Duke J. Math {\bf 81} (1996), no.1, 159--226.
\bibitem{char2}H. Hofer, K. Wysocki and E. Zehnder. \textit{A characterization of the tight three sphere II.} Commun. Pure Appl. Math. {\bf 55} (1999), no. 9, 1139--1177.
\bibitem{fols} H. Hofer, K. Wysocki and E. Zehnder, {\it Finite energy foliations of tight three-spheres and Hamiltonian dynamics.} Ann. Math. 157 (2003), 125--255.
\bibitem{hryn} U. L. Hryniewicz. \textit{Fast finite-energy planes in symplectizations and applications.} Trans. Amer. Math. Soc. {\bf 364} (2012), 1859--1931.
\bibitem{openbook} U. L. Hryniewicz. {\it Systems of global surfaces of section for dynamically convex Reeb flows on the $3$-sphere.} J. Symplectic Geom. {\bf 12} (2014), 791--862.
\bibitem{HLS} U. L. Hryniewicz, J. E. Licata, P. A. S. Salom\~ao. {\it A dynamical characterization of universally tight lens spaces.} Proc. Lond. Math. Soc. (2014)
doi: 10.1112/plms/pdu043.
\bibitem{HMS} U. Hryniewicz, A. Momin and Pedro A. S. Salom\~ao, {\it A Poincar\'e-Birkhoff theorem for tight Reeb flows on $S^3$}, Invent. Math. (2015) 199:333--422.
\bibitem{HS} U. Hryniewicz  and Pedro A. S. Salom\~ao. \textit{On the existence of disk-like global sections for Reeb flows on the tight $3 $-sphere.} Duke Math. J. {\bf 160}, (2011), no.3  415--465.
\bibitem{global} U. Hryniewicz and Pedro A. S. Salom\~ao, \emph{Global properties of tight Reeb flows with applications to Finsler geodesic flows on $S^2$}, Math. Proc. Cam. Phil. Soc. {\bf 154} (2013), 1--27.
\bibitem{HLW} X. Hu, W. Wang, Y. Long. {\it Resonance identity, stability, and multiplicity of closed characteristics on compact convex hypersurfaces.} Duke Math. J. 139 (2007), no. 3, 411--462.
\bibitem{LZ} Y. Long and C. Zhu. {\it Closed characteristics on compact convex hypersurfaces in $\R^{2n}$.} Ann. of Math. (2) 155 (2002), no. 2, 317--368.
\bibitem{Kang} J. Kang, {\it On reversible maps and symmetric periodic points.} arXiv:1410.3997.
\bibitem{kat} A. B. Katok. \textit{Ergodic perturbations of degenerate integrable Hamiltonian systems.} Mathematics of the USSR-Izvestiya {\bf 7} (1973), no.3, 535.
\bibitem{kummer} M. Kummer. {\it On the stability of Hill's solutions of the plane restricted three body problem.} Am. J. Math. 101(6), 1333--1354 (1979).
\bibitem{mcdsal} D. McDuff  and D. Salamon. \textit{$J$-holomorphic curves and symplectic topology.} Amer. Math. Soc. Colloq. Publ. {\bf 52} (2004).
\bibitem{mcduff}D. Mc Duff. \textit{The local behaviour of J-holomorphic curves in almost-complex 4-manifolds.} J. Differential Geom. {\bf 34} (1991), 143-164.
\bibitem{McG} R. McGehee, \textit{Some homoclinic orbits for the restricted three-body problem}, University of Wisconsin--Madison, 1969.
\bibitem{rad} H-B. Rademacher, \textit{Existence of closed geodesics on positively curved Finsler manifolds.} Ergodic Theory and Dynamical Systems. {\bf 27} (2007), no.3, 957--969.
\bibitem{rad2} H-B. Rademacher, \textit{The length of a shortest geodesic loop.} Comptes Rendus Mathematique. {\bf 346}  (2008), no.13, 763--765.
\bibitem{sie1}R. Siefring. \textit{Relative asymptotic behavior of pseudoholomorphic half-cylinders.} Comm. Pure Appl. Math., {\bf 61}(12):1631--1684, 2008.
\bibitem{sie2} R. Siefring. \textit{Intersection theory of punctured pseudoholomorphic curves.}  Geometry \& Topology, {\bf 15} (2011),  2351--2457.
\bibitem{wendl} C. Wendl. \textit{Automatic transversality and orbifolds of punctured holomorphic curves in dimension four.} Comment. Math. Helv. {\bf 85} (2010), 347--407.
\end{thebibliography}
\end{document}